\documentclass[11pt,a4paper,reqno]{amsart}

\usepackage[utf8]{inputenc}
\usepackage[T1]{fontenc}

\usepackage[left=1in,right=1in]{geometry}

\usepackage{amsmath}
\usepackage{amssymb}
\usepackage{amsfonts}
\usepackage{graphicx}
\usepackage{amsthm}
\usepackage{physics}

\usepackage[colorlinks=true,linkcolor=cyan,citecolor=magenta]{hyperref}

\usepackage{tikz}
\usetikzlibrary{hobby}

\usepackage[type1]{garamondlibre}
\usepackage{mathrsfs}
\usepackage{sobolev}
\DefaultSet{\Gs}
\usepackage{esint}
\usepackage{bbm}
\usepackage{mathtools}
\usepackage{cases}
\usepackage{enumerate}
\usepackage{stmaryrd}
\SetSymbolFont{stmry}{bold}{U}{stmry}{m}{n}

\theoremstyle{plain}
\newtheorem{thm}{Theorem}[section]
\newtheorem{prop}[thm]{Proposition}
\newtheorem{cor}[thm]{Corollary}
\newtheorem{lem}[thm]{Lemma}

\theoremstyle{definition}
\newtheorem{defi}[thm]{Definition}

\theoremstyle{remark}
\newtheorem*{rk}{Remark}

\newcommand\mrm[1]{\ensuremath{\mathrm{#1}}}

\newcommand{\wt}[1]{\widetilde{#1}}

\newcommand{\R}{\mathbb{R}}

\newcommand{\OO}{\varnothing}

\newcommand*{\I}{\mathbbm{1}}
\newcommand\pd{\partial}

\DeclareMathOperator{\dist}{\mathrm{dist}}
\DeclareMathOperator{\Div}{\mathrm{div}}
\newcommand*\lap{\mathop{}\!\mathbin{\triangle}}
\newcommand*\DD{\mathop{}\!\mathbin{\mathrm{D}}}

\providecommand\given{} 
\newcommand\SetSymbol[1][]{\nonscript\:#1\vert \allowbreak \nonscript\: \mathopen{}}
\DeclarePairedDelimiterX\Set[1]\{\}{ \renewcommand\given{\SetSymbol[\delimsize]} #1 }

\newcommand{\ini}[1]{\qty(#1)_{\mathrm{I}}}
\newcommand{\itn}[1]{{#1}^{(n)}}
\newcommand{\itm}[1]{{#1}^{(n+1)}}

\newcommand{\vv}{{\vb{v}}}
\newcommand{\vh}{{\vb{h}}}
\newcommand{\vbu}{{\vb{u}}}
\newcommand{\vw}{{\vb{w}}}
\newcommand{\vnu}{{\vb*{\nu}}}
\newcommand{\vmu}{{\vb*{\mu}}}

\newcommand{\vf}{{\vb{f}}}
\newcommand{\vg}{{\vb{g}}}
\newcommand{\vJ}{{\vb{J}}}
\newcommand{\vj}{{\vb{j}}}
\newcommand{\veta}{{\vb*{\eta}}}
\newcommand{\vxi}{{\vb*{\xi}}}
\newcommand{\vV}{{\vb{V}}}
\newcommand{\vH}{{\vb{H}}}
\newcommand{\vW}{{\vb{W}}}

\newcommand{\Om}{{\Omega}}
\newcommand{\Gt}{{\Gamma_t}}
\newcommand{\vn}{{\vb{N}}}
\newcommand{\ii}{\ensuremath{I\!\!I}}
\newcommand{\II}{{\vb{I\!I}}}
\newcommand{\Gs}{{\Gamma_\ast}}
\newcommand{\kk}{{\frac{3}{2}k}}
\newcommand{\vom}{{\vb*{\omega}}}
\newcommand{\gt}{{\gamma_\Gt}}
\newcommand{\OGt}{{\Omega\setminus\Gt}}
\newcommand{\OGs}{{\Omega\setminus\Gs}}

\newcommand{\X}{\mathcal{X}}
\newcommand{\id}{\mathrm{Id}}
\newcommand{\ls}{\Lambda_*}
\newcommand{\gmgm}{{\gamma_\Gamma}}

\newcommand{\K}{{\mathfrak{K}}}
\newcommand{\ka}{{\varkappa_a}}
\newcommand{\F}{{\vb{F}}}
\newcommand{\LL}{{\mathscr{L}}}
\newcommand{\f}{{\mathfrak{f}}}
\newcommand{\g}{{\mathfrak{g}}}
\newcommand{\baf}{\bar{\f}}
\newcommand{\E}{\mathfrak{E}}

\newcommand{\h}{{\mathcal{H}}}
\newcommand{\n}{{\mathcal{N}}}
\newcommand{\bn}{\bar{\n}}
\newcommand{\wn}{\wt{\n}}
\newcommand{\wnh}{{\wn^{\frac{1}{2}}}}
\newcommand{\opa}{{\mathcal{A}}}
\newcommand{\opr}{{\mathcal{R}}}
\newcommand{\opA}{{\mathscr{A}}}
\newcommand{\opR}{{\mathscr{R}}}
\newcommand{\opC}{{\mathscr{C}}}
\newcommand{\opB}{{\mathscr{B}}}
\newcommand{\Pb}{{\mathbb{P}}}
\newcommand{\opb}{{\mathbf{B}}}
\newcommand{\opf}{{\mathbf{F}}}
\newcommand{\opg}{{\mathbf{G}}}
\newcommand{\opF}{{\mathscr{F}}}
\newcommand{\opG}{{\mathscr{G}}}
\newcommand{\opd}{{\mathcal{D}}}

\newcommand{\tdot}{{\textbf{:}}}
\newcommand{\Dt}{{\mathbb{D}_t}}
\newcommand{\Dtl}{{\mathbb{D}_{t_\lambda}}}
\newcommand{\Dtb}{{\mathbb{D}_{\bar{t}}}}
\newcommand{\Dts}{{\mathbb{D}_{t\ast}}}
\newcommand{\Dbt}{{\mathbb{D}_{\beta}}}

\newcommand{\less}{\lesssim_{\Lambda_*}}
\newcommand{\lesq}{\lesssim_{Q}}

\numberwithin{equation}{section}

\begin{document}
	
	\title[Current-Vortex Sheet Problems]{Local Well-posedness of the Incompressible Current-Vortex Sheet Problems}
	
	\date{\today}
	
	\author{Sicheng Liu}
	\address{The Institute of Mathematical Sciences, the Chinese University of Hong Kong, Shatin, NT, Hong Kong SAR, China}
	\email[Corresponding author]{scliu@link.cuhk.edu.hk}
	\thanks{This is part of the Ph.D. thesis of the first author written under the guidance of the second author at the Institute of Mathematical Sciences, the Chinese University of Hong Kong. This research is supported in part by Zheng Ge Ru Foundation, Hong Kong RGC Earmarked Research Grants CUHK-14301421, CUHK-14300819, CUHK-14302819, CUHK-14300917, and the key project of NSFC (Grant No. 12131010).}
	
	\author{Zhouping Xin}
	\address{The Institute of Mathematical Sciences, the Chinese University of Hong Kong, Shatin, NT, Hong Kong SAR, China}
	\email{zpxin@ims.cuhk.edu.hk}
	
	
	
	\begin{abstract}
		We prove the local well-posedness of the incompressible current-vortex sheet problems in standard Sobolev spaces under the surface tension or the Syrovatskij condition, which shows that both capillary forces and large tangential magnetic fields can stabilize the motion of current-vortex sheets. Furthermore, under the Syrovatskij condition, the vanishing surface tension limit is established for the motion of current-vortex sheets. These results hold without assuming the interface separating the two plasmas being a graph.
	\end{abstract}
	
	\maketitle
	
	\tableofcontents
	
	\section{Introduction}
	\subsection{Formulations of the problems}
	We consider the free interface problems for ideal incompressible magnetohydrodynamics (MHD) equations, which describe the motions of two plasmas separating by a free interface (current-vortex sheet problems). If we denote by $ \Om_t^\pm \subset \R^3 $ the fluid domains at time $ t $ occupied by two kinds of plasmas respectively, the ideal incompressible MHD system can be written as
	\begin{subnumcases}
		{(\mathrm{MHD})\quad}
		\pd_t \vv_\pm + \qty(\vv_\pm \cdot \grad) \vv_\pm + \dfrac{1}{\rho_\pm} \grad p^\pm = \qty(\vh_\pm \cdot \grad) \vh_\pm &in $ \Omega_t^\pm $, \label{eqn Dtv} \\
		\pd_t \vh_\pm + (\vv_\pm \cdot \grad) \vh_\pm = (\vh_\pm \cdot \grad) \vv_\pm &in $ \Omega_t^\pm $, \label{eqn Dth} \\
		\div \vv_\pm = 0 = \div \vh_\pm &in $ \Omega_t^\pm $; \label{divfree}
	\end{subnumcases}
	here $\rho_\pm, \vv_\pm, \vh_\pm, p^\pm$ are the densities, velocities, magnetic fields and effective pressures for the two plasmas respectively (c.f. \cite{Landau-Lifshitz-Vol8} or \cite{book_Davidson}).
	The boundary conditions are:
	\begin{subnumcases}
		{(\mathrm{BC}) \quad}
		\vv_+ \cdot \vn_+ = \vv_- \cdot \vn_+ =: \theta &on $ \Gt $, \label{eqn bdry v} \\
		\llbracket p \rrbracket \coloneqq p^+ - p^- = \alpha^2 \kappa_+ &on $ \Gt $, \label{jump p} \\
		\vh_+ \cdot \vn_+ = \vh_- \cdot \vn_+ = 0 &on $ \Gt $,  \label{bdry mag}\\
		\vv_- \cdot \wt{\vn} = 0 &on $ \pd \Omega $, \label{bc out v}\\
		\vh_- \cdot \wt{\vn} = 0 &on $ \pd \Om $ \label{bc out h},
	\end{subnumcases}
	where $ \kappa_+ $ is the mean curvature of $ \Gt $ with respect to $ \vn_+ $, and $ 0 \le \alpha \le 1$ is a non-negative constant representing the surface tension coefficient.
	
	Here $ \Omega \subset \R^3 $ is a bounded domain with a fixed boundary $ \pd \Om $, and $ \Om = \Om_t^+ \cup \Gt \cup \Om_t^- $, $ \Gt = \pd \Om_t^+ $ is the moving interface with normal speed $ \theta $, and $ \pd\Om_t^- = \pd \Om \cup \Gt $. Denote by $ \vn_+ $ the outward unit normal of $ \pd\Om_t^+ = \Gt $, and $ \wt{\vn} $ the outward unit normal of $ \pd\Om $. 
	Assume further that $ \Gt \subset \Om $, $ \Gt \cap \pd \Om = \OO $, and $ \Gt $ separates $ \Om $ into two disjoint simply-connected domains $ \Om_t^\pm $. 	
	
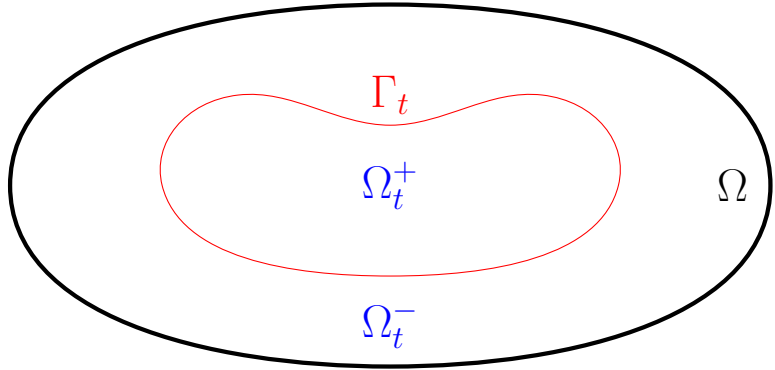
	\begin{center}
		\begin{tikzpicture}[use Hobby shortcut, yscale=0.8]
			\path
			(5,0) coordinate (a0)
			(0,3) coordinate (a1)
			(-5,0) coordinate (a2)
			(0,-3) coordinate (a3);
			\draw[closed,ultra thick] (a0) .. (a1) .. (a2) .. (a3);
			\node at (4.5,0) {\LARGE $ \Omega $};
			
			\path
			(0,1) coordinate (b0)
			(2,1.5) coordinate (b1)
			(3,0) coordinate (b2)
			(0,-1.5) coordinate (b3)
			(-3,0) coordinate (b4)
			(-2,1.5) coordinate (b5);
			\draw[red,closed] (b0) .. (b1) .. (b2) .. (b3) .. (b4) .. (b5);
			\node[red] at (0, 1.5) {\LARGE$ \Gamma_t $};
			\node[blue] at (0,0) {\LARGE$ \Omega_t^+ $};
			\node[blue] at (0,-2.3) {\LARGE $ \Omega_t^- $};
		\end{tikzpicture}
	\end{center}
	
	The equations \eqref{eqn Dtv} are the Euler equations in hydrodynamics, for which the Lorentz forces serve as the exterior body forces acting on the plasmas. Note that the displacement currents are neglected, due to the fact that the scale of the plasma velocities is much less than the speed of light. The equations \eqref{eqn Dth} are the combination of Faraday's Law and Ohm's Law, and \eqref{divfree} are the incompressibility of the plasmas and Gauss's Law for magnetism. The boundary condition \eqref{eqn bdry v} is also known as the kinematic boundary condition, which means that the free interface evolves with the two plasmas. \eqref{jump p} is derived from the balance of momentum between two sides of the interface, and \eqref{bdry mag} follows from the Gauss's Law for magnetism and physical characters of the materials. \eqref{bc out v} means that the outer plasma cannot penetrate the solid boundary, and \eqref{bc out h} follows from the assumption that the solid boundary is a perfect conductor.
	
	\subsection{Physical background}
	The motion of electrically conductive fluids (e.g., plasma, liquid metals, salt water, and electrolytes) under the influence of magnetic fields is governed by the MHD systems. The corresponding mathematical theories have numerous significant applications (e.g., drug targeting, earthquakes, sensors, and astrophysics). One of the fundamental differences between MHD and hydrodynamics is that the magnetic fields can induce currents in a moving conductive fluid, and these currents in turn polarize the fluid and change the magnetic and velocity fields in a reciprocal manner. The set of equations is a combination of those in fluid dynamics and electrodynamics, and these equations must be solved concurrently (c.f. \cite{Landau-Lifshitz-Vol8,book_Davidson}). Mathematically speaking, the effect of the magnetic field is governed by the Maxwell equations and acts as a Lorentz force on the Euler system for the plasma, which can induce many nontrivial interactions and lead to rich phenomena.
	
	The current-vortex sheet problems describe the plasma motion in a domain whose boundary evolves with the plasma itself. Such issues are significant not only because they describe numerous physical phenomena thus have significant applications in science and technology, but also since such studies give rise to profound and challenging theoretical interdisciplinary problems involving partial differential equations, differential geometry, analysis, mathematical physics, and dynamical systems.
	
	\subsection{Previous works}
	In the absence of magnetic fields, the equations are reduced to the incompressible Euler system. The free boundary problems in hydrodynamics have garnered considerable interest from the mathematical community. Although water waves are very universal in reality, from which one can see a vast diversity of phenomena, the corresponding mathematical theories are still in their infancy, because the full equations describing the motion of the waves are famously difficult to handle due to the free boundary and intrinsic nonlinearity. We refer to the works by Wu \cite{Wu1997,Wu1999}, Alazard-Burq-Zuily \cite{Alazard-Burq-Zuily2014} for the local well-posedness of irrotational water wave problems. When the vorticity of the fluid flow is non-zero, one can refer to Christodoulou-Lindblad \cite{Christodoulou-Lindblad2000}, Lindblad \cite{Lindblad2003,Lindblad2005}, Coutand-Shkoller \cite{Coutand-Shkoller2007}, Cheng-Coutand-Shkoller \cite{Cheng-Coutand-Shkoller2008}, Zhang-Zhang \cite{Zhang-Zhang2008}, Shatah-Zeng \cite{Shatah-Zeng2008-Geo,Shatah-Zeng2008-vortex,Shatah-Zeng2011} for the local well-posedness of the water wave and vortex sheet problems.
	
	In contrast to the long history of the study on the water wave problems, the free-interface problems for ideal MHD equations have been studied only in recent decades. Owing to the strong coupling of the magnetic and velocity fields, it is necessary to deal with multiple hyperbolic systems simultaneously, making it difficult to establish the nonlinear well-posedness theories. In particular, how magnetic fields affect the dynamics of a plasma is an important issue. As most of fluids are electrically conductive and magnetic fields are ubiquitous, the MHD model is certainly an important physical one with similar significance as the Euler or Navier-Stokes ones. When the effect of magnetic fields is not negligible, it is significant to study the dynamics of conducting fluids. Here are some representative works on the free interface problems for the ideal incompressible MHD.
	
	A current-vortex sheet is a hypersurface evolving with the conductive fluids, along which the magnetic and velocity fields possess tangential jumps. This sort of problems explain the motion of two conducting fluids with a free interface separating them. Around the middle of the twentieth century, Syrovatskij \cite{Syrovatskij} and Axford \cite{axford1962note} discovered the stability requirements for the planer incompressible current-vortex sheets and demonstrated that magnetic fields have a stabilizing influence on the plasma dynamics. The Syrovatskij stability conditions are (see Landau-Lifshitz \cite[\textsection~71]{Landau-Lifshitz-Vol8}):
	\begin{equation}\label{Syr 1"}
		\rho_+ \abs{\vh_+}^2 + \rho_- \abs{\vh_-}^2 > \frac{\rho_+ \rho_-}{\rho_+ + \rho_-}\abs{\llbracket\vv\rrbracket}^2,
	\end{equation}
	\begin{equation}\label{Syr 2"}
		\qty(\rho_+ + \rho_-)\abs{\vh_+ \cp \vh_-}^2 \ge \rho_+ \abs{\vh_+ \cp \llbracket\vv\rrbracket}^2 + \rho_-\abs{\vh_-\cp\llbracket\vv\rrbracket}^2,
	\end{equation}
	where $ \llbracket\vv\rrbracket \coloneqq \vv_+ - \vv_- $ is the velocity jump.  If the current-vortex sheet is assumed to be the graph of a function, there are some studies on the dynamics:
	Trakhinin \cite{Trakinin2005} proved the a priori estimate for the linearized equations under a strong stability condition:
	\begin{equation}\label{strong stabilibty"}
		\abs{\vh_+ \cp \vh_-} > \max\qty{\abs{\vh_+ \cp \llbracket\vv\rrbracket}, \abs{\vh_-\cp\llbracket\vv\rrbracket}}.
	\end{equation}
	Coulombel-Morando-Secchi-Trebeschi \cite{Coulombel-Morando-Secchi-Trebischi2012} showed the a priori estimate without loss of derivatives for the non-linear problem under \eqref{strong stabilibty"}. If the original Syrovatskij condition \eqref{Syr 2"} were replaced by the following strict one: 
	\begin{equation}\label{Syro 3"}
		\qty(\rho_+ + \rho_-)\abs{\vh_+ \cp \vh_-}^2 > \rho_+ \abs{\vh_+ \cp \llbracket\vv\rrbracket}^2 + \rho_-\abs{\vh_-\cp\llbracket\vv\rrbracket}^2,
	\end{equation}
	from which \eqref{Syr 1"} follows, Morando-Trakhinin-Trebeschi \cite{Morando-Trakhinin-Trebeschi2008} derived the a priori estimates for the linearized system with loss of derivatives. The nonlinear local well-posedness result under \eqref{Syro 3"} was first proven by Sun-Wang-Zhang \cite{Sun-Wang-Zhang2018}. The above works demonstrate that the strict Syrovatskij condition \eqref{Syro 3"} indeed has a nonlinear stabilizing effect on the free interface (at least for a graph surface in a short time period), in contrast to the Kelvin-Helmholtz instability for pure-fluid vortex-sheet issues due to the lack of surface tension (c.f. \cite{Ebin1988} and \cite[Chapters~9 \& 11]{Majda-Bertozzi2002} for more detailed discussions). Recently, the methods in \cite{Sun-Wang-Zhang2018} were also applied to the study for the case with surface tension, see \cite{Li-Li2022}. 
	
	For the plasma-vacuum interface problems, if the magnetic field is parallel to the free boundary and the one in the vacuum is vanishing, we refer to Hao-Luo \cite{Hao-Luo2014,Hao-Luo2021}, Gu-Wang \cite{Gu-Wang2019}, and Gu-Luo-Zhang \cite{Gu-Luo-Zhang2021,Gu-Luo-Zhang22} for the local well-posedness. If the magnetic field in the vacuum is nontrivial, one can see Mordando-Trakhinin-Trebeschi \cite{Morando-Trakhinin-Trebeschi2014}, and Sun-Wang-Zhang \cite{Sun-Wang-Zhang2019} for the local well-posedness under a stability condition. Hao-Luo \cite{Hao-Luo2020} also showed the ill-posedness for the plasma-vacuum problems without the Rayleigh-Taylor sign condition, as indicated by Ebin \cite{Ebin1987} for the pure fluid-vacuum case. Concerning the global well-posedness for free-boundary incompressible inviscid MHD equations, Wang-Xin \cite{Wang-Xin2021} established it for both the plasma-vacuum and the plasma-plasma problems. 
	
	Although these advances are significant, all of the aforementioned nonlinear local well-posedness results for MHD problems were founded on a crucial premise that the free interface is a graph. However, in reality, the moving surface cannot be represented simply by a graph in many significant cases. To remove these limitations seems quite challenging, even for the pure fluid problems (c.f. \cite{Coutand-Shkoller2007, Shatah-Zeng2011}). Using the partition of unity to characterize the general interface appears feasible, but the analysis of these transition maps is rather involved due to the intense interactions between the plasmas in different local charts. In view of the strong coupling of the magnetic and velocity fields (one direct consequence of which is that the vorticity transport formula will change), MHD problems must be analyzed with greater care than the pure fluid ones. For example, one of the difficulties is that the estimates of the velocity and magnetic fields must be derived simultaneously, which is much more complex than in the case for pure fluids. More significantly, the strategies on the local dynamic motion of a general current-vortex sheet will be indispensable to the study of long-time dynamics, particularly the finite-time formation of splash singularities from a generic perturbation of a current-vortex sheet (even of a graph type).
	
	This paper is to establish the nonlinear local well-posedness for the current-vortex sheet problems in standard Sobolev spaces, without the graph assumption on the free interface. Namely, we show that for more general physical models, both the capillary forces and large tangential magnetic fields (the Syrovatskij condition) can stabilize the motion of current-vortex sheets. In particular, our results can be applied to study the dynamics of free interfaces with turning-over points, and may be useful to construct splash singularities. 
	\section{Main Results}
	
	For convenience, we shall use the notation $ f \coloneqq f_+ \I_{\Om_t^+} + f_- \I_{\Om_t^-} $ to represent for functions $ f_\pm : \Om_t^\pm \to \R $.
	
	\subsection{The stabilization effect of the surface tension}
	
	If there exists surface tension on the free interface, the following local well-posedness result holds:
	\begin{thm}[$ \alpha = 1$ case]\label{thm s.t.}
		Suppose that $ \Om \subset \R^3 $ is a bounded domain with a $ C^1 \cap H^{\kk + 1}$ boundary, and $ k \ge 2 $ is an integer.
		Given the initial hypersurface $ \Gamma_0 \in H^{\kk+1} $ and two solenoidal vector fields $ \vv_0, \vh_0 \in H^{\kk}(\Om\setminus\Gamma_0) $, if $ \Gamma_0 $ separates $ \Om $ into two disjoint simply-connected parts, then there exists a constant $ T > 0 $ so that the current-vortex sheet problem (MHD)-(BC) has a solution in the space:
		\begin{equation*}
			\Gt \in C^0\qty([0, T]; H^{\kk+1}) \qand \vv, \vh \in C^0\qty([0, T]; H^{\kk}(\OGt)).
		\end{equation*}
		Furthermore, if $ k \ge 3 $, the solution is unique and it depends on the initial data continuously, i.e. the problem (MHD)-(BC) is locally well-posed.
	\end{thm}
	
	\subsection{The stabilization effect of the Syrovatskij condition}
	In the absence of surface tension, we show that the Syrovatskij condition \eqref{Syro 3"} can stabilize the motion of the current-vortex sheet, at least for a short time period, without the graph assumption on the interface.
	
	Due to the hairy ball theorem, \eqref{Syro 3"} cannot hold on a hypersurface homeomorphic to a sphere. Thereby, we assume that $ \Om = \mathbb{T}^2 \times (-1, 1) $, and $ \Gamma_t $ is a $ C^1 \cap H^2 $ hypersurface diffeomorphic to $ \mathbb{T}^2 $ (e.g. a surface with shape "$ \mho $" or "$ Z $", or a portion of sea waves), which separates $ \Omega $ into the upper and the lower parts. 
	
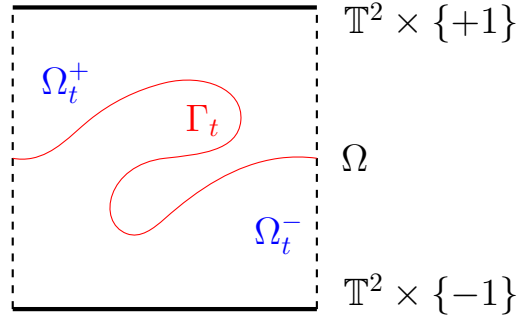
\begin{center}
		\begin{tikzpicture}
			\path
			(0,0) coordinate (a0)
			(4,0) coordinate (a1)
			(4,4) coordinate (a2)
			(0,4) coordinate (a3);
			\draw[dashed, thick] (a0) -- (a3);
			\draw[dashed, thick] (a1) -- (a2);
			\draw[ultra thick] (a0) -- (a1);
			\draw[ultra thick] (a2) -- (a3);
			
			\path;
			\draw[red,use Hobby shortcut] (0,2) .. (0.25,2) .. (1,2.5) .. (2,3) .. (3,2.5) .. (2,2) .. (1.5,1) .. (2,1.2) .. (3.5,2) .. (4,2);
			
			\node[blue] at (3.5,1) {\Large $ \Omega^-_t $};
			\node[blue] at (0.7,3) {\Large $ \Omega^+_t $};
			\node[red] at (2.5, 2.5) {\Large $ \Gamma_t $};
			\node at (4.5,2) {\Large $ \Omega $};
			\node at (5.5,0.2) {\Large $ \mathbb{T}^2 \times \{-1\} $};
			\node at (5.5,3.8) {\Large $ \mathbb{T}^2 \times \{+1\} $};
		\end{tikzpicture}
	\end{center}
	Accordingly, the boundary conditions \eqref{eqn bdry v}-\eqref{bc out h} are modified to
	\begin{equation}
		(\mathrm{BC'})\quad	\begin{cases*}
			\vv_+ \cdot \vn_+ = \vv_- \cdot \vn_+ =: \theta &on $ \Gt $,\\
			\llbracket p \rrbracket \coloneqq p^+ - p^- = \alpha^2\kappa_+ &on $ \Gt $, \\
			\vh_+ \cdot \vn_+ = \vh_- \cdot \vn_+ = 0 &on $ \Gt $,\\
			\vv_\pm \cdot \wt{\vn}_\pm = 0 &on $ \mathbb{T}^2 \times\{\pm1\} $,\\
			\vh_\pm \cdot \wt{\vn}_\pm = 0 &on $ \mathbb{T}^2 \times\{\pm1\} $,
		\end{cases*}
	\end{equation}
	where $ \wt{\vn}_\pm \equiv \pm \vb{e}_3 $, are the outward unit normals of $ \mathbb{T}^2 \times \{\pm1\} $.
	
	By Lemma \ref{lem syro equiv 1}, the strict Syrovatskij condition \eqref{Syro 3"} implies
	\begin{equation*}
		\begin{split}
			0 < \Upsilon\qty(\vh_{\pm}, \llbracket\vv\rrbracket) &\coloneqq \inf_{\substack{\vb{a} \in \mathrm{T}\Gamma_t; \\ \abs{\vb{a}}=1}} \inf_{z \in \Gamma_t} \frac{\rho_+}{\rho_+ + \rho_-}\abs{\vb{a} \vdot \vh_+(z)}^2 + \frac{\rho_-}{\rho_+ + \rho_-}\abs{\vb{a} \vdot \vh_-(z)}^2 \\
			&\hspace{6em}  - \frac{\rho_+ \rho_-}{(\rho_+ + \rho_-)^2}\abs{\vb{a} \vdot \llbracket\vv\rrbracket(z)}^2.
		\end{split}
	\end{equation*}
	
	We prove the following two theorems under the Syrovatskij condition \eqref{Syro 3"}:
	
	\begin{thm}[$ \alpha = 0 $ case]\label{thm alpha=0 case}
		Let $ k \ge 3 $ be an integer and $ \Om \coloneqq \mathbb{T}^2 \times (-1, 1)$. Suppose that $ \Gamma_0 $ is an $ H^{\kk+\frac{1}{2}} $ hypersurface diffeomorphic to $ \mathbb{T}^2 $ separating $ \Om $ into two parts (the upper one and the lower one). Assume that $ \vv_0, \vh_0 \in H^{\kk}(\Om\setminus\Gamma_0) $ are two $ H^{\kk} $ solenoidal vector fields satisfying
		\begin{equation*}
			\begin{split}
				\Upsilon\qty(\vh_{0\pm}, \llbracket\vv_0\rrbracket) \ge 2\mathfrak{s}_0 > 0.
			\end{split}
		\end{equation*}
		If $ \Gamma_0 $ does not touch the top or the bottom boundary, namely, for some positive constant $ c_0 $,
		\begin{equation*}
			\dist\qty(\Gamma_0, \mathbb{T}^2 \times \{\pm1\}) \ge 2c_0 > 0;
		\end{equation*}
		then there exists a constant $ T > 0 $, so that the current-vortex sheet problem (MHD)-(BC') admits a unique solution in the space
		\begin{equation*}
			\Gt \in C^0\qty([0, T]; H^{\kk+\frac{1}{2}}) \qand \vv, \vh \in C^0\qty([0, T]; H^{\kk}(\OGt)).
		\end{equation*}
		Furthermore, for $0 \le t \le T$, the solution $ (\Gt, \vv, \vh) $ satisfies
		\begin{equation*}
			\Upsilon\qty(\vh_\pm, \llbracket\vv\rrbracket) \ge \mathfrak{s}_0 \qand \dist\qty(\Gt, \mathbb{T}^2 \times \{\pm1\}) \ge c_0,
		\end{equation*}
		and it depends on the initial data continuously. That is, the problem (MHD)-(BC') is locally well-posed.
	\end{thm}
	
	Furthermore, the following result on the vanishing surface tension limit holds:
	\begin{thm}[$ \alpha \to 0 $ limit]\label{thm vanishing surf limt}
		Assume that $ 0 \le \alpha \le 1,  k \ge 3 $, and $ \Om = \mathbb{T}^2 \times (-1, 1) $. Take initial data $ \Gamma_0 \in H^{\kk+1} $ diffeomorphic to $ \mathbb{T}^2 $ with $ \dist(\Gamma_0, \mathbb{T}^2\times\{\pm1\}) \ge 2 c_0 > 0 $ and solenoidal vector fields $ \vv_{0\pm}, \vh_{0\pm} \in H^{\kk}(\Om_0^\pm)$ satisfying $ \Upsilon\qty(\vh_{0\pm}, \llbracket\vv_0\rrbracket) \ge 2\mathfrak{s}_0 > 0 $. Then, there is a constant $ T > 0 $, independent of $ \alpha $, so that the problem, (MHD)-(BC'), is well-posed for $ t \in [0, T] $. Furthermore, as $ \alpha \to 0 $, by passing to a subsequence, the solution to (MHD)-(BC') with surface tension converges weakly to a solution to (MHD)-(BC') with $ \alpha = 0 $ in the space $ \Gt \in H^{\kk+\frac{1}{2}} $ and $ \vv_\pm, \vh_\pm \in H^{\kk}(\Om_t^\pm) $. 
	\end{thm}
	
	\subsection{Main ideas}
	Inspired by the works of Shatah-Zeng \cite{Shatah-Zeng2011}, we choose a geometric approach to analyze the problems. First of all, a reference hypersurface $ \Gs $ diffeomorphic to the initial one is taken, and one may choose a transversal vector field  $ \vnu $ defined on the reference hypersurface of the same regularity and close to the unit normal in the $ C^1 $-topology. Therefore, any hypersurface near the reference one can be expressed uniquely by the height function defined on $ \Gs $ via:
	\begin{equation*}
		\Phi_\Gamma (p) = p + \gamma_\Gamma (p) \vnu(p) \qq{for} p \in \Gs.
	\end{equation*}
	Every hypersurface $ \Gamma $ in a small $ C^1 $-neighborhood of $ \Gs $ is associated to a unique function $ \gamma_\Gamma $, and the mean curvature $ \kappa $ of $ \Gamma $ can be expressed by $ \gamma_\Gamma $. Conversely, by taking an auxiliary constant $ a > 0 $ determined by $ \Gs $ and $ \vnu $, the height function $ \gamma_\Gamma $ can be determined uniquely by the function $ \kappa \circ \Phi_\Gamma + a^2 \gamma_\Gamma : \Gs \to \R $, whose leading order term is the mean curvature $ \kappa $ of $ \Gamma $. Hence, the analysis of the evolution equation for the mean curvature $ \kappa $ can determine the evolution of the free hypersurface, which is the crucial part of the settlement of such free interface problems. 
	
	On the other hand, any vector field defined in a simply-connected domain is uniquely determined by its divergence, curl, and normal projection on the boundary (c.f. \cite{Cheng-Shkoller2017}). The evolution equation for the free interface yields the normal part of the velocity on the boundary, and the normal projection of the magnetic field is zero (for the current-vortex sheet problems considered here); thus, the incompressibility of the plasma ($ \div \vv = 0 $) and Gauss's law for magnetism ($ \div \vh = 0 $) imply that one only needs to determine the current ($ \vj \coloneqq \curl \vh $) and vorticity ($ \vom \coloneqq \curl \vv $). The evolution equations for $ \vom $ and $ \vj $ are given respectively as
	\begin{equation*}
		\begin{cases*}
			\pd_t \vom + (\vv \vdot \grad)\vom - (\vh \vdot \grad)\vj = (\vom \vdot \grad)\vv - (\vj \vdot \grad) \vh, \\
			\pd_t\vj + (\vv \vdot \grad)\vj - (\vh \vdot \grad) \vom = (\vj \vdot \grad)\vv - (\vom \vdot \grad) \vh - 2\tr(\grad \vv \cp \grad \vh).
		\end{cases*}
	\end{equation*}
	Since $ \vh $ is parallel to the free hypersurface, as indicated in \cite{Sun-Wang-Zhang2018}, it follows that $ \vv \pm \vh $ are also evolution velocities of the interface and the fluid domain. Thus, one may use characteristic methods to transform the above equations into a system of ordinary differential equations, whose well-posedness is standard to obtain (see \cite{Caflisch-Klapper_Steele97}).
	
	Once the evolution of the free interfaces, currents and vorticities are known, the original problem can be resolved via working on the div-curl systems.
	
	Such an approach can not only resolve the free interface problems for general interfaces (in particular, with no graph assumptions) but also help to understand various stability conditions more clearly. Indeed, it will be seen in the evolution equations for the mean curvatures that the surface tension corresponds to a third order positive differential operator on the free interface, which serves as a stabilizer for the surface motion; the strict Syrovatskij condition \eqref{Syro 3"} corresponds to a second order positive differential operator. Thus, concerning the stabilization effect, surface tensions are stronger than the strict Syrovatskij condition. Moreover, due to the existence of the unstable term in the evolution equations for the mean curvature (which is a second order differential operator resulting in the Kelvin-Helmholtz instability for the vortex sheet problems), the small tangential magnetic fields (the Syrovatskij conditions can be understood as a largeness assumption) cannot stabilize the current-vortex sheet in the absence of surface tension.
	
	\subsection{Structure of the paper}
	In \textsection~\ref{sec prelimi}, we introduce some geometric relations and some analytical tools. \textsection~\ref{sec reform} - \textsection~\ref{sec nonlinear} are devoted to the proof of Theorem~\ref{thm s.t.}. More precisely, in \textsection~\ref{sec reform}, we rewrite the current-vortex sheet problems in a geometric manner, and derive the corresponding evolution equations. In \textsection~\ref{sec linear}, we study the uniform linear estimates for the linearized systems; and in \textsection~\ref{sec nonlinear} we consider the nonlinear problems and show the local well-posedness of the original current-vortex sheet ones. \textsection~\ref{sec syro} is devoted to the proof of Theorem~\ref{thm alpha=0 case} and \ref{thm vanishing surf limt}. In the Appendices, we prove two technical lemmas.
	
	\section{Preliminaries}\label{sec prelimi}
	\subsection{Geometry of hypersurfaces}
	
	For a family of  hypersurfaces $ \Gt \subset \R^d $ evolving with the velocity $ \vv : \Gt \to \R^d $, consider the local charts for the initial hypersurface $ \F : U \to \Gamma_0 \subset \R^d $. Assume that $ \Xi_t $ is the flow map induced by $ \vv $, then one can take a coordinate map of $ \Gt $ as $ \F(t) \coloneqq \Xi_t \circ \F : U \to \Gt $. Standard geometric arguments (c.f. \cite{Ecker2004}) give that the coordinate tangent vectors $ \pd_i \F (t, z) : U \to \R^{d}, (1 \le i \le d-1) $ form a natural basis of the tangent space $ \mrm{T}_p\Gt $ at $ p = \F(t, z) \in \Gt $ for each $ z \in U $. The submanifold metric of $ \Gt \subset \mathbb{R}^d $ is given by
	\begin{equation*}
		g_{ij} = \F_{,i} \vdot \F_{,j}
	\end{equation*}
	for $ 1 \le i, j \le d-1 $, where $ f_{,i} $ represents $ \pd_i f $ for any function $ f : U \to \R $. The inverse metric is defined to be
	\begin{equation*}
		(g^{ij}) \coloneqq (g_{ij})^{-1}, 
	\end{equation*} 
	and the area element of $ \Gt $ is 
	\begin{equation*}
		\sqrt{g} = \sqrt{\det(g_{ij})}.
	\end{equation*}
	Furthermore, there is a natural Riemannian connection on $ \Gt $, whose Christoffel symbols are given by
	\begin{equation*}
		\Gamma_{ij}^k \coloneqq g^{kl} \F_{,ij} \vdot \F_{,l}  =  \frac{1}{2}g^{kl}\qty(g_{jl,i} + g_{il,j} - g_{ij,l}),
	\end{equation*}
	where the summation convention for tensors has been used.
	For a tangent vector $ \vb{X} \in \mrm{T}\Gt $, one can write
	\begin{equation*}
		\vb{X} = X^i \F_{,i} = g^{ij}X_j\F_{,i},
	\end{equation*}
	where $ X_j \coloneqq \vb{X} \vdot \F_{,j} $. The covariant derivative of $ \vb{X} $ is defined to be
	\begin{equation*}
		\qty(\DD^\Gt_i X)^j  \coloneqq X^j_{,i} + \Gamma_{ik}^j X^k,
	\end{equation*}
	Hence, the divergence of $ \vb{X} $ on $\Gt$ is defined to be
	\begin{equation}\label{def bdry div}
		\Div_\Gt \vb{X} \coloneqq  g^{ij} \vb{X}_{,i} \vdot \vb{F}_{,j}.
	\end{equation}
	One can also extend (\ref{def bdry div}) to all vector fields defined on $ \Gt $, not necessarily being tangential. For a function $ h : \Gt \to \R $, the tangential gradient is defined by
	\begin{equation}
		\grad_{\Gt} h = g^{ij} h_{,i} \F_{,j},
	\end{equation}
	and the Laplace-Beltrami operator on $ \Gt $ is given by
	\begin{equation}
		\lap_\Gt h = \Div_\Gt \grad_\Gt h = g^{ij}\qty(h_{,ij} - \Gamma_{ij}^k h_{,k}) = \frac{1}{\sqrt{g}}\pd_i \qty(\sqrt{g} g^{ij} h_{,j}).
	\end{equation}
	For a general vector field $ \vb{Y} : \Gt \to \R^d $, the notation $ (\DD\vb{Y})^\top $  represents a (0,2)-tensor on $ \Gt $ (here $ \DD $ is the covariant derivative on $ \R^d $, and $ \top $ is the tangential projection):
	\begin{equation}
		\qty[(\DD\vb{Y})^\top]_{ij} \coloneqq \vb{Y}_{,i} \vdot \F_{,j},
	\end{equation}
	so
	\begin{equation*}
		\Div_\Gt \vb{Y} = \tr[(\DD\vb{Y})^\top].
	\end{equation*}
	
	Denote by $ \vn : \Gt \to \mathbb{S}^{d-1} $ the unit normal vector field of $ \Gt $, i.e.
	\begin{equation*}
		\vn \vdot \F_{,i} = 0
	\end{equation*}
	for $ 1 \le i \le d-1 $. Since $ \vn \vdot \vn \equiv 1 $, it is clear that
	\begin{equation*}
		\vn_{,i} \vdot \vn = 0,
	\end{equation*}
	namely, $ \vn_{,i} \in \mrm{T}\Gt $ for all $ 1 \le i \le d-1 $. The second fundamental form $ \II $ of $ \Gt $ is a (0, 2)-tensor defined by
	\begin{equation}\label{def II}
		\ii_{ij} \coloneqq \vn_{,i} \vdot \F_{,j} = - \vn \vdot \F_{,ij}.
	\end{equation}
	The mean curvature is defined to be the trace of $ \II $, i.e.
	\begin{equation}\label{eqn kp div N}
		\kappa \coloneqq \tr(\II) = \ii_{ij}g^{ij} = g^{ij} \vn_{,i} \vdot \F_{,j} = \Div_\Gt \vn.
	\end{equation}
	Here we mention several useful identities, whose calculations can be found in \cite[Appendix~A]{Ecker2004}:
	\begin{equation*}
		\lap_\Gt \ii_{ij} = \kappa_{;ij} + \kappa \ii_{i}^k \ii_{kj}  - \abs{\II}^2 \ii_{ij},
	\end{equation*}
	namely,
	\begin{equation}\label{simons' identity}
		\lap_\Gamma \II = \qty(\DD_\Gamma)^2 \kappa + (\kappa \II - \abs{\II}^2 \vb{I}) \vdot \II,
	\end{equation}
	which is called Simons' identity. Here the dot product of tensors is defined to be
	\begin{equation*}
		\vb{C} \coloneqq \vb{A} \cdot \vb{B} \qc C_{ij} \equiv A_i^k B_{kj} = A_{il}B_{kj}g^{lk}.
	\end{equation*}
	Furthermore, it follows from the Codazzi equation that
	\begin{equation}\label{lap vn}
		\lap_\Gt \vn = - \abs{\II}^2 \vn + \grad_\Gt{\kappa}.
	\end{equation}
	
	\begin{rk}
		Derivatives of functions and vector fields on hypersurfaces can also be defined in terms of the projections from $ \R^d $ onto the tangent space of $ \Gt $. In particular, for a function $ f $ and a vector field $ \vb{X} $ defined in a neighborhood of $ \Gt \subset \R^d $, the tangential gradient of $ f $ is
		\begin{equation}
			\grad_\Gt f = \grad f - (\vn \vdot \grad f) \vn,
		\end{equation}
		where $ \grad f $ is the gradient in $ \R^d $; and the tangential divergence of $ \vb{X} $ is given by
		\begin{equation}\label{eqn div Gt div Rd}
			\Div_\Gt \vb{X} = \Div_{\R^d} \vb{X} - \vn \vdot \DD_{\vn} \vb{X},
		\end{equation}
		where $ \DD $ is the covariant derivative in $ \R^d $. The above definitions are identical to the intrinsic ones given earlier.
		The Laplace-Beltrami operator can be calculated in an equivalent way by:
		\begin{equation}\label{lap gt f alt}
			\begin{split}
				\lap_\Gt f &= \Div_\Gt \grad_\Gt f = \Div_\Gt \grad f - \Div_\Gt \qty[(\vn\vdot\grad f)\vn] \\
				&= \lap_{\R^d} f - \DD^2 f(\vn, \vn) - \kappa \vn \vdot \grad f,
			\end{split}
		\end{equation}
		for any $ C^2 $ function defined in a neighborhood of $ \Gt \subset \R^d $.
	\end{rk}
	
	Next, we shall derive the evolution equations. For the evolution of $ \vn $, it holds that
	\begin{equation*}
		0 \equiv \pd_t \qty(\vn \vdot \F_{,i}) = \pd_t\vn \vdot \F_{,i} + \vn \vdot \vv_{,i},
	\end{equation*} 
	which, together with the fact that $ \vn $ has unit length, implies that
	\begin{equation}\label{eqn pdt n}
		\pd_t \vn = - g^{ij} \qty(\vn \vdot \vv_{,j}) \F_{,i},
	\end{equation}
	in other words,
	\begin{equation}\label{eqn dt n}
		\Dt \vn = - \qty[(\grad\vv)^* \vdot\vn]^\top,
	\end{equation}
	here $\Dt$ is the material derivative along the trajectory of $\vv$.
	For the metric tensor, observe that
	\begin{equation}\label{dt g ij}
		\pd_t g_{ij} = \vv_{,i} \vdot \F_{,j} + \vv_{,j} \vdot \F_{,i} =: 2A_{ij}.
	\end{equation}
	One can check that $ \vb{A} $ is a tensor on $ \Gt $.  In fact, 
	\begin{equation}
		\vb{A} = (\mrm{Def} \, \vv^\top) + v^\perp\II,
	\end{equation}
	where "$\mrm{Def}$" represents the deformation tensor on $ \Gt $. In particular, the material derivative of the area element is:
	\begin{equation*}
		\begin{split}
			\dv{t}(\sqrt{\det(g_{ij})}) = \frac{1}{2} \sqrt{\det(g_{ij})} g^{kl} \pd_t (g_{kl}) = (\Div_\Gt \vv) \sqrt{\det(g_{ij})},
		\end{split}
	\end{equation*}
	i.e., 
	\begin{equation}\label{eqn Dt dS}
		\Dt \dd{S_t} = \Div_\Gt \vv \dd{S_t}.
	\end{equation}
	The evolution equation for the second fundamental form is
	\begin{equation}\label{eqn dt II}
		\pd_t \ii_{ij} = - \vn \vdot (\vv_{,ij} - \Gamma_{ij}^k \vv_{,k}),
	\end{equation}
	in particular,
	\begin{equation}\label{eqn dt II tan}
		(\Dt\II)^\top = - \DD^\top \qty[(\DD\vv)^*\vn]^\top - \II \vdot \qty(\DD\vv)^\top.
	\end{equation}
	The evolution of $ \kappa $ is given by
	\begin{equation}\label{eqn dt kappa}
		\begin{split}
			\Dt \kappa \coloneqq\, &\pd_t(\kappa \circ\F) = \pd_t\qty(\ii_{ij}g^{ij}) \\
			=\, &\vn\vdot\qty[g^{ij}\qty(-\vv_{,ij}+\Gamma_{ij}^k\vv_{,k})] - 2\ii_{ij}{A}^{ij} \\
			=\, &-\vn \vdot \lap_\Gt \vv - 2 \ip{\II}{\vb{A}},
		\end{split}
	\end{equation}
	where $ \ip{\cdot}{\cdot} $ is the standard inner product of tensors defined by
	\begin{equation*}
		\ip{\vb{A}}{\vb{B}} \coloneqq A_{ij} B^{ij} = A_{ij} B_{kl} g^{ik} g^{jl}.
	\end{equation*}
	The second order evolution equation is:
	\begin{equation}
		\begin{split}
			\Dt^2\kappa =\, &-\vn\vdot\lap_\Gt(\Dt\vv)+ 2\vn\vdot(\grad\vv)\vdot(\lap_\Gt\vv)^\top + 4 \ip{\vb{A}}{\vn\vdot\qty(\DD_\Gt)^2\vv} -\kappa \abs{\qty[(\grad\vv)^*\vdot\vn]^\top}^2 \\ &+ 4 \ip{\II \vdot\vb{A} + \vb{A}\vdot\II}{\vb{A}} -2\ip{\II}{\DD_\Gt \vv \vdot \DD_\Gt \vv} - 2\ip{\II}{[\DD(\Dt\vv)]^\top}.
			\end{split}
	\end{equation}
	Here we explain some terms appeared in the last expressions. If one assumes that
	\begin{equation*}
		\vn \equiv N^\alpha \vb{e}_{(\alpha)} \qand \vv \equiv v^\alpha \vb{e}_{(\alpha)},
	\end{equation*}
	for which $ \vb{e}_{(\alpha)} (1 \le \alpha \le d) $ is an orthonormal basis of $ \R^d $, then
	\begin{equation*}
		\vn \vdot (\grad \vv) \vdot (\lap_\Gt \vv)^\top = \sum_{\alpha} N^\alpha \grad v^\alpha \vdot  (\lap_\Gt \vv)^\top \qc
		\vn \vdot (\DD_\Gt)^2 \vv = \sum_{\alpha} N^\alpha (\DD_\Gt)^2 v^\alpha,
	\end{equation*}
	and
	\begin{equation*}
		\DD_\Gt \vv \vdot \DD_\Gt \vv = \sum_{\alpha} \DD_\Gt v^\alpha \otimes \DD_\Gt v^\alpha.
	\end{equation*}
	By using the identity \eqref{lap vn}, one can derive an alternate formula:
	\begin{equation}\label{eqn Dt2 kappa alt}
		\begin{split}
			\Dt^2 \kappa =\,  &-\lap_\Gt(\vn\vdot\Dt\vv) - \abs{\II}^2 (\vn\vdot\Dt\vv) + \grad_\Gt\kappa \vdot\Dt\vv + 2\vn\vdot(\grad\vv)\vdot(\lap_\Gt\vv)^\top \\
			&+ 4 \ip{\vb{A}}{\vn\vdot\qty(\DD_\Gt)^2\vv}
			-\kappa \abs{\qty[(\grad\vv)^*\vdot\vn]^\top}^2 -2\ip{\II}{\DD_\Gt \vv \vdot \DD_\Gt \vv}  + 4 \ip{\II \vdot\vb{A} + \vb{A}\vdot\II}{\vb{A}}.
		\end{split}
	\end{equation}
	
	\subsection{Reference hypersurface}
	
	Let $ k $ be an integer with $ k \ge 2 $, and $ \Gs \subset \Om $ a compact reference hypersurface without boundary separating $ \Om $ into two disjoint simply-connected domains $ \Om_*^\pm $. Assume that $ \Gs $  is of Sobolev class $ H^{\kk + 1} $. Denote by $ \vn_{*+} $ the outward unit normal of $ \pd \Om_*^+ = \Gs $ and $ \vn_{*-} \coloneqq - \vn_{*+} $ the outward unit normal of $ \Gs \subset \pd \Om_*^- $. Let $ \II_{*\pm} $ be the second fundamental form of $ \Gs $ with respect to $ \vn_{*\pm} $, and $ \kappa_{*\pm} $ the corresponding mean curvature.
	
	As in \cite{Shatah-Zeng2011}, we shall consider the evolution of hypersurfaces in a tubular neighborhood of $ \Gs $. Although it is natural to take normal bundle coordinates of $ \Gs $ in classical geometric arguments, it would be better not to do so. Indeed, if $ \Gs $ is of finite regularity, $ \vn_* $ has one less derivatives than $ \Gs $, hence one needs to take another transversal vector field to obtain the Fermi coordinates of the same regularity as that of $ \Gs $. For example, one can take a unit vector field $ \vb*{\nu} \in H^{\kk+1} (\Gs; \mathbb{R}^{2})$ for which $ \vnu \cdot \vn_{*+} \ge 9/10 $ by mollifying $ \vn_* $.
	
	It follows from the implicit function theorem that there exists a constant $ \delta_0 > 0 $ depending on $ \Gs $ and $ \vnu $ so that 
	\begin{equation*}
		\begin{split}
			\varphi: \, &\Gs \times (-\delta_0, \delta_0) \to \R^3 \\
			&(p, \gamma) \mapsto p + \gamma \vnu
		\end{split}
	\end{equation*}
	is an $ H^{\kk + 1} $ diffeomorphism onto a neighborhood of $ \Gs $. Therefore, each hypersurface $ \Gamma $ close to $ \Gs $ in the $ C^1 $ topology is associated to a unique height function $ \gamma_\Gamma : \Gs \to \R $ so that
	\begin{equation}
		\Phi_\Gamma (p) \coloneqq p + \gamma_\Gamma(p)\vnu(p)
	\end{equation}
	is a diffeomorphism from $ \Gs $ to $ \Gamma $. Thus, one can use the function $ \gamma_\Gamma $ to represent the hypersurface $ \Gamma $.
	
	\begin{defi}\label{def Lambda}
		For $ \delta>0 $ and $ \frac{1+3}{2} < s \le 1+\kk $, define $ \Lambda(\Gs, s, \delta) $ to be the collection of all hypersurfaces $ \Gamma $ close to $ \Gs $, whose associated coordinate functions $ \gamma_\Gamma $ satisfy $ \abs{\gamma_\Gamma}_{H^s(\Gs)} < \delta $.
	\end{defi}
	
	As $ s > \frac{3-1}{2} + 1 $ implies $ H^s(\Gs) \hookrightarrow C^1(\Gs) $, $ \delta \ll 1 $ yields that each $ \Gamma \in \Lambda $ also separates $ \Om $ into two disjoint simply-connected domains.
	
	\subsection{Recovering a hypersurface from its mean curvature}
	Here, we characterize the moving hypersurface by its mean curvature $ \kappa_+ \coloneqq \tr \II_+ $. Recall that the second fundamental form is defined by
	\begin{equation}\label{second fundamental form}
		\II_+ (\vb*{\tau}) \coloneqq \DD_{\vb*{\tau}} \vn_+ \qfor \vb*{\tau} \in \mrm{T}\Gamma.
	\end{equation}
	For an $ H^s $ hypersurface $ \Gamma \in \Lambda(\Gs, s, \delta_0) $ with $ s> 2 $, the unit normal $ \vn_+ $ has the same regularity as $ \grad \gamma_\Gamma $. Then the mapping from $ \gamma_\Gamma \in \H{s} $ to the mean curvature $ \kappa_+ \circ \Phi_\Gamma \in \H{s-2} $ is smooth. 
	
	In order to establish a bijection between them, one may consider a modification
	\begin{equation}\label{def ka}
		\K [\gmgm](p) \equiv \ka (p) \coloneqq \kappa_+ \circ \Phi_\Gamma (p) + a^2 \gmgm (p) \qfor p \in \Gs,
	\end{equation}
	where $ a $ is a parameter depending only on $ \Gs $ and $ \vnu $ (c.f. \cite{Shatah-Zeng2011}).
	
	For a small constant $ \delta_0 > 0 $, define
	\begin{equation}\label{def lambda*}
		\Lambda_* \coloneqq \Lambda \qty(\Gs, \kk - \frac{1}{2}, \delta_0).
	\end{equation}
	Then, the following lemma holds:
	\begin{lem}\label{est ii}
		For $ \Gamma \in \Lambda_* $ with $ \kappa_+ \in H^s(\Gamma) $, $ \kk - \frac{5}{2} \le s \le \kk - 1 $, the following estimate holds:
		\begin{equation}
			\abs{\vn_+}_{H^{s+1}(\Gamma)} + \abs{\II_+}_{H^s(\Gamma)} \le C_* \qty(1 + \abs{\kappa_+}_{H^s(\Gamma)}),
		\end{equation}
		for some constant $ C_* $ depending only on $ \Lambda_* $.
	\end{lem}
	The proof of the lemma follows from the bootstrap arguments and Simons' identity \eqref{simons' identity}.
	For the details, one can refer to \cite[p. 719]{Shatah-Zeng2008-Geo}.
	
	If $ \Lambda_*  $ is regarded as an open subset of $ \H{\kk-\frac{1}{2}} $, then $ \K $ is a $ C^3$-morphism from $ \Lambda_* \subset \H{\kk - \frac{1}{2}} $ to $ \H{\kk -\frac{5}{2}} $. Furthermore, by taking $ a \gg 1 $, one may deduce from the positivity of $ \eval{(\fdv*{{\K}}{\gamma_\Gamma})}_{\Gs}^{} $ that $ \K $ is actually a $ C^3 $ diffeomorphism, and the following proposition holds (c.f. \cite[Lemma 2.2]{Shatah-Zeng2011}):
	
	\begin{prop}\label{prop K}
		There are positive constants $ C_*, \delta_0, \delta_1, a_0 $ depending only on $ \Gs $ and $ \vnu $ such that for $ a \ge a_0 $, $ \K $ is a $ C^3 $ diffeomorphism from $ \Lambda_* \subset \H{\kk-\frac{1}{2}} $ to $ \H{\kk-\frac{5}{2}} $. Denote by
		\begin{equation*}
			B_{\delta_1} \coloneqq \Set*{\ka \given \abs{\ka - \kappa_{*+}}_{\H{\kk-\frac{5}{2}}} < \delta_1},
		\end{equation*}
		where $ \kappa_{*+} $ is the mean curvature of $ \Gs $ with respect to $ \vn_{*+} $, then
		\begin{equation*}
			\abs{\K^{-1}}_{C^3\qty(B_{\delta_1}; \H{\kk-\frac{1}{2}})} \le C_*.
		\end{equation*}
		Furthermore, if $ \ka \in B_{\delta_1} \cap \H{s-2}$ with $ \kk-\frac{1}{2} \le s \le \kk +1 $, then $ \gmgm,  \Phi_\Gamma \in \H{s}$, and for $ \max\qty{s'-2, -s} \le s'' \le s' \le s $, it holds that
		\begin{equation}\label{var est K^-1}
			\abs{\var{\K^{-1}}(\ka)}_{\LL\qty(\H{s''}; \H{s'})} \le C_* a^{s'-s''-2} \qty(1+ \abs{\ka}_{\H{s-2}}),
		\end{equation}
		where $\var{\K^{-1}}$ is the functional variation of $\K^{-1}$.		  		
	\end{prop}
	
	\subsection{Harmonic coordinates and Dirichlet-Neumann operators}\label{sec harmonic coord}
	
	For a given hypersurface $$ \Gamma \in \Lambda(\Gs, s, \delta) $$, define a map $ \mathcal{X}_\Gamma^\pm : \Om_*^\pm \to \Om_\Gamma^\pm $ by
	\begin{equation}
		\begin{cases*}
			\lap_y \X_\Gamma^\pm = 0 &for $ y \in \Om_*^\pm $, \\
			\X_\Gamma^\pm (z) = \Phi_\Gamma(z) &for $ z \in \Gamma_* $, \\
			\X_\Gamma^- (z') = z' &for $ z' \in \pd \Omega $.
		\end{cases*}
	\end{equation}
	Then, it is clear that
	\begin{equation}
		\norm{\nabla \X_\Gamma^\pm - \id|_{\Om_*^\pm}}_{H^{s-\frac{1}{2}}(\Om_*^\pm)} \le C \abs{\gamma_\Gamma}_{H^s(\Gs)} < C\delta,
	\end{equation}
	where $ C > 0 $ is uniform in $ \Gamma \in \Lambda(\Gs, s, \delta) $. Thus there is a constant $ \delta_0 >0 $ determined by $ \Gs $ and $ \vnu $, for which $ \X_\Gamma^\pm $ are diffeomorphisms from $ \Om_*^\pm $ to $ \Om^\pm_\Gamma $ respectively, whenever $ \delta \le \delta_0 $.
	
	With the notations in \eqref{def lambda*}, we list some basic inequalities, whose proofs are standard (c.f. \cite{Shatah-Zeng2008-Geo,Bahouri-Chemin-Danchin2011}).
	\begin{lem}\label{lem composition harm coordi}
		Suppose that $ \Gamma \in \ls $. Then there are constants $ C_1, C_2 > 0 $, depending on $ \ls $, so that
		\begin{enumerate}
			\item If $ u_\pm \in H^\sigma (\Omega_\Gamma^\pm) $ for $ \sigma \in \qty[-\kk, \kk] $, then 
			\begin{equation*}
				\dfrac{1}{C_1} \norm{u_\pm}_{H^\sigma(\Om_\Gamma^\pm)} \le \norm{u_\pm \circ \X_\Gamma^\pm}_{H^\sigma(\Om_*^\pm)} \le C_1  \norm{u_\pm}_{H^\sigma(\Om_\Gamma^\pm)}.
			\end{equation*}
			\item If $ f \in H^s (\Gamma) $ for $ s \in \qty[\frac{1}{2}-\kk, \kk-\frac{1}{2}] $, then 
			\begin{equation*}
				\dfrac{1}{C_2} \abs{f}_{H^{s}(\Gamma)} \le \abs{f\circ\Phi_\Gamma}_{\H{s}} \le C_2 \abs{f }_{H^{s}(\Gamma)}.
			\end{equation*}
		\end{enumerate}
	\end{lem}
	
	\begin{lem}\label{lem product est}
		Assume that $ \Gamma \in \Lambda_* $. Then there are constants $ C_1, C_2 > 0 $ determined by $ \Lambda_* $ such that
		\begin{enumerate}
			\item For $ u_\pm \in H^{\sigma_1}(\Om_\Gamma^\pm)  $, $ w_\pm \in H^{\sigma_2}(\Om_\Gamma^\pm)$ and $ \sigma_1 \le \sigma_2 $,
			\begin{equation*}
				\begin{split}
					\norm{u_\pm \cdot w_\pm}_{H^{\sigma_1 + \sigma_2 - \frac{3}{2}}(\Om_\Gamma^\pm)} \le C_1  \norm{u_\pm}_{H^{\sigma_1}(\Om_\Gamma^\pm)} \norm{w}_{H^{\sigma_2}(\Om_\Gamma^\pm)} \qif \sigma_2 < \dfrac{3}{2} \qc 0 <  \sigma_1 + \sigma_2 \le \kk.
				\end{split}
			\end{equation*}
			\begin{equation*}
				\begin{split}
					\norm{u_\pm \cdot w_\pm}_{H^{\sigma_1}(\Om_\Gamma^\pm)} \le C_1 \norm{u_\pm}_{H^{\sigma_1}(\Om_\Gamma^\pm)} \norm{w_\pm}_{H^{\sigma_2}(\Om_\Gamma^\pm)}  \qif \frac{3}{2} < \sigma_2 \le \kk \qc \sigma_1 + \sigma_2 > 0.
				\end{split}
			\end{equation*}
			
			\item For $ f \in H^{s_1}(\Gamma) $, $ g \in H^{s_2}(\Gamma) $ and $ s_1 \le s_2 $, 
			\begin{equation*}
				\abs{fg}_{H^{s_1+s_2-\frac{2}{2}}(\Gamma)} \le C_2 \abs{f}_{H^{s_1}(\Gamma)} \abs{g}_{H^{s_2}(\Gamma)} \qif s_2 < 1 \qc 0 \le s_1 + s_2 \le \kk-\frac{1}{2}.
			\end{equation*}
			\begin{equation*}
				\abs{fg}_{H^{s_1}(\Gamma)} \le C_2 \abs{f}_{H^{s_1}(\Gamma)} \abs{g}_{H^{s_2}(\Gamma)}\qif 1 < s_2 \le \kk-\frac{1}{2} \qc s_1 + s_2 > 0.
			\end{equation*}
		\end{enumerate}
	\end{lem}
	
	For any smooth function $ f $ defined on $ \Gamma \in \Lambda_* $, denote by $ \h_\pm f $ the harmonic extensions to $ \Om_\Gamma^\pm $, namely
	\begin{equation}\label{harmonic ext +}
		\begin{cases*}
			\lap \h_+ f = 0 \qfor x \in \Om_\Gamma^+, \\
			\h_+ f = f \qfor x \in \Gamma,
		\end{cases*}
	\end{equation}
	and
	\begin{equation}\label{harmonic ext 2}
		\begin{cases*}
			\lap \h_- f = 0 \qfor x \in \Om_\Gamma^-, \\
			\h_- f = f \qfor x \in \Gamma, \\
			\DD_{\wt{\vn}} \h_- f = 0 \qfor x \in \pd\Om.
		\end{cases*}
	\end{equation}
	The Dirichlet-Neumann operators are defined to be 
	\begin{equation}\label{def DN op}
		\n_\pm f \coloneqq \vn_\pm \vdot (\grad \h_\pm f)|_{\Gamma}^{}.
	\end{equation}
	
	Assume that $ \Gamma \in \Lambda_* \subset H^{\kk-\frac{1}{2}} $ and $ \frac{3}{2}-\kk \le s \le \kk-\frac{1}{2} $. The Dirichlet-Neumann operators $ \n_\pm : H^{s}(\Gamma) \to H^{s-1}(\Gamma) $ satisfy the following properties (c.f. \cite[pp. 738-741]{Shatah-Zeng2008-Geo}):
	\begin{enumerate}[1.]
		\item $ \n_\pm $ are self-adjoint on $ L^2(\Gamma) $ with compact resolvents;
		\item $\ker(\n_\pm) = \qty{\mathrm{const.}}$;
		\item There is a constant $ C_* > 0 $ uniform in $ \Gamma \in \Lambda_* $  so that
		\begin{equation*}
			C_*\abs{f}_{H^{s}(\Gamma)} \ge \abs{\n_\pm(f)}_{H^{s-1}(\Gamma)} \ge \frac{1}{C_*}\abs{f}_{H^{s}(\Gamma)},
		\end{equation*} 
		for any $ f $ satisfying $ \int_\Gamma f \dd{S} = 0 $;
		\item For $ \frac{1}{2}-\kk \le s_1 \le \kk-\frac{1}{2} $, there is a constant $ C_* $ determined by $ \Lambda_* $ so that
		\begin{equation*}
			\dfrac{1}{C_*} \qty(\mathrm{I}-\lap_\Gamma)^{\frac{s_1}{2}} \le \qty(\mathrm{I}+\n_\pm)^{s_1} \le C_*\qty(\mathrm{I}-\lap_\Gamma)^{\frac{s_1}{2}},
		\end{equation*}
		i.e., the norms on $ H^{s_1}(\Gamma) $ defined by interpolating $ \qty(\mathrm{I}-\lap_\Gamma)^\frac{1}{2} $ and $ \qty(\mathrm{I}+\n_\pm) $ are equivalent;
		\item For $ \frac{1}{2}-\kk \le s_2 \le \kk-\frac{3}{2} $,
		\begin{equation*}
			(\n_\pm)^{-1} : H^{s_2}_{0}(\Gamma) \to H^{s_2+1}_{0}(\Gamma),
		\end{equation*}
		\begin{equation*}
			H^{s_2}_{0}(\Gamma) = \Set*{f\in H^{s_2}(\Gamma) \given \int_\Gamma f \dd{S} = 0 }
		\end{equation*}
		are well-defined and bounded uniformly in $ \Gamma \in \Lambda_* $.
	\end{enumerate}
	
	\begin{rk}
		We shall denote by $ (\n_\pm)^{-1} \coloneqq (\n_\pm)^{-1} \circ\mathcal{P} $ for simplicity, where
		\begin{equation*}
			\mathcal{P}f \coloneqq f - \fint_\Gamma f \dd{S} \equiv f - \ev{f}.
		\end{equation*}
		is the projection into mean-free functions on $ \Gamma $.
	\end{rk}
	
	The following notations will be used later:
	\begin{equation}\label{def operator n}
		\bn \coloneqq \dfrac{1}{\rho_+} \n_+ + \dfrac{1}{\rho_-} \n_-,
	\end{equation}
	\begin{equation}\label{def operator n tilde}
		\wt{\n} \coloneqq \qty(\dfrac{1}{\rho_+} \n_+) \bn^{-1} \qty(\dfrac{1}{\rho_-}\n_-) = \qty[\qty(\dfrac{1}{\rho_+}\n_+)^{-1} + \qty(\dfrac{1}{\rho_-} \n_-)^{-1}]^{-1}.
	\end{equation}
	
	At the end of this subsection, we state an important lemma (c.f. \cite[p. 863]{Shatah-Zeng2008-vortex}):
	\begin{lem}\label{lem lap-n}
		Suppose that $ \Gamma \in \Lambda_* $ with $ \kappa \in H^{\kk-\frac{3}{2}}(\Gamma) $. Then for $ \frac{1}{2} - \kk \le s \le \kk-\frac{1}{2} $, one has
		\begin{equation}
			\abs{\qty(-\lap_\Gamma)^{\frac{1}{2}} - \n_\pm}_{\LL\qty(H^s(\Gamma))} \le C_* \qty(1+\abs{\kappa}_{H^{\kk-\frac{3}{2}}(\Gamma)}),
		\end{equation}
		where the constant $ C_*>0 $ is uniform in $ \Gamma \in \Lambda_* $.
	\end{lem}
	
	\subsection{Commutator estimates}
	For vector fields (not necessarily solenoidal) $ \vv_\pm(t) : \Om_t^\pm \to \R^3 $, denote by
	\begin{equation*}
		\Dt_\pm \coloneqq \pd_t + \DD_{\vv_\pm}.
	\end{equation*}
	Then one has the following lemma:
	\begin{lem}\label{Dt comm est lemma}
		Suppose that $ \Gt \in \Lambda_* $, and $ \vv_\pm \in H^{\kk}(\Om_t^\pm) $ are the evolution velocities of $ \Om_t^\pm $, so $ \vv_\pm $ are both evolution velocities of $ \Gt $. Let $ f (t, x) : \Gt \to \R $ and $ h (t, x) : \OGt \to \R $ be two functions. Then the following commutator estimates hold:
		\begin{enumerate}[I.]
			\item For $ 1 \le s \le \kk $, $\norm{\comm{\Dt_\pm}{\h_\pm}f}_{H^s(\Om_t^\pm)} \lesssim_{\Lambda_*} \abs{f}_{H^{s-\frac{1}{2}}(\Gt)}\cdot \norm{\vv_\pm}_{H^{\kk}(\Om_t^\pm)}$;
			
			\item For $ 1 \le s \le \kk $, $\norm{\comm{\Dt_\pm}{\lap_\pm^{-1}}h_\pm}_{H^s(\Om_t^\pm)} \lesssim_{\Lambda_*} \norm{h_\pm}_{H^{s-2}(\Om_t^\pm)} \cdot \norm{\vv_\pm}_{H^\kk(\Om_t^\pm)} $;
	
			\item For $ -\frac{1}{2} \le s \le \kk- \frac{3}{2} $, $\abs{\comm{\Dt_\pm}{\n_\pm}f}_{H^s(\Gt)} \lesssim_{\Lambda_*} \abs{f}_{H^{s+1}(\Gt)} \cdot \abs{\vv_\pm}_{H^{\kk-\frac{1}{2}}(\Gt)} $;
			
			\item For $ \frac{1}{2} \le s \le \kk - \frac{1}{2} $, $\abs{\comm{\Dt_\pm}{\n_\pm^{-1}}f}_{H^s(\Gt)} \lesssim_{\Lambda_*} \abs{f}_{H^{s-1}(\Gt)} \cdot \abs{\vv_\pm}_{H^{\kk-\frac{1}{2}}(\Gt)}$;
			
			\item For $ -2 \le s \le \kk - \frac{5}{2} $, $\abs{\comm{\Dt_\pm}{\lap_\Gt}f}_{H^s(\Gt)} \lesssim_{\Lambda_*} \abs{f}_{H^{s+2}(\Gt)} \cdot \abs {\vv_\pm}_{H^{\kk-\frac{1}{2}}(\Gt)} $;
			
			\item For $ 0 \le s \le \kk-1 $, $ \norm{\comm{\Dt_\pm}{\DD}h_\pm}_{H^s(\Om_t^\pm)} \less \norm{h_\pm}_{H^{s+1}(\Om_t^\pm)} \cdot \norm{\vv_\pm}_{H^\kk(\Om_t^\pm)}$.
		\end{enumerate}
	\end{lem}
	The proof of Lemma \ref{Dt comm est lemma} follows from the identities introduced in \cite[pp. 709-710]{Shatah-Zeng2008-Geo} and standard product estimates.
	
	\subsection{Div-curl systems}\label{sec div-cul system}
	In this subsection, we list some basic results on div-curl systems (c.f. \cite{Cheng-Shkoller2017} for details):
	\begin{thm}\label{thm div-curl}
		Suppose that $ U $ is a bounded domain in $ \R^3 $ for which $\pd U \in H^{\kk-\frac{1}{2}} $. Given $ \vb{f}, g \in H^{l-1}(U) $ with $ \div \vb{f} = 0 $ and $ h \in H^{l-\frac{1}{2}}(\pd U) $, consider the system:
		\begin{subnumcases}{}\label{div-curl eqn}
			\curl \vbu = \vb{f} &in $ U $,\\
			\div \vbu = g &in $ U $, \\
			\vbu \vdot \vn = h &on $ \pd U $.
		\end{subnumcases}
		If on each connected component $ \Gamma $ of $ \pd U $, one has that
		\begin{equation}\label{compatibility div-curl}
			\int_\Gamma \vf \vdot \vn \dd{S} = 0,
		\end{equation}
		and the following compatibility condition holds true:
		\begin{equation}
			\int_{\pd U} h \dd{S} = \int_U g \dd{x},
		\end{equation}
		then for $ 1 \le l \le \kk-1 $, there is a solution $ \vbu \in H^l(U) $  such that
		\begin{equation}
			\norm{\vbu}_{H^l(U)} \le C\qty(\abs{\pd U}_{H^{\kk-\frac{1}{2}}}) \cdot \qty(\norm{\vf}_{H^{l-1}(U)} + \norm{g}_{H^{l-1}(U)} + \abs{h}_{H^{l-\frac{1}{2}}(\pd U)}).
		\end{equation}
		The solution is unique whenever $ U $ is simply-connected.
	\end{thm}
	\begin{rk}
		If $ \vb{f} = \curl \vbu $ for some vector field $ \vbu $, then (\ref{compatibility div-curl}) holds naturally (see \cite[Remark 1.2]{Cheng-Shkoller2017}).
	\end{rk}
	
	\section{Reformulation of the Problem}\label{sec reform}
	
	\subsection{Velocity fields on the interface}
	
	Since the interface $ \Gt $ separates two plasmas, and $ \vv_\pm $ have the same normal components on $ \Gt $, it is natural to consider the evolution of $ \kappa_+ $ with respect to some weighted velocity
	\begin{equation}
		\vbu_\lambda \coloneqq \lambda \vv_+ + (1-\lambda) \vv_-
	\end{equation}
	for some $ 0 \le \lambda \le 1 $.
	
	Denote by
	\begin{equation*}
		\Dtl \coloneqq \pd_t + \DD_{\vbu_\lambda},
	\end{equation*}
	then
	\begin{equation*}
		\Dtl \vbu_\lambda = \qty(\pd_t + \lambda \DD_{\vv_+} + (1-\lambda)\DD_{\vv_-}) \qty(\lambda \vv_+ + (1-\lambda) \vv_-).
	\end{equation*}
	In view of (\ref{eqn Dtv}), it holds that
	\begin{equation*}
		\begin{split}
			\Dtl \vbu_\lambda &= \lambda^2 \qty(-\dfrac{1}{\rho_+} \grad p^+ + \DD_{\vh_+} \vh_+) + (1-\lambda)^2 \qty(-\dfrac{1}{\rho_-}\grad p^- + \DD_{\vh_-} \vh_-) \\
			&\quad\, + \lambda(1-\lambda) (\Dt_+ \vv_- + \Dt_- \vv_+).
		\end{split}
	\end{equation*}
	Since
	\begin{equation*}
		\begin{split}
			\Dt_+ \vv_- = \pd_t \vv_- + \DD_{\vv_+} \vv_- = \DD_{\vv_+ - \vv_-} \vv_- - \dfrac{1}{\rho_-} \grad p^- + \DD_{\vh_-} \vh_-,
		\end{split}
	\end{equation*}
	and
	\begin{equation*}
		\Dt_- \vv_{+} = \DD_{\vv_- - \vv_+} \vv_+ - \dfrac{1}{\rho_+} \grad p^+ + \DD_{\vh_+} \vh_+,
	\end{equation*}
	one may write that
	\begin{equation*}
		\Dt_+ \vv_- + \Dt_- \vv_+ = - \DD_{\vw} \vw - \qty(\dfrac{1}{\rho_+} \grad p^+ + \dfrac{1}{\rho_-} \grad p^- ) + \DD_{\vh_+} \vh_+ + \DD_{\vh_-} \vh_-,
	\end{equation*} 
	where $ \vw \in \mathrm{T} \Gt $ is defined to be
	\begin{equation}
		\vw \equiv \llbracket \vv \rrbracket \coloneqq \vv_+ - \vv_-.
	\end{equation}
	Therefore,
	\begin{equation}\label{eqn dtl ulambda}
		\begin{split}
			\Dtl \vbu_\lambda  = - \dfrac{\lambda}{\rho_+} \grad p^+ - \dfrac{1-\lambda}{\rho_-}\grad p^-   + \lambda \DD_{\vh_+}\vh_+ + (1-\lambda) \DD_{\vh_-} \vh_- - \lambda(1-\lambda) \DD_\vw \vw.
		\end{split}
	\end{equation}
	
	Next, we introduce a useful decomposition of the pressure:
	\begin{equation}\label{decomp p}
		\dfrac{1}{\rho_\pm} p^\pm = p^\pm_{\vv,\vv} - p^\pm_{\vh, \vh} + \alpha^2 p^\pm_\kappa + p^\pm_b.
	\end{equation}
	Here $ p_{\vb{a},\vb{b}}^\pm $ are the solutions to the following elliptic problems respectively:
	\begin{equation}\label{def p_ab^+}
		\begin{cases*}
			\lap p^+_{\vb{a}, \vb{b}} = - \tr(\DD\vb{a}_+ \vdot \DD\vb{b}_+) \qfor x \in\Omega_t^+,\\
			p^+_{\vb{a}, \vb{b}} = 0 \qfor x \in \Gt;
		\end{cases*}
	\end{equation}
	\begin{equation}\label{def p_ab^-}
		\begin{cases*}
			\lap p_{\vb{a}, \vb{b}}^- = - \tr(\DD\vb{a}_- \vdot \DD\vb{b}_-) \qfor x \in \Om_t^-, \\
			p_{\vb{a}, \vb{b}}^- = 0 \qfor x \in \Gt, \\
			\DD_{\wt{\vn}} p_{\vb{a}, \vb{b}}^- =  \wt{\II}(\vb{a}_-, \vb{b}_-) \qfor x \in \pd\Om,
		\end{cases*}
	\end{equation}
	where $ \vb{a} = \vb{a}_+ \I_{\Om_t^+} + \vb{a}_- \I_{\Om_t^-} $ and $ \vb{b} = \vb{b}_+ \I_{\Om_t^+} + \vb{b}_-\I_{\Om_t^-} $ are solenoidal vector fields satisfying $ \vb{a}_- \vdot \wt{\vn} = 0 = \vb{b}_- \vdot \wt{\vn}  $ on $ \pd \Omega $.
	$ p_\kappa $ and $ p_b $ are given respectively by (c.f. \cite{Shatah-Zeng2011}):
	\begin{equation}\label{def p_kappa}
		p_\kappa^\pm \coloneqq \dfrac{1}{\rho_+ \rho_-} \h_\pm \bn^{-1} \n_\mp \kappa_\pm, \qand p_b^\pm \coloneqq \dfrac{1}{\rho_\pm}\h_\pm \mathfrak{p},
	\end{equation}
	where $ \mathfrak{p} $ is a function defined on $ \Gamma_t $ whose expression will be determined later.
	
	With this decomposition, it is routine to check that $ p_{\vv,\vv}^\pm = p_{\vh, \vh}^\pm = 0 $ , $ \rho_+ p_b^+ = \rho_- p_b^- $, and $ \rho_+ p_\kappa^+ - \rho_- p_\kappa^- = \kappa_+ $ hold simultaneously on $ \Gt $. Namely, (\ref{jump p}) is satisfied automatically.
	
	Next, we will derive the explicit formula of $ \mathfrak{p} $ by using 
	(\ref{eqn Dtv}), (\ref{eqn bdry v}) and (\ref{bdry mag}).
	Indeed, multiplying (\ref{eqn Dtv}) by $ \vn_+ $, one has
	\begin{equation*}
		\vn_+ \vdot	\Dt_+ \vv_+ + \dfrac{1}{\rho_+} \DD_{\vn_+} p^+ = \vn_+ \vdot \DD_{\vh_+} \vh_+,
	\end{equation*}
	which implies that
	\begin{equation*}
		\begin{split}
			\pd_t \theta + \dfrac{1}{\rho_+} \DD_{\vn_+} p^+  =\vv_+ \vdot \Dt_+ \vn_+ - \DD_{\vv_+}\theta + \vn_+ \vdot \DD_{\vh_+} \vh_+.
		\end{split}
	\end{equation*}
	It follows from  (\ref{bdry mag}), \eqref{def II} and \eqref{eqn dt n} that
	\begin{equation*}
		\begin{split}
			\pd_t \theta + \dfrac{1}{\rho_+} \DD_{\vn_+} p^+ 
			&= - \vn_+ \vdot (\DD \vv_+) \vdot \vv_+^\top - \DD_{\vv_+}\theta + \vn_+ \vdot \DD_{\vh_+} \vh_+ \\
			&= - \DD_{\vv_+}\theta -\DD_{\vv_+^\top} \theta + \II_+ \qty(\vv_+^\top, \vv_+^\top) - \II_+ \qty(\vh_+, \vh_+).
		\end{split}
	\end{equation*}
	Similarly,
	\begin{equation*}
		\begin{split}
			-\pd_t \theta + \dfrac{1}{\rho_-} \DD_{\vn_-}p^-
			=\DD_{\vv_-}\theta + \DD_{\vv_-^\top} \theta + \II_- \qty(\vv_-^\top, \vv_-^\top) - \II_- \qty(\vh_-, \vh_-).
		\end{split} 
	\end{equation*}
	Due to the conventions that $ \vn_+ + \vn_- = \vb{0} $, $ \II_+ + \II_- = \vb{0} $, and the relation that $ (\vv_+ - \vv_-) \in \mathrm{T}\Gt $, summing the above two equations yields
	\begin{equation*}
		\begin{split}
			\dfrac{1}{\rho_+} \DD_{\vn_+} p^+ + \dfrac{1}{\rho_-} \DD_{\vn_-} p^- = &- \qty[2\DD_{\vv_+^\top}\theta - \II_+\qty(\vv_+^\top, \vv_+^\top) + \II_+\qty(\vh_+, \vh_+)] \\
			&+ \qty[2\DD_{\vv_-^\top}\theta - \II_+\qty(\vv_-^\top, \vv_-^\top) + \II_+\qty(\vh_-, \vh_-)].
		\end{split}
	\end{equation*}
	According to the decomposition (\ref{decomp p}) of $ p^\pm $ and the relation that
	\begin{equation*}
		\begin{split}
			\DD_{\vn_+} p_\kappa^+ + \DD_{\vn_-} p_\kappa^- 
			= \qty(\dfrac{1}{\rho_+}\n_+)\bar{\n}^{-1}\qty(\dfrac{1}{\rho_-}\n_-)\kappa_+ +\qty(\dfrac{1}{\rho_-}\n_-) \bar{\n}^{-1}\qty(\dfrac{1}{\rho_+}\n_+)\kappa_- 
			= 0,
		\end{split}
	\end{equation*}
	it holds that
	\begin{equation*}
		\begin{split}
			 \dfrac{1}{\rho_+} \DD_{\vn_+} p^+ + \dfrac{1}{\rho_-} \DD_{\vn_-} p^- 
			= \DD_{\vn_+} \qty(p_{\vv,\vv}^+ - p_{\vh,\vh}^+ - p_{\vv,\vv}^- + p_{\vh,\vh}^-)  + \qty(\dfrac{1}{\rho_+}\n_+ + \dfrac{1}{\rho_-} \n_-) \mathfrak{p}.
		\end{split}
	\end{equation*}
	Therefore, one gets the formula
	\begin{equation}\label{def g pm}
		\begin{split}
			\bn \mathfrak{p} = &- \qty[2\DD_{\vv_+^\top}\theta - \II_+\qty(\vv_+^\top, \vv_+^\top) + \II_+\qty(\vh_+, \vh_+) + \DD_{\vn_+}\qty(p_{\vv, \vv}^+ - p_{\vh, \vh}^+)] \\
			&+ \qty[2\DD_{\vv_-^\top}\theta - \II_+\qty(\vv_-^\top, \vv_-^\top) + \II_+\qty(\vh_-, \vh_-) + \DD_{\vn_+}\qty(p_{\vv,\vv}^- - p_{\vh, \vh}^-)] \\
			=: &-g^+ + g^-,
		\end{split}
	\end{equation}
	namely,
	\begin{equation}\label{def frak p}
		\mathfrak{p} = \bn^{-1}(-g^+ + g^-).
	\end{equation}
	In conclusion, if $ \qty(\Gt, \vv, \vh) $ is a solution to (MHD)-(BC) with $ \Gt \in \Lambda_* $, $ \Gt \in H^{\kk+1} $ and $ \vv, \vh \in H^{\kk}(\OGt) $, then the following estimate holds:
	\begin{equation}\label{est Dtl ulam}
		\abs{\Dtl \vbu_\lambda}_{H^{\kk-2}(\Gt)} \le Q_\lambda\qty(\alpha\abs{\ka}_{\H{\kk-1}}, \abs{\ka}_{\H{\kk-\frac{3}{2}}}, \norm{\vv}_{H^{\kk}(\OGt)}, \norm{\vh}_{H^{\kk}(\OGt)}), 
	\end{equation}
	where $ Q_\lambda $ is a generic polynomial depending only on $ \Lambda_* $ and $ \lambda $.
	(Indeed, for $ \frac{1}{2} < s \le \kk - 2 $, one has $ \abs{\grad p_\kappa}_{H^{s}(\Gt)} \lesssim_{\Lambda_*} \abs{\kappa}_{H^{s+1}(\Gt)} $,
	so the best estimate on $ (\Dtl\vbu_\lambda)|_\Gt $ is its $ H^{\kk-2} $ norm.)
	\begin{rk}
		The following formula holds as long as $ \vv_\pm $ are the evolution velocities of $ \Om_t^\pm $ (which is, in particular, independent of the MHD problems):
		\begin{equation}\label{int g^+ - g^-}
			\int_\Gt g^+ - g^- \dd{S} = -\int_{\OGt} \qty(\div \vv)^2 \dd{x}.
		\end{equation}
		So \eqref{def g pm} makes sense whenever $ \vv_\pm $ are both solenoidal. Indeed,
		\begin{equation*}			
			\begin{split}
				&\hspace{-1.5em}\int_\Gt g^+ - g^- \dd{S_t} \\
				=\, &\int_\Gt \qty(\Dt_+ + \DD_{\vv_+^\top}) (\vv_+ \vdot \vn_+)  - \II_+\qty(\vv_+^\top, \vv_+^\top) + \II_+\qty(\vh_+, \vh_+) + \DD_{\vn_+}\qty(p_{\vv, \vv}^+ - p_{\vh, \vh}^+) \dd{S_t} \\
				&+\int_\Gt \qty(\Dt_- + \DD_{\vv_-^\top})(\vv_- \vdot \vn_-)  - \II_-\qty(\vv_-^\top, \vv_-^\top) + \II_-\qty(\vh_-, \vh_-) + \DD_{\vn_-}\qty(p_{\vv,\vv}^- - p_{\vh, \vh}^-) \dd{S_t} \\
				=\, &\int_{\Om_t^+} \Dt_+ \qty(\div\vv_+) \dd{x} +\int_{\Om_t^-} \Dt_- \qty(\div\vv_-) \dd{x} \\
				=\, &\dv{t}(\int_{\Om_t^+} \qty(\div\vv_+) \dd{x}) - \int_{\Om_t^+} \qty(\div \vv_+)^2 \dd{x} +\dv{t}(\int_{\Om_t^-}\qty(\div \vv_-)\dd{x}) - \int_{\Om_t^-} \qty(\div\vv_-)^2 \dd{x} \\
				=\, &-\int_{\OGt} \qty(\div\vv)^2 \dd{x}.
			\end{split}
		\end{equation*}
	\end{rk}
	
	\subsection{Transformation of the velocity}\label{sec var v_*}
	As stated in the preliminary, a vector field defined in a bounded simply-connected domain is determined by its divergence, curl and appropriate boundary conditions. Since both the velocity and magnetic fields are solenoidal, they are uniquely determined by the vorticities, currents and boundary conditions.
	
	Therefore, denote by
	\begin{equation}
		\vom_{*\pm} \coloneqq \qty(\curl \vv_\pm) \circ \X_\Gt^\pm,
	\end{equation}
	then the velocity field $ \vv $ can be uniquely determined by $ \ka $, $ \theta $ and $ \vom_* $ via solving the following div-curl problems:
	\begin{equation}\label{sys div-curl v}
		\begin{cases*}
			\div \vv_\pm = 0 &in $\Om_t^\pm$, \\
			\curl \vv_\pm = \vom_{*\pm} \circ (\X_\Gt^\pm)^{-1} &in $\Om_t^\pm$, \\
			\vv_\pm \vdot \vn_+ = \theta &on $\Gt$, \\
			\vv_- \vdot \wt{\vn} = 0 &on $\pd\Om$.
		\end{cases*}
	\end{equation}
	
	Next, for a function $ f : \Gt \to \R $, it is natural to pull back $ \Dtl f $ to $ \Gs $ via $ \Phi_\Gt $, namely, one needs to look for a vector field $ \vbu_{\lambda *} : \Gs \to \R^3 $ so that
	\begin{equation}\label{pull back Dt to Gs}
		\Dtl_* \qty(f \circ \Phi_\Gt) = \qty(\Dtl f) \circ \Phi_\Gt,
	\end{equation}
	where
	\begin{equation}
		\Dtl_* \coloneqq \pd_t + \DD_{\vbu_{\lambda *}}.
	\end{equation}
	It is necessary that
	\begin{equation*}
		\qty(\DD f)\circ \Phi_\Gt \vdot \qty(\pd_t \Phi_\Gt + \DD\Phi_\Gt \vdot \vbu_{\lambda *}) = \qty(\DD f)\circ \Phi_\Gt \vdot \vbu_\lambda \circ \Phi_\Gt.
	\end{equation*}
	On the other hand, it suffices to define
	\begin{equation}\label{def u lambda *}
		\begin{split}
			\vbu_{\lambda *} \coloneqq  \qty(\DD\Phi_\Gt)^{-1}\qty(\vbu_\lambda \circ\Phi_\Gt - \pd_t\Phi_\Gt) 
			=  \qty(\DD\Phi_\Gt)^{-1}\qty[\vbu_\lambda \circ\Phi_\Gt - (\pd_t \gamma_\Gt)\vnu ].
		\end{split}
	\end{equation} 
	Since
	\begin{equation*}
		\theta = \qty(\pd_t \gamma_\Gt \vnu)\circ\qty(\Phi_\Gt)^{-1} \vdot \vn_+,
	\end{equation*}
	one has
	\begin{equation*}
		\qty[\vbu_\lambda - \qty(\pd_t \gamma_\Gt \vnu)\circ\qty(\Phi_\Gt)^{-1} ] \vdot \vn_+ = 0,
	\end{equation*}
	i.e., $ \qty[\vbu_\lambda - \qty(\pd_t \gamma_\Gt \vnu)\circ\qty(\Phi_\Gt)^{-1} ] \in \mrm{T}\Gt $ and  $ \vbu_{\lambda*} \in \mrm{T}\Gs $.
	
	\subsubsection*{Variational estimates}
	
	In order to compute the variation of $ \vbu_{\lambda*} $, one can assume that $ \ka $ and $ \vom_* $ depend on a parameter $ \beta $. By (\ref{def u lambda *}), it suffices to compute $ \pd_\beta \vv_{\pm*} $. Applying $ \pdv*{\beta} $ to the identity
	\begin{equation*}
		\qty(\DD \Phi_\Gt) \vdot \vv_{\pm*} = \vv_\pm \circ \Phi_\Gt - (\pd_t\gamma_\Gt)\vnu,
	\end{equation*}
	one has
	\begin{equation*}
		\DD\qty(\pd_\beta \gt \vnu) \vdot \vv_{\pm*} + (\DD\Phi_\Gt) \vdot \pd_\beta\vv_{\pm*} = \pd_\beta\qty(\vv_\pm \circ \Phi_\Gt) - \qty(\pd^2_{t\beta}\gt)\vnu,
	\end{equation*}
	for which
	\begin{equation*}
		\pd_\beta (\vv_\pm \circ \Phi_\Gt) = \qty(\pd_\beta \vv_\pm + \DD_{(\pd_\beta \gt \vnu)\circ (\Phi_\Gt)^{-1}} \vv_\pm) \circ \Phi_\Gt.
	\end{equation*}
	Denote by $ \vb*{\mu} \coloneqq (\pd_\beta \gt \vnu)\circ (\Phi_\Gt)^{-1}  $ and use the notation
	\begin{equation}\label{def Dbt}
		\Dbt \coloneqq \pd_\beta + \DD_{\vb*{\mu}}.
	\end{equation}
	Then
	\begin{equation}\label{eqn pd beta v+*}
		\pd_\beta \vv_{\pm*} = (\DD \Phi_\Gt)^{-1} \vdot \qty[\qty(\Dbt\vv_{\pm})\circ\Phi_\Gt - \qty(\pd^2_{t\beta}\gt)\vnu -\DD(\pd_\beta\gt\vnu)\vdot\vv_{\pm*} ].
	\end{equation}
	In particular, $ \qty[(\DD\Phi_\Gt) \vdot \pd_\beta\vv_{\pm*}] \circ \Phi_\Gt^{-1} \vdot \vn_+ \equiv 0 $, so $ \pd_\beta \vv_{\pm*} \in \mrm{T}\Gs $.
	
	Applying $\Dbt$ to \eqref{sys div-curl v} and utilizing the commutator estimates together with the div-curl estimates, one can derive that for $ \frac{1}{2} \le \sigma \le \kk -\frac{3}{2} $,
	\begin{equation}\label{est pd beta v+*}
		\begin{split}
			\abs{\pd_\beta \vv_{\pm*}}_{\H{\sigma}} \lesssim_{\Lambda_*} &\abs{\pd_\beta \gt}_{\H{\sigma + 1}} \qty(\norm{\vom_{*\pm}}_{H^{\kk-1}(\Om_*^\pm)} + \abs{\pd_t \gt }_{\H{\kk-\frac{1}{2}}}) \\
			&+ \norm{\pd_\beta\vom_{*\pm}}_{H^{\sigma-\frac{1}{2}}(\Om_*^\pm)} + \abs{\pd^2_{t\beta}\gt}_{\H{\sigma}}.
		\end{split}
	\end{equation}
	Recalling that $ \gt = \K^{-1}(\ka) $, one can obtain that
	\begin{equation}\label{eqn pd gt}
		\pd_\beta \gt = \var{\K^{-1}}(\ka) [\pd_\beta\ka],
	\end{equation}
	and
	\begin{equation}\label{eqn pd2 gt}
		\pd^2_{t\beta}\gt = \var{\K^{-1}}(\ka)[\pd^2_{t\beta}\ka] +  \var[2]{\K^{-1}}(\ka)[\pd_t\ka, \pd_\beta\ka].
	\end{equation}
	In conclusion, the linear relations imply the existence of six linear operators $ \opb_\pm(\ka) $, $ \opf_\pm(\ka) $ and $ \opg_\pm(\ka, \pd_t\ka, \vom_{*\pm}) $ whose ranges are all in $ \mrm{T}\Gs $, so that
	\begin{equation}\label{eqn pd beta v}
		\pd_\beta \vv_{\pm*} = \opb_\pm(\ka)\pd^2_{t\beta}\ka + \opf_\pm(\ka)\pd_\beta\vom_{*\pm} + \opg_\pm(\ka, \pd_t \ka, \vom_{*\pm})\pd_\beta\ka.
	\end{equation}
	Moreover, the following lemma holds, whose proof will be given in the Appendix:
	\begin{lem}\label{lem3.1}
		Suppose that $ a \ge a_0 $ and $ \ka \in B_{\delta_1} \subset \H{\kk-\frac{5}{2}} $, where $ a_0 $ and $ B_{\delta_1} $ are given in Proposition \ref{prop K}. If   
		$ s'-2 \le s'' \le s' \le \kk-\frac{3}{2} $, $ s' \ge \frac{1}{2}  $, and $ \frac{1}{2} \le s \le \kk-\frac{3}{2} $, then the following estimates hold:
		\begin{equation}\label{est B}
			\abs{\opb_\pm(\ka)}_{\LL\qty(\H{s''}; H^{s'}(\Gs; \mrm{T}\Gs)) } \le C_* a^{s'-s''-2},
		\end{equation}
		\begin{equation}\label{est var B}
			\abs{\var\opb_\pm(\ka)}_{\LL\qty[\H{\kk-\frac{5}{2}}; \LL\qty(\H{s-2}; \H{s})]} \le C_*,
		\end{equation}
		\begin{equation}\label{est F}
			\abs{\opf_\pm(\ka)}_{\LL\qty(H^{s-\frac{1}{2}}(\Om_*^\pm); \H{s})} \le C_*,
		\end{equation}
		and
		\begin{equation}\label{est var F}
			\abs{\var{\opf_\pm}(\ka)}_{\LL\qty[\H{\kk-\frac{5}{2}};  \LL\qty(H^{s-\frac{1}{2}}(\Om_*^\pm); \H{s})]} \le C_*,
		\end{equation}
		where $ C_* $ is a constant depending only on $ \Lambda_{*} $. 
		
		Moreover, if $ \ka \in B_{\delta_1} \cap \H{\kk-\frac{3}{2}} $, $ \pd_t \ka \in \H{\kk-\frac{5}{2}} $, and $ \vom_* \in H^{\kk-1}(\Om\setminus\Gs) $, then for $  \sigma'-2 \le \sigma'' \le \sigma' \le \kk + \frac{1}{2} $, $ \sigma' \ge \frac{1}{2} $, and $s$ given above, it holds that
		\begin{equation}\label{est B version2}
			\abs{\opb_\pm(\ka)}_{\LL\qty[\H{\sigma''}; \H{\sigma'}]} \le a^{\sigma'-\sigma''-2} Q\qty(\abs{\ka}_{\H{\kk-\frac{3}{2}}}),
		\end{equation}
		\begin{equation}\label{est G}
			\begin{split}
				\abs{\opg_\pm(\ka, \pd_t\ka, \vom_{*\pm})}_{\LL\qty[\H{s-1}; \H{s}]}  \le Q\qty(\abs{\pd_t\ka}_{\H{\kk-\frac{5}{2}}}, \norm{\vom_{*\pm}}_{H^{\kk-1}(\Om_*^\pm)}),
			\end{split}
		\end{equation}
		and for $ \frac{1}{2} \le \sigma \le \kk-\frac{5}{2} $,
		\begin{equation}\label{est var G}
			\begin{split}
				&\hspace{-2em}\abs{\var{\opg_\pm}(\ka, \pd_t\ka, \vom_{*\pm})}_{\LL\qty[\H{\sigma-1}\times \H{\sigma-2}\times H^{\sigma-\frac{1}{2}}(\Om_*^\pm); \LL\qty(\H{\sigma-1}; \H{\sigma}) ]} \\
				\le\, &Q\qty(\abs{\pd_t\ka}_{\H{\kk-\frac{5}{2}}}, \norm{\vom_{*\pm}}_{H^{\kk-1}(\Om_*^\pm)}),
			\end{split}
		\end{equation}
		where $ Q $ is a generic polynomial depending only on $ \Lambda_* $.
	\end{lem}
	
	\subsection{Evolution of the mean curvature}\label{sectoin evolution kappa}
	Suppose that $ (\Gt, \vv, \vh) $ is a solution to the interface problem (MHD)-(BC) for $ 0 \le t \le T $ with $ \Gt \in \Lambda_* $, $ \Gt \in H^{\kk + 1} $ and $ \vv(t), \vh(t) \in H^{\kk}(\Om \setminus \Gt) $. The hypersurface $ \Gt $ is uniquely determined by the function $ \ka(t) : \Gs \to \R $, whose leading order term is $ \kappa_+ $. Then, it is natural to consider the evolution equation for $ \kappa_+ $ under a weighted velocity $ \vbu_\lambda $. Plugging $\vbu_\lambda$ into \eqref{eqn Dt2 kappa alt} yields
	\begin{equation}\label{eqn dt2 k+}
		\begin{split}
			\Dtl^{\!2} \kappa_+ = &-\lap_\Gt \qty(\Dtl \vbu_\lambda \vdot \vn_+) + \Dtl\vbu_\lambda \vdot \lap_\Gt \vn_+ + 2\vn_+ \vdot (\grad\vbu_\lambda) \vdot (\lap_\Gt \vbu_\lambda)^\top \\
			&+4 \ip{\vb{A}_\lambda}{\vn_+ \vdot (\DD_\Gt)^2\vbu_\lambda} - \kappa_+ \abs{(\grad\vbu_\lambda)^*\vdot\vn_+}^2 \\
			&- 2\ip{\II_+}{\DD_\Gt\vbu_\lambda \vdot \DD_\Gt\vbu_\lambda}
			+4 \ip{\II_+\vdot\vb{A}_\lambda + \vb{A}_\lambda\vdot\II_+}{\vb{A}_\lambda},
		\end{split}
	\end{equation}
	where \begin{equation*}
		\vb{A}_\lambda \coloneqq \dfrac{1}{2}\qty{(\DD\vbu_\lambda)^\top + \qty[(\DD\vbu_\lambda)^\top]^*}.
	\end{equation*}
	Denoting by $ Q_\lambda $ a generic polynomial determined only by $ \Lambda_*$ and $ \lambda $, one gets from \eqref{est Dtl ulam}, the product estimates, and Lemma \ref{est ii} that
	\begin{equation}
			\begin{split}
			&\abs{\Dtl^{\!2}\kappa_+ + \lap_\Gt \qty(\Dtl \vbu_\lambda \vdot \vn_+)}_{H^{\kk-\frac{5}{2}}(\Gt)} \\
			&\quad \le \begin{cases*}
			Q_\lambda \qty(\abs{\ka}_{\H{\kk-1}}, \norm{(\vv, \vh)}_{H^{\kk}(\Om\setminus\Gt)}) \qif k = 2, \\ Q_\lambda \qty(\alpha\abs{\ka}_{\H{\kk-1}}, \abs{\ka}_{\H{\kk-\frac{3}{2}}}, \norm{(\vv, \vh)}_{H^{\kk}(\Om\setminus\Gt)}) \qif k \ge 3. 
			\end{cases*} 
			\end{split}
	\end{equation}
	Since $ \vw \equiv \llbracket \vv \rrbracket \coloneqq \vv_+ - \vv_- \in \mrm{T}\Gt $, it is routine to calculate
	\begin{equation}
		\begin{split}
			\vn_+ \vdot \Dtl \vbu_\lambda =\, & \dfrac{1}{\rho_-}\DD_{\vn_-}p^- - \lambda \qty(\dfrac{1}{\rho_+}\DD_{\vn_+}p^+ + \dfrac{1}{\rho_-}\DD_{\vn_-}p^-) \\
			&+ \vn_+ \vdot \qty[\lambda \DD_{\vh_+}\vh_+ + (1-\lambda) \DD_{\vh_-} \vh_- - \lambda(1-\lambda) \DD_\vw \vw] \\
			= &-\alpha^2\wt{\n}\kappa_+  + \qty(\dfrac{\lambda}{\rho_+}\n_+ - \dfrac{1-\lambda}{\rho_-}\n_-)\bn^{-1}\qty(g^+-g^-) \\
			&-\DD_{\vn_+}\qty[\lambda\qty(p_{\vv,\vv}^+ - p_{\vh,\vh}^+) + (1-\lambda)\qty(p_{\vv,\vv}^- - p_{\vh,\vh}^-) ] \\
			&-\lambda \II_+ \qty(\vh_+, \vh_+) - (1-\lambda)\II_+(\vh_-, \vh_-) + \lambda(1-\lambda)\II_+(\vw, \vw).			
		\end{split}
	\end{equation}
	In order to control the $ H^{\kk-\frac{1}{2}} $ norm of $ \qty(\dfrac{\lambda}{\rho_+}\n_+ - \dfrac{1-\lambda}{\rho_-}\n_-)\bn^{-1}\qty(g^+-g^-) $,
	it suffices to have 
	\begin{equation}\label{est lam n+ - n-}
		\abs{\qty(\dfrac{\lambda}{\rho_+}\n_+ - \dfrac{1-\lambda}{\rho_-}\n_-)}_{\LL\qty(H^s(\Gt); H^s(\Gt))} \le Q \qty(\abs{\ka}_{\H{\kk-\frac{3}{2}}})
	\end{equation}
	for $ \frac{1}{2} - \kk \le s \le \kk - \frac{1}{2} $ and some generic polynomial $ Q $ determined by $ \Lambda_* $. Thanks to Lemma~\ref{lem lap-n}, \eqref{est lam n+ - n-} holds as long as $ \lambda $ satisfies
	\begin{equation}
		\dfrac{\lambda}{\rho_+} = \dfrac{1-\lambda}{\rho_-} \Longleftrightarrow \lambda = \dfrac{\rho_+}{\rho_+ + \rho_-}.
	\end{equation}
	
	Denote by
	\begin{equation}\label{def vbu}
		\vbu \coloneqq \dfrac{\rho_+}{\rho_+ + \rho_-} \vv_+ + \dfrac{\rho_-}{\rho_+ + \rho_-} \vv_-, \qand \Dtb \coloneqq \pd_t + \DD_\vbu.
	\end{equation}
	Then $ \vbu \vdot \vn_+ = \theta $ and
	\begin{equation*}
		\begin{split}
			\vn_+ \vdot \Dtb \vbu =\, &- \alpha^2\wt{\n}\kappa_+ + \dfrac{\rho_+ \rho_-}{\qty(\rho_+ + \rho_-)^{2}} \II_{+}(\vw, \vw) \\
			&- \dfrac{\rho_+}{\rho_+ + \rho_-}\II_+ \qty(\vh_+, \vh_+) - \dfrac{\rho_-}{\rho_+ + \rho_-} \II_+\qty(\vh_-, \vh_-) + \mathfrak{r}_0,
		\end{split}
	\end{equation*}
	where
	\begin{equation}\label{def frak[r_0]}
		\begin{split}
			\mathfrak{r}_0 \coloneqq \dfrac{1}{\rho_+ \rho_-}\qty(\n_+ - \n_-)\bn^{-1}\qty(g^+ - g^-) - \DD_{\vn_+}\qty[\dfrac{\rho_+}{\rho_+ + \rho_-}\qty(p_{\vv,\vv}^+ - p_{\vh,\vh}^+) + \dfrac{\rho_-}{\rho_+ + \rho_-}\qty(p_{\vv,\vv}^- - p_{\vh,\vh}^-)]
		\end{split}
	\end{equation}
	satisfies
	\begin{equation*}
		\abs{\mathfrak{r}_0}_{H^{\kk-\frac{1}{2}}(\Gt)} \le Q \qty(\abs{\ka}_{\H{\kk-\frac{3}{2}}}, \norm{\vv}_{H^{\kk}(\Om\setminus\Gt)}, \norm{\vh}_{H^{\kk} (\Om\setminus\Gt)}).
	\end{equation*}
	
	Define the following two operators:
	\begin{equation}\label{def op a}
		\opa(\Gamma) \coloneqq - \lap_{\Gamma} \wt{\n}
	\end{equation}
	and 
	\begin{equation}\label{def op r}
		\opr\qty(\Gamma, \vb{J}) \coloneqq \qty(\DD_\Gamma)_{\vb{J}} \qty(\DD_\Gamma)_{\vb{J}}
	\end{equation}
	for a vector field $ \vb{J} \in \mrm{T}\Gamma $.
	It follows from Lemma \ref{Dt comm est lemma} that
	\begin{equation}\label{comm lap ii R}
		\abs{\lap_\Gt\qty[\II_+(\vb{J}, \vb{J})] - \opr(\Gt, \vb{J})\kappa_+}_{H^{\kk-\frac{5}{2}}(\Gt)} \le  Q\qty(\abs{\ka}_{\H{\kk-\frac{3}{2}}}) \abs{\vb{J}}^2_{H^{\kk-\frac{1}{2}}(\Gt)},
	\end{equation}
	where $ \vJ \in \mrm{T}\Gt $ is an $ H^{\kk-\frac{1}{2}} $ tangential vector field and $ Q $ is a generic polynomial determined by $ \Lambda_* $.
	
	Thus, by using the following notations:
	\begin{equation}\label{eqn mathfrak[R]_0}
		\begin{split}
			\mathfrak{R}_0 \coloneqq\, &\va{\mathfrak{a}} \vdot \lap_\Gt\vn_+ + 2 \vn_+ \vdot (\grad\vbu) \vdot (\lap_\Gt \vbu)^\top + 4\ip{\vb{A}}{\vn_+ \vdot (\DD_\Gt)^2\vbu}\\
			&-\kappa_+ \abs{(\grad\vbu)^*\vdot\vn_+}^2 - 2\ip{\II_+}{\DD_\Gt\vbu \vdot \DD_\Gt\vbu} + 4\ip{\II_+ \vdot \vb{A} + \vb{A} \vdot \II_+}{\vb{A}} \\
			&+ \dfrac{\rho_+\rho_-}{(\rho_+ + \rho_-)^2}\qty{\opr(\Gt, \vw)\kappa_+ - \lap_{\Gt}[\II_+(\vw, \vw)]} \\
			&+ \dfrac{\rho_+}{\rho_+ + \rho_-}\qty{\lap_\Gt\qty[\II_+(\vh_+, \vh_+)] - \opr(\Gt, \vh_+)\kappa_+} \\
			&+ \dfrac{\rho_-}{\rho_+ + \rho_-}\qty{\lap_\Gt\qty[\II_+(\vh_-, \vh_-)] - \opr(\Gt, \vh_-)\kappa_+} \\
			&- \lap_\Gt\mathfrak{r}_0,
		\end{split}
	\end{equation}
	with
	\begin{equation}\label{eqn va a}
		\begin{split}
			\va{\mathfrak{a}} \coloneqq \dfrac{1}{\rho_+ + \rho_-} \qty(\grad p^+ + \grad p^-) + \dfrac{\rho_+ \rho_-}{\qty(\rho_+ + \rho_-)^2}\DD_{\vw}\vw - \dfrac{\rho_+}{\rho_+ + \rho_-}\DD_{\vh_+}\vh_+ - \dfrac{\rho_-}{\rho_+ + \rho_-}\DD_{\vh_-}\vh_-,
		\end{split}
	\end{equation}
	and $ p $ given by (\ref{decomp p}), (\ref{def g pm}) and \eqref{def frak p}, one can obtain the following lemma:
	\begin{lem}
		There exists a generic polynomial $ Q $ depending only on $ \Lambda_* $, so that for any solution to (MHD)-(BC) for $ t \in [0, T] $, $ k \ge 2 $, $ \Gt \in \Lambda_* \cap H^{\kk + 1} $ and $ \vv, \vh \in H^{\kk}(\Om \setminus \Gt) $, the mean curvature $ \kappa_+ $ satisfies the equation
		\begin{equation}\label{eqn Dt2 kappa+}
			\begin{split}
				\Dtb^2 \kappa_+ &+ \alpha^2 \opa(\Gt)\kappa_+ + \dfrac{\rho_+\rho_-}{\qty(\rho_+ + \rho_-)^2}\opr(\Gt, \vw)\kappa_+ \\ &-\dfrac{\rho_+}{\rho_+ + \rho_-}\opr(\Gt, \vh_+)\kappa_+ - \dfrac{\rho_-}{\rho_+ + \rho_-}\opr(\Gt, \vh_-)\kappa_+ = \mathfrak{R}_0,
			\end{split}
		\end{equation}
		where $ \Dtb $ is defined in \eqref{def vbu}, and $ \mathfrak{R}_0 : \Gt \to \R$ satisfies
		\begin{equation}\label{est frak R_0}
			\abs{\mathfrak{R}_0}_{H^{\kk-\frac{5}{2}}(\Gt)} \le Q \qty(\abs{\ka}_{\H{\kk-1}}, \norm{\vv}_{H^{\kk}(\Om\setminus\Gt)}, \norm{\vh}_{H^{\kk} (\Om\setminus\Gt)}). 
		\end{equation} 
		Furthermore, if $ k \ge 3 $, the estimate of $ \mathfrak{R}_0 $ can be refined to
		\begin{equation}\label{est frak R_0'}
			\abs{\mathfrak{R}_0}_{H^{\kk-\frac{5}{2}}(\Gt)} \le Q \qty(\alpha\abs{\ka}_{\H{\kk-1}}, \abs{\ka}_{\H{\kk-\frac{3}{2}}} \norm{\vv}_{H^{\kk}(\Om\setminus\Gt)}, \norm{\vh}_{H^{\kk} (\Om\setminus\Gt)}). \tag{\ref{est frak R_0}'}
		\end{equation}
	\end{lem}
	
	\subsection{Evolution of the modified mean curvature}\label{sec evo ka}
	
	In order to compute the evolution equation for $ \ka $, it is convenient to pull back $ \opa $ and $ \opr $ to $ \Gs $. More precisely, for $ \Gamma \in \Lambda_* $, $ \vb{J}_* \in \mrm{T}\Gs $, and $ f : \Gs \to \R $, define:
	\begin{equation}\label{def opA}
		\opA(\ka)f \coloneqq \qty[\opa(\Gamma)(f \circ \Phi_\Gamma^{-1})]\circ\Phi_\Gamma,
	\end{equation}
	and
	\begin{equation}\label{def opR}
		\opR(\ka, \vb{J}_*)f \coloneqq \qty[\opr(\Gamma, \vb{J})(f\circ\Phi_\Gamma^{-1})]\circ\Phi_\Gamma,
	\end{equation}
	where $ \vb{J} \coloneqq \mrm{T}\Phi_\Gamma (\vb{J}_*)  = \qty[(\DD\Phi_\Gamma) \vdot \vb{J}_*] \circ (\Phi_\Gamma)^{-1} \in \mrm{T}\Gamma$.
	Furthermore, the following lemma holds (c.f. \cite{Shatah-Zeng2011}):
	\begin{lem}\label{lem 3.3}
		There are positive constants $ C_*, \delta_1 $ depending on $ \Lambda_* $, so that for $ \ka \in B_{\delta_1} \subset \H{\kk-\frac{5}{2}} $, $ \vb{J}_* \in \H{\kk-\frac{1}{2}} $, $ 2 \le s \le \kk - \frac{1}{2} $, $ 1 \le \sigma \le \kk-\frac{1}{2} $, and $ 2 \le s_1 \le \kk - \frac{1}{2} $, the following estimates hold:
		\begin{equation}\label{est opA}
			\abs{\opA(\ka)}_{\LL\qty[\H{s}; \H{s-3}]} \le C_*,
		\end{equation}
		\begin{equation}\label{est opR}
			\abs{\opR(\ka, \vJ_*)}_{\LL\qty[\H{\sigma}; \H{\sigma-2}]} \le C_* \abs{\vJ_*}_{\H{\kk-\frac{1}{2}}}^2,
		\end{equation}
		and
		\begin{equation}\label{est var opA}
			\abs{\var{\opA}(\ka)}_{\LL\qty[\H{\kk-\frac{5}{2}}; \LL\qty(\H{s_1}; \H{s_1-3})]} \le C_*.
		\end{equation}
		Furthermore, if $ k \ge 3 $, it holds for $ 2 \le s_2 \le \kk-1 $ that
		\begin{equation}\label{est var opR}
			\begin{split}
				\abs{\var{\opR}(\ka, \vJ_*)}_{\LL\qty[\H{\kk-\frac{5}{2}}\times\H{\kk-2}; \LL\qty(\H{s_2}; \H{s_2-2})]} 
				\le C_*\qty(1 +\abs{\vJ_*}^2_{\H{\kk-\frac{1}{2}}}).
			\end{split}
		\end{equation}
	\end{lem}
	\begin{proof}
		\eqref{est opA} and \eqref{est opR} follow from the definitions. As for the variational estimates, suppose that $ \{\Gt\}_{t} \subset \Lambda_* $ is a family of hypersurfaces parameterized by $ t $. Then by considering the velocity of the moving hypersurface:
		\begin{equation*}
			\vv \coloneqq \qty(\pd_t\gt \nu) \circ \X_\Gt^{-1}
		\end{equation*}
		and the material derivative
		\begin{equation*}
			\Dt \coloneqq \pd_t + \DD_{\vv},
		\end{equation*}
		one gets
		\begin{equation*}
			\begin{split}
				\pdv{t} \opA(\ka)f = \Dt \qty[-\lap_{\Gt}\wt{\n}(f\circ\Phi_\Gt^{-1})]\circ\Phi_\Gt = -\qty{\comm{\Dt}{\lap_\Gt\wt{\n}}(f\circ\Phi_\Gt^{-1})}\circ\Phi_\Gt.
			\end{split}
		\end{equation*}
		Thus \eqref{est var opA} follows from Lemma \ref{Dt comm est lemma}.
		
		As for \eqref{est var opR}, one may observe that
		\begin{equation*}
			\pdv{t}\opR(\ka,\vJ_*)f = \Dt \qty[\DD_{\vJ}\DD_\vJ (f\circ\Phi_\Gt^{-1})]\circ\Phi_\Gt
		\end{equation*}
		and for $ \phi : \Gt \to \R $,
		\begin{equation*}
			\comm{\Dt}{\DD_\vJ}\phi = \DD_{(\Dt\vJ - \DD_\vJ \vv)}\phi.
		\end{equation*}
		Hence for $ 0 \le s' \le \kk -2 $ (here $ k \ge 3 $), it holds that
		\begin{equation*}
			\begin{split}
				\abs{\comm{\Dt}{\DD_\vJ}\phi}_{H^{s'}(\Gt)} &\less \abs{\DD\phi \cdot (\Dt\vJ - \DD_\vJ \vv)}_{H^{s'}(\Gt)} \\
				&\less \abs{\phi}_{H^{s'+1}(\Gt)} \cdot \abs{\Dt\vJ - \DD_\vJ\vv}_{H^{\kk-2}(\Gt)},
			\end{split}
		\end{equation*}
		which, together with Lemma \ref{Dt comm est lemma}, implies (\ref{est var opR}).
	\end{proof}
	
	Next, we shall derive the evolution equation for $ \ka $. First, define a vector field $ \vW : \Gt \to \R^3 $ by
	\begin{equation}\label{def vW}
		\begin{split}
			\vW &\coloneqq \Dtb \vbu + \dfrac{1}{\rho_+ + \rho_-} \qty(\grad q^+ + \grad q^-) + \dfrac{\rho_+ \rho_-}{\qty(\rho_+ + \rho_-)^2}\DD_{\vw}\vw  - \dfrac{\rho_+}{\rho_+ + \rho_-}\DD_{\vh_+}\vh_+ - \dfrac{\rho_-}{\rho_+ + \rho_-}\DD_{\vh_-}\vh_- \\
			&\equiv \Dtb \vbu + \va{\mathfrak{b}},
		\end{split}
	\end{equation}
	for which
	\begin{gather*}
		q^\pm \coloneqq \rho_\pm \qty( p^\pm_{\vv, \vv} + p^\pm_{\vh, \vh}) \pm \alpha^2 \dfrac{1}{\rho_\mp}\h_\pm\bar{\n}^{-1}\n_\mp\kappa_+ + \h_\pm \bar{\n}^{-1}\mathfrak{q}, \\
		\mathfrak{q} \coloneqq -g^+ + g^- - \frac{1}{\abs{\Gt}}\int_{\OGt} \qty(\div \vv)^2 \dd{x}, \\
		g^{\pm} \coloneqq 2 \DD_{\vv_\pm^\top}\theta - \II_+\qty(\vv_\pm^\top, \vv_\pm^\top) + \II_+\qty(\vh_\pm^\top, \vh_\pm^\top) + \DD_{\vn_+}\qty(p_{\vv, \vv}^\pm - p_{\vh, \vh}^\pm).
	\end{gather*}
	Thus, $ \vW \equiv 0 $ if $ \qty(\Gt, \vv, \vh) $ is a solution to (MHD)-(BC). 
	
	Substituting
	\begin{equation*}
		\begin{split}
			\Dtb \vbu = \vW  - \dfrac{1}{\rho_+ + \rho_-} \qty(\grad q^+ + \grad q^-) - \dfrac{\rho_+ \rho_-}{\qty(\rho_+ + \rho_-)^2}\DD_{\vw}\vw + \dfrac{\rho_+}{\rho_+ + \rho_-}\DD_{\vh_+}\vh_+ + \dfrac{\rho_-}{\rho_+ + \rho_-}\DD_{\vh_-}\vh_-
		\end{split}
	\end{equation*}
	into (\ref{eqn dt2 k+}) with $ \lambda = \rho_+/(\rho_+ + \rho_-) $, and pulling back it to $ \Gs $ via (\ref{pull back Dt to Gs}), one has
	\begin{equation}\label{eqn dt2 ka}
		\begin{split}
			&\Dts^{\!2}(\kappa_+ \circ \Phi_\Gt) + \alpha^2\opA(\ka)(\kappa_+\circ\Phi_\Gt) + \dfrac{\rho_+\rho_-}{\qty(\rho_+ + \rho_-)^2}\opR(\ka, \vw_*)(\kappa_+ \circ \Phi_\Gt)\\ &-\dfrac{\rho_+}{\rho_+ + \rho_-}\opR(\ka, \vh_{+*})(\kappa_+\circ\Phi_\Gt)
			- \dfrac{\rho_-}{\rho_+ + \rho_-}\opR(\ka, \vh_{-*})(\kappa_+ \circ \Phi_\Gt) \\ 
			&\quad=\qty{\mathfrak{R}_1  - \lap_\Gt \qty(\vW \vdot \vn_+) + \vW \vdot \lap_\Gt \vn_+} \circ\Phi_\Gt,
		\end{split}
	\end{equation}
	where
	\begin{equation*}
		\vbu_* = \dfrac{\rho_+}{\rho_+ + \rho_-}\vv_{+*} + \dfrac{\rho_-}{\rho_+ + \rho_-}\vv_{-*} \qc \Dts \coloneqq \pd_t + \DD_{\vbu_*},
	\end{equation*}
	\begin{equation*}
		\vh_{\pm*} \coloneqq (\DD\Phi_\Gt)^{-1}\vdot(\vh_\pm \circ \Phi_\Gt) \qc
		\vw_* \coloneqq (\DD\Phi_\Gt)^{-1}\vdot(\vw\circ\Phi_\Gt) = \vv_{+*} - \vv_{-*},
	\end{equation*}
	and
	\begin{equation}\label{def frak[R]_1}
		\begin{split}
			\mathfrak{R}_1 \coloneqq\, &-\va{\mathfrak{b}} \vdot \lap_\Gt\vn_+ + 2 \vn_+ \vdot (\grad\vbu) \vdot (\lap_\Gt \vbu)^\top + 4\ip{\vb{A}}{\vn_+ \vdot (\DD_\Gt)^2\vbu}\\
			&-\kappa_+ \abs{(\grad\vbu)^*\vdot\vn_+}^2 - 2\ip{\II_+}{\DD_\Gt\vbu \vdot \DD_\Gt\vbu} + 4\ip{\II_+ \vdot \vb{A} + \vb{A} \vdot \II_+}{\vb{A}} \\
			&+ \dfrac{\rho_+\rho_-}{(\rho_+ + \rho_-)^2}\qty{\opr(\Gt, \vw)\kappa_+ - \lap_{\Gt}[\II_+(\vw, \vw)]} \\
			&+ \dfrac{\rho_+}{\rho_+ + \rho_-}\qty{\lap_\Gt\qty[\II_+(\vh_+, \vh_+)] - \opr(\Gt, \vh_+)\kappa_+} \\
			&+ \dfrac{\rho_-}{\rho_+ + \rho_-}\qty{\lap_\Gt\qty[\II_+(\vh_-, \vh_-)] - \opr(\Gt, \vh_-)\kappa_+} \\
			&- \lap_\Gt\mathfrak{r}_0,
		\end{split}
	\end{equation}
	with $ \mathfrak{r}_0 $ given by (\ref{def frak[r_0]}).
	
	Due to the relation
	\begin{equation*}
		\Dts^{\!2} = \pd^2_{tt} + \DD_{\vbu_*}\DD_{\vbu_*} + 2 \DD_{\vbu_*}\pd_t + \DD_{\pd_t \vbu_*}
	\end{equation*}
	and (\ref{eqn pd beta v}), the term $ \pd_t \vbu_* $ involves $ \pd^2_{tt} \ka $, so (\ref{eqn dt2 ka}) is a nonlinear equation for $ \pd^2_{tt}\ka $. In order to get a equation which is linear for $ \pd^2_{tt}\ka $, one may drive from (\ref{eqn pd beta v}) that
	\begin{equation}\label{eqn pdt2 kappa+}
		\begin{split}
			&\pd^2_{tt}(\kappa_+ \circ \Phi_\Gt) + \opC_\alpha(\ka, \pd_t\ka, \vv_*, \vh_*)(\kappa_+ \circ \Phi_\Gt) \\
			&+ \grad^\top(\kappa_+ \circ \Phi_\Gt) \vdot \qty[\opb(\ka)\pd^2_{tt}\ka + \opf(\ka)\pd_t\vom_*+\opg(\ka, \pd_t\ka, \vom_*)\pd_t\ka] \\
			&\quad =\qty{\mathfrak{R}_1  - \lap_\Gt \qty(\vW \vdot \vn_+) + \vW \vdot \lap_\Gt \vn_+} \circ\Phi_\Gt,
		\end{split}
	\end{equation}
	where the following notations have been used:
	\begin{equation}
		\opb (\text{resp. } \opf \text{ or } \opg) \coloneqq \dfrac{\rho_+}{\rho_+ + \rho_-}\opb_+ (\text{resp. } \opf_+ \text{ or } \opg_+) + \dfrac{\rho_-}{\rho_+ + \rho_-}\opb_- (\text{resp. } \opf_- \text{ or } \opg_-),
	\end{equation}
	and for simplicity,
	\begin{equation}\label{def opC}
		\begin{split}
			\opC_\alpha(\ka, \pd_t\ka, \vv_*, \vh_*) \coloneqq\, &2\DD_{\vbu_*}\pd_t + \DD_{\vbu_*}\DD_{\vbu_*} +  \alpha^2\opA(\ka) + \dfrac{\rho_+\rho_-}{\qty(\rho_+ + \rho_-)^2}\opR(\ka, \vw_*) \\ &-\dfrac{\rho_+}{\rho_+ + \rho_-}\opR(\ka, \vh_{+*}) - \dfrac{\rho_-}{\rho_+ + \rho_-}\opR(\ka, \vh_{-*}).
		\end{split}
	\end{equation}
	
	Since $ \ka = \kappa_+ \circ \Phi_\Gt + a^2 \gt $, one also needs to calculate $ \pd^2_{tt}\gt $. Notice that for every evolution velocity $ \vv : \Gt \to \R^3 $ of $ \Gt $, it holds that
	\begin{equation*}
		\pd_t \gt \vnu \vdot (\vn_+ \circ \Phi_\Gt) = (\vv \vdot \vn_+) \circ \Phi_\Gt,
	\end{equation*}
	which, together with (\ref{eqn dt n}), implies
	\begin{equation*}
		\begin{split}
			(\pd^2_{tt}\gt\vnu)\circ (\Phi_\Gt)^{-1} \vdot \vn_+ =\, &\vn_+ \vdot \DD_{\qty[(\pd_t\gt \vnu)\circ(\Phi_\Gt)^{-1} - \vbu]}\qty[(\pd_t\gt\vnu)\circ(\Phi_\Gt)^{-1}]\\
			&+ \vn_+ \vdot \qty[\Dtb \vbu - \DD_{\vbu}\vbu + \DD_{(\pd_t\gt\vnu)\circ(\Phi_\Gt)^{-1}}\vbu],
		\end{split}
	\end{equation*}
	that is,
	\begin{equation}\label{eqn pd2 tt gt}		
		\begin{split}
			\pd^2_{tt}\gt =\, &\dfrac{(\vn_+ \circ \Phi_\Gt)}{\vnu \vdot (\vn_+ \circ \Phi_\Gt)} \vdot \qty[\qty(\Dtb\vbu)\circ\Phi_\Gt - \DD_{\vbu_*}\qty(\vbu\circ\Phi_\Gt + \pd_t\gt\vnu) ] \\
			=\, &\dfrac{(\vn_+ \circ \Phi_\Gt)}{\vnu \vdot (\vn_+ \circ \Phi_\Gt)} \vdot \qty[\qty(\vW - \va{\mathfrak{b}})\circ\Phi_\Gt - \DD_{\vbu_*}\qty(\vbu\circ\Phi_\Gt + \pd_t\gt\vnu) ].
		\end{split}
	\end{equation}
	In particular, $ \pd^2_{tt}\gt $ does not involve the term $ \pd^2_{tt}\ka $.
	
	Combining (\ref{def ka}) and (\ref{eqn pdt2 kappa+}) yields
	\begin{equation}\label{eqn pdt2 ka premitive}
		\begin{split}
			&\qty[\mrm{I} + \grad^\top(\kappa_+ \circ \Phi_\Gt) \vdot \opb(\ka) ]\pd^2_{tt}\ka + \opC_\alpha(\ka, \pd_t\ka, \vv_*, \vh_*) \ka \\ &+ \grad^\top(\kappa_+ \circ \Phi_\Gt) \vdot \qty[\opf(\ka)\pd_t\vom_* + \opg(\ka, \pd_t\ka, \vom_*)\pd_t\ka] \\
			&+ a^2 \dfrac{\vn_+ \circ \Phi_\Gt}{\vnu \vdot (\vn_+ \circ \Phi_\Gt)} \vdot \qty[\va{\mathfrak{b}} \circ \Phi_\Gt + \DD_{\vbu_*}\qty(\vbu\circ\Phi_\Gt + \pd_t\gt\vnu)]-a^2 \opC_\alpha(\ka, \pd_t\ka, \vv_*, \vh_*)\gt \\
			&\quad = \mathfrak{R}_1 \circ \Phi_\Gt + \qty[-\lap_\Gt \qty(\vW \vdot \vn_+) + \vW \vdot \lap_\Gt \vn_+ + a^2 \dfrac{\vW \vdot \vn_+}{\vn_+ \vdot (\vnu \circ \Phi_\Gt^{-1})}] \circ \Phi_\Gt.
		\end{split}
	\end{equation}
	Define a new operator:
	\begin{equation}\label{def opB}
		\opB(\ka) \coloneqq \grad^\top(\kappa_+ \circ \Phi_\Gt) \vdot \opb(\ka).
	\end{equation}
	It then follows from Lemma \ref{lem3.1} that
	\begin{equation}\label{est opB}
		\abs{\opB(\ka)}_{\LL\qty[\H{s''}; \H{s'}]} \less a^{s'-s''-2+\epsilon} \abs{\ka}_{\H{\kk-1}},
	\end{equation}
	for $ s'-2 \le s'' \le s' \le \kk - 2 $, $ s' \ge \frac{1}{2} $, and $ 0 < \epsilon \le s'' -s' +2 $. If $ k \ge 3 $, one may take $ \epsilon = 0 $, and it holds for $ \sigma'-2\le \sigma'' \le \sigma' \le \kk-\frac{5}{2} $, $ \sigma' \ge \frac{1}{2} $ that
	\begin{equation}\label{est opB'}
		\abs{\opB(\ka)}_{\LL\qty[\H{\sigma''}; \H{\sigma'}]} \less a^{\sigma'-\sigma''-2} \abs{\ka}_{\H{\kk-\frac{3}{2}}}. \tag{\ref{est opB}'}
	\end{equation}
	
	Letting $ s'=s'' $, $ 0 < \epsilon < \frac{1}{2} $ ($ \epsilon = 0 $ if $ k \ge 3 $) and $ a_0 $ large enough compared to $ \abs{\ka}_{\H{\kk-1}} $ (or $ a_0 $ large compared to $ \abs{\ka}_{\H{\kk-\frac{3}{2}}} $ if $ k \ge 3 $), one has
	\begin{equation}\label{est opcal b}
		\abs{\opB(\ka)}_{\LL(\H{s'})} \le \frac{1}{2} < 1,
	\end{equation}
	for $ \frac{1}{2} \le s' \le \kk-2 $ (or $ \frac{1}{2} \le s' \le \kk - \frac{5}{2} $ if $ k \ge 3 $). Namely, $ \qty[\mathrm{I}+\opB(\ka)] $ is an isomorphism on $ \H{s'} $. 
	Set
	\begin{equation}
		\vj \coloneqq \curl \vh \qand \vj_* \coloneqq \vj \circ \X_\Gt.
	\end{equation}
	Then $ \vh $ can be recovered from $ (\ka, \vj_*) $ by solving div-curl problems. Applying $ \qty[\mathrm{I}+\opB(\ka)]^{-1} $ to (\ref{eqn pdt2 ka premitive}), one can get the evolution equation for $ \ka $ as (which is, in particular, irrelevant to the MHD system):
	\begin{equation}\label{eqn pd2 tt ka}
		\begin{split}
			&\pd^2_{tt}\ka + \opC_\alpha(\ka, \pd_t\ka, \vv_*, \vh_*) \ka  - \opF(\ka)\pd_t\vom_* - \opG(\ka, \pd_t\ka, \vom_*, \vj_*) \\
			&\quad = \qty[\mathrm{I} + \opB(\ka)]^{-1}\qty{\qty[-\lap_\Gt \qty(\vW \vdot \vn_+) + \vW \vdot \lap_\Gt \vn_+ + a^2 \dfrac{\vW \vdot \vn_+}{\vn_+ \vdot (\vnu \circ \Phi_\Gt^{-1})}] \circ \Phi_\Gt}.
		\end{split}
	\end{equation}
	The operators $ \opF $ and $ \opG $ defined above satisfy the following lemma, whose proof is given in the Appendix:
	\begin{lem}\label{lem 3.4}
		Assume that $ a \ge a_0 $ and $ \ka \in B_{\delta_1} $ as in Proposition \ref{prop K}. For $ \frac{1}{2} \le s \le \kk-2 $ and $ \epsilon > 0 $, there are some positive constants $ C_* $ and generic polynomials $ Q $ determined by $ \Lambda_* $, so that
		\begin{equation}\label{est opF}
			\abs{\opF(\ka)}_{\LL\qty[H^{s+\epsilon-\frac{1}{2}}(\OGs); \H{s}]} \le C_* \abs{\ka}_{\H{\kk-1}},
		\end{equation}
		and
		\begin{equation}\label{est opG}
			\begin{split}
				&\hspace{-2em}\abs{\opG(\ka, \pd_t\ka, \vom_*, \vj_*)}_{\H{\kk-\frac{5}{2}}} \\ \le\, &a^2 Q\qty(\abs{\ka}_{\H{\kk-1}}, \abs{\pd_t\ka}_{\H{\kk-\frac{5}{2}}}, \norm{\vom_*}_{H^{\kk-1}(\OGs)}, \norm{\vj_*}_{H^{\kk-1}(\OGs)}).
			\end{split}
		\end{equation}
		Furthermore, if $ k \ge 3 $, for $ \frac{1}{2} \le \sigma \le \kk-\frac{5}{2} $, there hold
		\begin{equation}\label{est opF'}
			\abs{\opF(\ka)}_{\LL\qty[H^{\sigma-\frac{1}{2}}(\OGs); \H{\sigma}]} \le C_* \abs{\ka}_{\H{\kk-\frac{3}{2}}},
		\end{equation}
		\begin{equation}\label{est var opF}
			\abs{\var\opF(\ka)}_{\LL\qty[\H{\kk-\frac{5}{2}}; \LL\qty(H^{\kk-4}(\OGs); \H{\kk-4})]} \le C_*,
		\end{equation}
		\begin{equation}\label{est opG'}
			\begin{split}
				&\hspace{-1em}\abs{\opG(\ka, \pd_t\ka, \vom_*, \vj_*)}_{\H{\kk-\frac{5}{2}}} \\ &\le a^2 Q\qty(\alpha\abs{\ka}_{\H{\kk-1}}, \abs{\ka}_{\H{\kk-\frac{3}{2}}}, \abs{\pd_t\ka}_{\H{\kk-\frac{5}{2}}}, \norm{(\vom_*, \vj_*)}_{H^{\kk-1}(\OGs)}),
			\end{split}
		\end{equation}
		and
		\begin{equation}\label{est var opG}
			\begin{split}
				&\hspace{-2em}\abs{\var\opG}_{\LL\qty[\H{\kk-\frac{5}{2}}\times\H{\kk-4}\times H^{\kk-\frac{5}{2}}(\OGs) \times H^{\kk-\frac{5}{2}}(\OGs); \H{\kk-4}]} \\
				\le\, &a^2 Q\qty(\abs{\pd_t\ka}_{\H{\kk-\frac{5}{2}}}, \norm{\vom_*}_{H^{\kk-1}(\OGs)}, \norm{\vj_*}_{H^{\kk-1}(\OGs)}).
			\end{split}
		\end{equation}
	\end{lem}
	
	\subsection{Evolution of the current and vorticity}\label{sec curr-vorticity eqn}
	We shall use the vorticity, current and the corresponding boundary conditions to recover the solenoidal vector fields $ \vv $ and $ \vh $ by solving the corresponding div-curl systems. Hence, it is necessary to consider the evolution of the vorticity and the current.
	
	Set
	\begin{equation*}
		\vom_\pm \coloneqq \curl \vv_\pm \qand
		\vj_\pm \coloneqq \curl \vh_\pm.
	\end{equation*}
	Then it follows from taking curl of the equations (\ref{eqn Dtv}) and (\ref{eqn Dth}) that
	\begin{equation}\label{eqn pdt vom}
		\pd_t \vom + (\vv \vdot \grad)\vom - (\vh \vdot \grad)\vj = (\vom \vdot \grad)\vv - (\vj \vdot \grad) \vh,
	\end{equation}
	and
	\begin{equation}\label{eqn pdt vj}
		\pd_t\vj + (\vv \vdot \grad)\vj - (\vh \vdot \grad) \vom = (\vj \vdot \grad)\vv - (\vom \vdot \grad) \vh - 2\tr(\grad \vv \cp \grad \vh),
	\end{equation}
	where in the Cartesian coordinate
	\begin{equation*}
		\tr(\grad \vv \cp \grad \vh ) = \sum_{l=1}^3 \grad v^l \cp \grad h^l.
	\end{equation*}

	\section{Linear Systems}\label{sec linear}
	In this section, we shall study the linearized systems deriving from (\ref{eqn pd2 tt ka}), (\ref{eqn pdt vom}) and (\ref{eqn pdt vj}). More precisely, the uniform estimates will be shown, from which the local well-posedness of the linear systems follows.
	
	\subsection{Linearized problem for the modified mean curvature}\label{sec linear ka}
	Suppose that $ \Gs \in H^{\kk+1} $ ($ k \ge 2 $) is a reference hypersurface, and $ \Lambda_* $, defined by (\ref{def lambda*}), satisfies all the properties given in the preliminary. Now, assume that there are a $ t $-parameterized family of hypersurfaces $ \Gt \in \Lambda_* $ and four tangential vector fields $ \vv_{\pm*}, \vh_{\pm*} : \Gs \to \mathrm{T}\Gs $ satisfying:
	\begin{equation}
		\ka \in C^0\qty{[0, T]; \H{\kk-1}} \cap C^1\qty{[0, T]; B_{\delta_1} \subset \H{\kk-\frac{5}{2}}}, \tag{H1}
	\end{equation}
	and
	\begin{equation}
		\vv_{\pm*}, \vh_{\pm*} \in C^0\qty{[0, T]; \H{\kk-\frac{1}{2}}} \cap C^1\qty{[0, T]; \H{\kk-2}}. \tag{H2}
	\end{equation}
	
	For the sake of convenience, suppose that there are positive constants $ L_0, L_1,  L_2 $ so that
	\begin{equation}
		\sup_{t\in[0, T]} \qty{\abs{\ka (t)}_{\H{\kk-1}},  \abs{\pd_t \ka(t)}_{\H{\kk-\frac{5}{2}}}, \abs{\qty(\vv_{\pm*}(t), \vh_{\pm*}(t))}_{\H{\kk-\frac{1}{2}}}} \le L_1,
	\end{equation}
	\begin{equation}
		\sup_{t\in [0, T]} \abs{\qty(\pd_t\vv_{\pm*}(t), \pd_t\vh_{\pm*}(t))}_{\H{\kk-2}} \le L_2,
	\end{equation}
	and
	\begin{equation}
		\abs{\qty(\vv_{+*}(0), \vv_{-*}(0) )}_{\H{\kk-2}} \le L_0.
	\end{equation}
	
	Using the following notations as in the previous section:
	\begin{equation*}
		\vw_* = \vv_{+*} - \vv_{-*}, \quad \vbu_{*} = \dfrac{\rho_+}{\rho_+ + \rho_-}\vv_{+*} + \dfrac{\rho_-}{\rho_+ + \rho_-}\vv_{-*},
	\end{equation*}
	we consider the linear initial value problem
	\begin{equation}\label{eqn linear 1}
		\begin{cases}
			\pd^2_{tt} \f + \opC(\ka, \pd_t\ka, \vv_*, \vh_*) \f = \g, \\
			\f(0) = \f_0, \quad \pd_t \f(0) = \f_1,
		\end{cases}
	\end{equation}
	where $ \f_0, \f_1, \g(t) : \Gs \to \R $ are three given functions, and $ \opC $ is given by:
	\begin{equation*}
		\begin{split}
			\opC(\ka, \pd_t\ka, \vv_*, \vh_*) \coloneqq\, &2\DD_{\vbu_*}\pd_t + \DD_{\vbu_*}\DD_{\vbu_*} +  \opA(\ka) + \dfrac{\rho_+\rho_-}{\qty(\rho_+ + \rho_-)^2}\opR(\ka, \vw_*) \\ &-\dfrac{\rho_+}{\rho_+ + \rho_-}\opR(\ka, \vh_{+*}) - \dfrac{\rho_-}{\rho_+ + \rho_-}\opR(\ka, \vh_{-*}),
		\end{split}
	\end{equation*}
	which is exactly \eqref{def opC} with $ \alpha = 1 $.

	We shall derive some uniform estimates, from which the uniqueness and continuous dependence on the initial data follow, while the existence can be obtained via the Galerkin approximations (or the standard semigroup theory as in \cite{Shatah-Zeng2011}).
	
	In order to derive the energy estimates for (\ref{eqn linear 1}), one notes first that $ \opA $ is the highest order spacial derivative term (3rd order). Then for an integer $ l \in [0, k-2] $, if one multiplies (\ref{eqn linear 1}) by 
	\begin{equation*}
		\det(\DD\Phi_\Gt) \cdot	\qty{\wn\qty[\qty(\opA^l \pd_t \f)\circ (\Phi_\Gt)^{-1}]}\circ\Phi_\Gt 
	\end{equation*}
	and integrates on $ \Gs $, the leading order terms will be obtained:
	\begin{equation*}
		\begin{split}
			\dfrac{1}{2}\dv{t} \int_\Gs &\pd_t \f \cdot \qty{\wn\qty[\qty(\opA^l \pd_t \f)\circ (\Phi_\Gt)^{-1}]}\circ\Phi_\Gt  \\
			&+ f \cdot \qty{\wn\qty[\qty(\opA^{1+l}  \f)\circ (\Phi_\Gt)^{-1}]}\circ\Phi_\Gt \cdot \det(\DD\Phi_\Gt) \dd{S_*}. 			
		\end{split}
	\end{equation*}
	The energy is defined as:
	\begin{equation}\label{eqn E_l}
		\begin{split}
			E_l(t, \f, \pd_t\f) \coloneqq \int_{\Gt}
			&\abs{\qty(-\wn^{\frac{1}{2}}\lap_\Gt \wn^{\frac{1}{2}})^{\frac{l}{2}} \wn^{\frac{1}{2}} \qty[\qty(\pd_t \f + \DD_{\vbu_*} \f) \circ \Phi_\Gt^{-1}] }^2 \\
			&+ \abs{\qty(-\wn^{\frac{1}{2}}\lap_\Gt  \wn^{\frac{1}{2}})^{\frac{l+1}{2}}\wn^{\frac{1}{2}} \qty(\f \circ \Phi_\Gt^{-1})}^2 \\
			&- \dfrac{\rho_+\rho_-}{\qty(\rho_+ + \rho_-)^2} \abs{\qty(-\wn^{\frac{1}{2}}\lap_\Gt  \wn^{\frac{1}{2}})^{\frac{l}{2}}\wn^{\frac{1}{2}} \qty[\qty(\DD_{\vw_*} \f) \circ\Phi_\Gt^{-1}]}^2 \\
			&+ \dfrac{\rho_+}{\rho_+ + \rho_-} \abs{\qty(-\wn^{\frac{1}{2}}\lap_\Gt \wn^{\frac{1}{2}})^{\frac{l}{2}}\wn^{\frac{1}{2}} \qty[\qty(\DD_{\vh_{+*}} \f) \circ\Phi_\Gt^{-1}] }^2 \\
			&+ \dfrac{\rho_-}{\rho_+ + \rho_-} \abs{\qty(-\wn^{\frac{1}{2}}\lap_\Gt \wn^{\frac{1}{2}})^{\frac{l}{2}}\wn^{\frac{1}{2}}  \qty[\qty(\DD_{\vh_{-*}} \f) \circ\Phi_\Gt^{-1}]}^2 \dd{S_t}.
		\end{split}
	\end{equation}
	\begin{lem}\label{lem est E_l}
		For any integer $ 0 \le  l \le k-2 $, and $ 0 \le t \le T $, it holds that
		\begin{equation}\label{est E_l}
			\begin{split}
				&\hspace{-2em}E_l (t, \f, \pd_t\f) - E_l (0, \f_0, \f_1) \\
				\le&Q(L_1, L_2)\int_0^t \qty(\abs{\f(s)}_{\H{\frac{3}{2}l + 2}} + \abs{\pd_t\f(s)}_{\H{\frac{3}{2}l + \frac{1}{2}}} + \abs{\g(s)}_{\H{\frac{3}{2}l + \frac{1}{2}}}) \times \\
				&\hspace{8em} \times \qty(\abs{\f(s)}_{\H{\frac{3}{2}l + 2}} + \abs{\pd_t\f(s)}_{\H{\frac{3}{2}l + \frac{1}{2}}}) \dd{s},
			\end{split}
		\end{equation}
		where $ Q $ is a generic polynomial determined by $ \Lambda_* $.
	\end{lem}
	\begin{proof}
		Denote by 
		\begin{equation*}
			\Dt \coloneqq \pd_t + \DD_\vmu, \qq{with} \vmu \coloneqq (\pd_t\gt \vnu) \circ \Phi_\Gt^{-1} : \Gt \to \R^3 ,
		\end{equation*} 
		and for any function $ f : \Gs \to \R $ and vector field $ \vb{a}_*: \Gs \to \mathrm{T}\Gs $
		\begin{equation*}
			\bar{f} \coloneqq f \circ \Phi_\Gt^{-1} : \Gt \to \R, \quad \vb{a} \coloneqq \qty(\DD\Phi_\Gt \cdot \vb{a}_*) \circ \Phi_\Gt^{-1} : \Gt \to \mathrm{T}\Gt.
		\end{equation*} 
		Thus
		\begin{equation*}
			(\pd_t\f) \circ \Phi_\Gt^{-1} = \Dt \bar{\f}, \quad \qty(\DD_{\vb{a}_*}\f) \circ\Phi_\Gt^{-1} = \DD_{\vb{a}} \bar{\f}.
		\end{equation*}
		For simplicity, we will use the conventions that
		\begin{equation*}
			\abs{f}_{s} \coloneqq \abs{f}_{H^s(\Gt)},
		\end{equation*}
		and
		\begin{equation*}
			\mathfrak{u} \lesssim_{L_j} \mathfrak{v}
		\end{equation*}
		if there exists a constant $ C = C(\Lambda_*, L_j) $ such that $ \mathfrak{u} \le C \mathfrak{v} $.
		
		In order to commute $ \Dt $ with the differential operators, one observes first the facts that
		\begin{equation}\label{comm est dt wtn}
			\abs{\int_\Gt g \comm{\Dt}{\wt{\n}}f \dd{S_t}} \less \abs{\vmu}_{\kk-\frac{1}{2}} \abs{f}_{\frac{1}{2}}\abs{g}_{\frac{1}{2}},
		\end{equation}
		and
		\begin{equation}\label{comm est Dt lapGt}
			\abs{\int_\Gt g \comm{\Dt}{\lap_\Gt}f \dd{S_t}} \less \abs{\vmu}_{\kk-\frac{1}{2}}\abs{f}_{1}\abs{g}_{1}.
		\end{equation}
		Indeed, it follows from the properties of $ \n_+ $ that
		\begin{equation*}
			\begin{split}
				\int_\Gt g \comm{\Dt}{\n_+}f \dd{S_t} &= \int_\Gt \n_+^{-1}\n_+g \comm{\Dt}{\n_+}f \dd{S_t} + \qty(\int_\Gt g)\qty(\int_\Gt \comm{\Dt}{\n_+}f ) \\
				&= \int_\Gt \n_+^{\frac{1}{2}}g \cdot \n_+^{-\frac{1}{2}}\comm{\Dt}{\n_+}f \dd{S_t} + \qty(\int_\Gt g)\qty(\int_\Gt \comm{\Dt}{\n_+}f ),
			\end{split}
		\end{equation*}
		where the last equality follows from the self-adjointness of $ \n_+ $. It follows from the following commutator formula (c.f. \cite[p. 710]{Shatah-Zeng2008-Geo} for the derivation):
		\begin{equation}\label{comm formula Dt n}			
			\begin{split}
				\comm{\Dt}{\n_+}f = \DD_{\vn_+} \lap^{-1}\qty(2\DD\vmu \tdot \DD^2\h_+ f + \grad\h_+ f \vdot \lap\vmu)  - \grad\h_+ f \vdot \DD_{\vn_+} \vv - \DD_{\grad^\top f}\vmu \vdot \vn_+,
			\end{split}
		\end{equation}
		that
		\begin{equation*}
			\begin{split}
				\int_\Gt \comm{\Dt}{\n_+}f \dd{S_t} &= \int_\Gt f \comm{\Dt}{\n_+}1 + \qty(1\n_+ f)(-\Div_\Gt \vmu) \dd{S_t} \\
				&= - \int_\Gt \n_+(f) \cdot \Div_\Gt \vmu \dd{S_t} \\
				&= - \int_\Gt f \cdot \n_+(\Div_\Gt\vmu) \dd{S_t}.
			\end{split}
		\end{equation*}
		Thus, the above two relations imply that
		\begin{equation}\label{comm est Dt op n+}
			\abs{\int_\Gt g \comm{\Dt}{\n_+}f \dd{S_t}} \less \abs{\vmu}_{\kk-\frac{1}{2}} \abs{f}_{\frac{1}{2}}\abs{g}_{\frac{1}{2}}.
		\end{equation}
		As $ \wt{\n} $ is defined via \eqref{def operator n tilde}, the estimate \eqref{comm est dt wtn} follows from \eqref{comm est Dt op n+} and Lemma \ref{Dt comm est lemma}.
		The relation \eqref{comm est Dt lapGt} can be derived via the following formula (c.f. \cite[p. 710]{Shatah-Zeng2008-Geo} for the derivation):
		\begin{equation}\label{comm formula Dt lapGt}
			\comm{\Dt}{\lap_\Gt} = - 2 (\DD_\Gt)^2f\vdot (\DD\vmu)^\top - \grad^\top f \vdot \lap_\Gt \vmu + \kappa_+ \DD_{\grad^\top f} \vmu \vdot \vn_+,
		\end{equation}
		and the integration-by-parts on $ \Gt $.
		
		Now, it follows from the self-adjointness of $ \wt{\n} $ and $ (-\lap_\Gt) $ that
		\begin{equation}\label{eqn lin est decomp}
			\begin{split}
				&\int_\Gt \qty(-\wt{\n}^{\frac{1}{2}}\lap_\Gt\wt{\n}^\frac{1}{2})^{\frac{l}{2}}\wt{\n}^{\frac{1}{2}}\qty(\Dt\baf + \DD_\vbu\baf) \cdot \qty(-\wt{\n}^{\frac{1}{2}}\lap_\Gt\wt{\n}^\frac{1}{2})^{\frac{l}{2}}\wt{\n}^{\frac{1}{2}}\qty(\Dt\baf + \DD_\vbu\baf) \dd{S_t} \\
				&\quad = \int_\Gt \wt{\n}^{\frac{1}{2}}\qty(\Dt\baf + \DD_\vbu\baf) \wt{\n}^{\frac{1}{2}}(-\lap_\Gt)\wt{\n}^{\frac{1}{2}}\qty(-\wt{\n}^{\frac{1}{2}}\lap_\Gt\wt{\n}^\frac{1}{2})^{l-1} \wt{\n}^{\frac{1}{2}}(\Dt\baf + \DD_\vbu\baf) \dd{S_t} \\
				&\quad = \int_\Gt \wt{\n} (\Dt\baf + \DD_\vbu\baf) \qty(-\lap_\Gt\wt{\n})^{l}(\Dt\baf + \DD_\vbu\baf) \dd{S_t}.
			\end{split}
		\end{equation}
		
		Observing that
		\begin{equation}
			\Dt \dd{S_t} = \qty(\Div_\Gt \vmu) \dd{S_t},
		\end{equation}
		then, if $ l = 0 $, one can calculate that
		\begin{equation}\label{eqn l = 0 case}
			\begin{split}
				&\hspace{-2em}\dv{t}\int_{\Gt} (\Dt\baf + \DD_\vbu \baf) \cdot \wt{\n} (\Dt\baf + \DD_\vbu\baf) \dd{S_t} \\
				=\, &\int_\Gt \qty(\Dt^2\baf + \Dt\DD_\vbu\baf) \cdot \wt{\n} (\Dt\baf + \DD_\vbu\baf) \dd{S_t} \\
				&+ \int_{\Gt} (\Dt\baf + \DD_\vbu\baf) \cdot \wt{\n}\qty(\Dt^2\baf + \Dt\DD_\vbu \baf) \dd{S_t} \\
				&+ \int_{\Gt} (\Dt\baf + \DD_\vbu\baf) \cdot \comm{\Dt}{\wt{\n}}(\Dt\baf + \DD_\vbu\baf) \dd{S_t} \\
				&+ \int_\Gt (\Dt\baf + \DD_\vbu \baf) \cdot \wt{\n} (\Dt\baf + \DD_\vbu\baf) \cdot (\Div_\Gt\vmu) \dd{S_t}.
			\end{split}
		\end{equation}
		It follows from the self-adjointness of $ \wt{\n} $ and \eqref{comm est dt wtn} that
		\begin{equation}
			\begin{split}
				&\abs{\frac{1}{2}\dv{t}(\int_\Gt \abs{\wt{\n}^{\frac{1}{2}}(\Dt\baf + \DD_\vbu \baf)}^2 \dd{S_t}) - \qty(\int_\Gt \qty(\Dt^2\baf + \Dt\DD_\vbu\baf) \cdot \wt{\n} (\Dt\baf + \DD_\vbu\baf) \dd{S_t}) } \\
				&\quad \less \abs{\vmu}_{\kk-\frac{1}{2}} \cdot \qty(\abs{\pd_t\f}_{\H{\frac{1}{2}}}^2 + \abs{\f}_{\H{\frac{3}{2}}}^2)
			\end{split}
		\end{equation}
		If $ l = 1 $ (so $ k \ge 3 $, since $ l \le k-2 $), direct computations give that
		\begin{equation}\label{eqn l = 1 case}
			\begin{split}
				&\hspace{-2em}\dv{t} \int_\Gt \wt{\n}(\Dt\baf + \DD_\vbu\baf) \cdot (-\lap_\Gt)\wt{\n}(\Dt\baf + \DD_\vbu\baf) \dd{S_t} \\
				=\, &\int_\Gt \wt{\n}\qty(\Dt^2\baf + \DD_\vbu\baf) \cdot (-\lap_\Gt)\wt{\n}(\Dt\baf + \DD_\vbu\baf) \dd{S_t} \\
				&+ \int_\Gt \wt{\n}(\Dt\baf + \DD_\vbu\baf) \cdot (-\lap_\Gt)\wt{\n}\qty(\Dt^2\baf + \Dt\DD_\vbu\baf) \dd{S_t} \\
				&+ \int_\Gt \comm{\Dt}{\wt{\n}}(\Dt\baf + \DD_\vbu\baf) \cdot  (-\lap_\Gt)\wt{\n}(\Dt\baf + \DD_\vbu\baf) \dd{S_t} \\
				&+ \int_\Gt \wt{\n}(\Dt\baf + \DD_\vbu\baf) \cdot \comm{\Dt}{(-\lap_\Gt)} \wt{\n}(\Dt\baf + \DD_\vbu\baf) \dd{S_t} \\
				&+ \int_\Gt \wt{\n}(\Dt\baf + \DD_\vbu\baf) \cdot (-\lap_\Gt) \comm{\Dt}{\wt{\n}}(\Dt\baf + \DD_\vbu\baf) \dd{S_t} \\
				&+ \int_\Gt \wt{\n}(\Dt\baf + \DD_\vbu\baf) \cdot (-\lap_\Gt)\wt{\n}(\Dt\baf + \DD_\vbu\baf) \cdot (\Div_\Gt\vmu) \dd{S_t}.
			\end{split}
		\end{equation}
		Hence, one may derive from the self-adjointness of $ \wt{\n}, \lap_\Gt $ and the estimates \eqref{comm est dt wtn}-\eqref{comm est Dt lapGt} that
		\begin{equation}
			\begin{split}
				&\frac{1}{2}\dv{t} \int_{\Gt} \abs{\qty(-\wt{\n}^\frac{1}{2}\lap_\Gt\wt{\n}^\frac{1}{2})^\frac{1}{2}\wt{\n}^\frac{1}{2}(\Dt\baf + \DD_\vbu\baf)}^2 \dd{S_t} \\
				& \quad = \int_\Gt \wt{\n}\qty(\Dt^2\baf + \Dt\DD_\vbu\baf) \cdot (-\lap_\Gt)\wt{\n}(\Dt\baf + \DD_\vbu \baf) \dd{S_t} + r_0,
			\end{split}
		\end{equation}
		with the estimate:
		\begin{equation}
			\abs{r_0} \lesssim_{\Lambda_{*}, L_1} \abs{\vmu}_{\kk-\frac{1}{2}} \times \qty(\abs{\pd_t\f}_{\H{2}}^2 + \abs{\f}_{\H3}^2).
		\end{equation}
		Therefore, by using the following relations:
		\begin{equation*}
			\begin{split}
				&\int_\Gt \abs{\qty(-\wt{\n}^\frac{1}{2}\lap_\Gt\wt{\n}^\frac{1}{2})^\frac{2m+1}{2}\wt{\n}^\frac{1}{2}(\Dt\baf + \DD_\vbu \baf)}^2 \dd{S_t} \\
				&\quad = \int_\Gt \qty(-\lap_\Gt\wt{\n})^m\wt{\n}(\Dt\baf + \DD_\vbu\baf) \cdot \qty(-\lap_\Gt\wt{\n})^{m+1}(\Dt\baf + \DD_\vbu\baf) \dd{S_t},
			\end{split}
		\end{equation*}
		and
		\begin{equation*}
			\begin{split}
				&\int_\Gt \abs{\qty(-\wt{\n}^\frac{1}{2}\lap_\Gt\wt{\n}^\frac{1}{2})^\frac{2m}{2}\wt{\n}^\frac{1}{2}(\Dt\baf + \DD_\vbu \baf)}^2 \dd{S_t} \\
				&\quad = \int_\Gt \qty(-\lap_\Gt\wt{\n})^m\wt{\n}(\Dt\baf + \DD_\vbu\baf) \cdot \qty(-\lap_\Gt\wt{\n})^{m}(\Dt\baf + \DD_\vbu\baf) \dd{S_t},
			\end{split}
		\end{equation*}
		one can argue as the cases for $ l = 0 $ and for $ l = 1 $ to obtain the following estimate:
		
		\begin{equation}\label{est I_0}
			\begin{split}
				&\hspace{-2em}\abs{\dfrac{1}{2}\dv{t} \int_{\Gt} \abs{\qty(-\wt{\n}^\frac{1}{2}\lap_\Gt\wt{\n}^\frac{1}{2})^\frac{l}{2} \wn^{\frac{1}{2}}(\Dt \baf + \DD_\vbu\baf)}^2 \dd{S_t} - I } \\
				\lesssim_{\Lambda_{*}, L_1} &\abs{\vmu}_{\kk-\frac{1}{2}} \cdot \qty(\abs{\f}^2_{\H{\frac{3}{2}l + \frac{3}{2}}} + \abs{\pd_t\f}^2_{\H{\frac{3}{2}l + \frac{1}{2}}}),
			\end{split} 
		\end{equation}
		where
		\begin{equation}
			I \coloneqq \int_\Gt \qty(-\wt{\n}^\frac{1}{2}\lap_\Gt\wt{\n}^\frac{1}{2})^\frac{l}{2} \wn^{\frac{1}{2}} \qty(\Dt^2\baf + \Dt\DD_\vbu \baf) \cdot \qty(-\wt{\n}^\frac{1}{2}\lap_\Gt\wt{\n}^\frac{1}{2})^\frac{l}{2} \wn^{\frac{1}{2}} \qty(\Dt\baf + \DD_{\vbu}\baf) \dd{S_t}.
		\end{equation}
		For simplicity, set
		\begin{equation}
			\opd \coloneqq \qty(-\wt{\n}^\frac{1}{2}\lap_\Gt\wt{\n}^\frac{1}{2})^\frac{1}{2}.
		\end{equation}
		The equation in \eqref{eqn linear 1} is equivalent to
		\begin{equation}\label{eqn Dt2 baf}
			\begin{split}
				&\Dt^2 \baf  + 2 \DD_\vbu \Dt \baf + \DD_\vbu\DD_\vbu \baf + \qty(-\lap_\Gt\wt{\n})\baf + \frac{\rho_+ \rho_-}{(\rho_+ + \rho_-)^2}\DD_\vw\DD_\vw \baf \\
				&\quad- \frac{\rho_+}{\rho_+ + \rho_-}\DD_{\vh_+}\DD_{\vh_+}\baf - \frac{\rho_-}{\rho_+ + \rho_-}\DD_{\vh_-}\DD_{\vh_-}\baf = \bar{\g},
			\end{split}
		\end{equation}
		which implies
		\begin{equation}
			\begin{split}
				I = &\int_\Gt \opd^l \wn^{\frac{1}{2}} \qty(\Dt \baf + \DD_\vbu \baf) \cdot \opd^l \wn^{\frac{1}{2}} \bar{\g} \dd{S_t} \\
				&+\int_\Gt \opd^l \wn^{\frac{1}{2}} \qty(\Dt \DD_\vbu \baf - 2 \DD_\vbu \Dt \baf - \DD_\vbu \DD_\vbu \baf) \cdot \opd^l \wn^{\frac{1}{2}} \qty(\Dt \baf + \DD_\vbu \baf) \dd{S_t} \\
				&-\int_\Gt \opd^l \wn^{\frac{1}{2}} \qty(-\lap_\Gt\wt{\n} \baf) \cdot \opd^l \wn^{\frac{1}{2}} \qty(\Dt\baf + \DD_\vbu\baf) \dd{S_t} \\
				&-\dfrac{\rho_+ \rho_-}{\qty(\rho_+ + \rho_-)^2}\int_\Gt \opd^l\wnh \qty(\DD_\vw\DD_\vw\baf) \cdot \opd^l\wnh(\Dt\baf + \DD_\vbu \baf) \dd{S_t} \\
				&+\dfrac{\rho_+}{\rho_+ + \rho_-} \int_\Gt \opd^l\wnh \qty(\DD_{\vh_+}\DD_{\vh_+}\baf) \cdot \opd^l\wnh(\Dt\baf + \DD_\vbu \baf) \dd{S_t} \\
				&+\dfrac{\rho_-}{\rho_+ + \rho_-}\int_\Gt \opd^l\wnh \qty(\DD_{\vh_-}\DD_{\vh_-}\baf) \cdot \opd^l\wnh(\Dt\baf + \DD_\vbu \baf) \dd{S_t} \\
				=: &I_1 + I_2 + I_3 + I_4 + I_5 + I_6.
			\end{split}
		\end{equation}
		It is clear that
		\begin{equation}\label{est I_1}
			\abs{I_1} \lesssim_{L_1} \abs{\g}_{\H{\frac{3}{2}l + \frac{1}{2}}} \cdot \qty(\abs{\pd_t\f}_{\H{\frac{3}{2}l + \frac{1}{2}}} + \abs{\f}_{\H{\frac{3}{2}l + \frac{3}{2}}}).
		\end{equation}
		Note that for two functions $ \phi, \psi : \Gt \to \R $, the integration-by-parts formula is
		\begin{equation}\label{formula int by parts Gt}
			\int_\Gt - (\DD_\vbu \phi) \cdot \psi \dd{S_t} = \int_\Gt (\DD_\vbu\psi)\cdot\phi + \phi\psi(\Div_\Gt\vbu) \dd{S_t},
		\end{equation}
		since $ \vbu : \Gt \to \mathrm{T}\Gt $ is a tangential field and $ \int_\Gt \Div_\Gt \qty(\vbu \phi \psi) \dd{S_t} \equiv 0$.
		
		For $ I_2 $, observe that
		\begin{equation}\label{int DuDuf * Dtf}
			\begin{split}
				\int_\Gt \opd^l \wt{\n}^\frac{1}{2}(\DD_\vbu\DD_\vbu \baf) \cdot \opd^l \wt{\n}^\frac{1}{2}(\Dt\baf) \dd{S_t}  = \int_\Gt \wt{\n}(\DD_\vbu\DD_\vbu\baf) \cdot (-\lap_\Gt\wt{\n})^l (\Dt\baf) \dd{S_t}.
			\end{split}
		\end{equation}
		Thus, commuting $ \DD_\vbu $ with $ \wt{\n} $ and $ \lap_\Gt $ via the arguments as \eqref{eqn l = 0 case} and \eqref{eqn l = 1 case}, it is routine to derive that
		\begin{equation}\label{est DuDu f * Dt f}
			\begin{split}
				&\hspace{-1em}\abs{ \int_\Gt \opd^l \wnh (\DD_\vbu \DD_\vbu \baf) \cdot \opd^l \wnh (\Dt\baf) + \opd^l\wnh(\DD_\vbu\baf) \cdot \opd^l\wnh (\DD_\vbu\Dt \baf)  \dd{S_t}} \\
				&\lesssim_{L_1}  \abs{\f}_{\H{\frac{3}{2}l +\frac{3}{2}}} \cdot \abs{\pd_t \f}_{\H{\frac{3}{2}l+\frac{1}{2}}}.
			\end{split}
		\end{equation}		
		Similarly, it follows from \eqref{formula int by parts Gt} that
		\begin{equation}
			\begin{split}
				\abs{\int_\Gt \opd^l \wnh (\DD_\vbu \DD_\vbu \baf) \cdot \opd^l \wnh (\DD_\vbu\baf) \dd{S_t}} \lesssim_{L_1}\abs{\f}_{\H{\frac{3}{2}l+\frac{3}{2}}}^2,
			\end{split}
		\end{equation}
		and
		\begin{equation}
			\begin{split}
				\abs{\int_\Gt \opd^l \wnh (\DD_\vbu \Dt \baf) \cdot \opd^l \wnh (\Dt \baf) \dd{S_t} } \lesssim_{L_1} \abs{\pd_t \f}_{\H{\frac{3}{2}l + \frac{1}{2}}}^2.
			\end{split}
		\end{equation}
		Furthermore, observing that 
		\begin{equation}\label{est I2 last}
			\abs{\comm{\Dt}{\DD_\vbu}\baf}_{\frac{3}{2}l+\frac{1}{2}} = \abs{\DD_{\qty(\Dt\vbu-\DD_\vbu \vmu)} \baf}_{\frac{3}{2}l+\frac{1}{2}} \lesq \abs{\f}_{\H{\frac{3}{2}l + \frac{7}{4}}},
		\end{equation}
		one can deduce from (\ref{int DuDuf * Dtf})-(\ref{est I2 last}) that
		\begin{equation}\label{est I_2}
			\abs{I_2} \lesq \abs{\pd_t\f}_{\H{\frac{3}{2}l+\frac{1}{2}}}^2 + \abs{\f}_{\H{\frac{3}{2}l+2}}^2.
		\end{equation}
		
		Next, the estimate on $ I_3 $ can be derived via:
		\begin{equation}
			\begin{split}
				I_3 = &-\int_\Gt \opd^l \wnh \qty(-\lap_\Gt\wn)\baf \cdot \opd^l\wnh \qty(\Dt\baf + \DD_\vbu \baf) \dd{S_t} \\
				= &-\int_\Gt \wt{\n}\baf \cdot (-\lap_\Gt\wt{\n})^{l+1}(\Dt\baf + \DD_\vbu\baf) \dd{S_t},
			\end{split}
		\end{equation}
		which, together with the previous arguments, yield
		\begin{equation}\label{est I_3}
			\begin{split}
				\abs{I_3 + \dfrac{1}{2}\dv{t}\int_\Gt \abs{\opd^{l+1}\wnh\baf}^2 \dd{S_t}} 
				\lesq \abs{\f}_{\H{\frac{3}{2}l + 2}}^2.
			\end{split}
		\end{equation}
		
		As for $ I_4 $, in view of the relation
		\begin{equation}
			\comm{\DD_\vbu}{\DD_\vw} \baf = \DD_{\comm{\vbu}{\vw}} \baf,
		\end{equation}
		it follows from the integration-by-parts that
		\begin{equation}
			\begin{split}
				\abs{\int_\Gt \opd^l\wnh(\DD_\vw\DD_\vw\baf) \cdot \opd^l\wnh(\DD_\vbu\baf)\dd{S_t}} 
				\lesssim_{L_1} \abs{\f}_{\H{\frac{3}{2}l+\frac{3}{2}}}^2.
			\end{split}
		\end{equation}		
		The previous arguments can be used to show
		\begin{equation}
			\begin{split}
				&\hspace{-1em}\abs{ \int_\Gt \opd^l \wnh (\DD_\vw \DD_\vw \baf) \cdot \opd^l \wnh (\Dt\baf) + \opd^l\wnh(\DD_\vw\baf) \cdot \opd^l\wnh (\DD_\vw\Dt \baf)  \dd{S_t}} \\
				&\lesssim_{L_1}  \abs{\f}_{\H{\frac{3}{2}l +\frac{3}{2}}} \times \abs{\pd_t \f}_{\H{\frac{3}{2}l+\frac{1}{2}}},
			\end{split}
		\end{equation}
		and
		\begin{equation}\label{I_4 est end}
			\begin{split}
				&\hspace{-1em}\abs{\int_\Gt \opd^l \wnh (\DD_\vw \DD_\vw \baf) \cdot \opd^l \wnh (\Dt\baf) \dd{S_t } + \dfrac{1}{2}\dv{t} \int_\Gt \abs{\opd^l\wnh(\DD_\vw\baf)}^2 \dd{S_t}} \\
				&\lesq\abs{\f}_{\H{\frac{3}{2}l+\frac{3}{2}}} \times \qty(\abs{\f}_{\H{\frac{3}{2}l+2}} + \abs{\pd_t\f}_{\H{\frac{3}{2}l+\frac{1}{2}}}),
			\end{split}
		\end{equation}
		that is,
		\begin{equation}\label{est I_4}
			\begin{split}
				&\abs{I_4 - \dfrac{\rho_+\rho_-}{2\qty(\rho_+ + \rho_-)^2} \dv{t}\int_\Gt \abs{\opd^l\wnh(\DD_\vw\baf)}^2 \dd{S_t}} \\
				&\quad\lesq \abs{\f}_{\H{\frac{3}{2}l+\frac{3}{2}}} \times \qty(\abs{\f}_{\H{\frac{3}{2}l+2}} + \abs{\pd_t\f}_{\H{\frac{3}{2}l+\frac{1}{2}}}).
			\end{split}
		\end{equation}
		Since $ I_5 $ and $ I_6 $ have the same form as $ I_4 $, there hold
		\begin{equation}\label{est I_5}
			\begin{split}
				&\abs{I_5 + \dfrac{\rho_+}{2\qty(\rho_+ + \rho_-)} \dv{t}\int_\Gt \abs{\opd^l\wnh(\DD_{\vh_+}\baf)}^2 \dd{S_t} } \\
				&\quad\lesq \abs{\f}_{\H{\frac{3}{2}l+\frac{3}{2}}} \times \qty(\abs{\f}_{\H{\frac{3}{2}l+2}} + \abs{\pd_t\f}_{\H{\frac{3}{2}l+\frac{1}{2}}}),
			\end{split}
		\end{equation}
		and
		\begin{equation}\label{est I_6}
			\begin{split}
				&\abs{I_6 + \dfrac{\rho_-}{2\qty(\rho_+ + \rho_-)} \dv{t}\int_\Gt \abs{\opd^l\wnh(\DD_{\vh_-}\baf)}^2 \dd{S_t} } \\
				&\quad\lesq \abs{\f}_{\H{\frac{3}{2}l+\frac{3}{2}}} \times \qty(\abs{\f}_{\H{\frac{3}{2}l+2}} + \abs{\pd_t\f}_{\H{\frac{3}{2}l+\frac{1}{2}}}).
			\end{split}
		\end{equation}
		
		In conclusion, the combination of (\ref{eqn E_l}), (\ref{est I_0}), (\ref{est I_1}), (\ref{est I_2}), (\ref{est I_3}), (\ref{est I_4})-(\ref{est I_6}) gives that
		\begin{equation}
			\begin{split}
				\abs{\dv{t} E_l} \lesq &\qty(\abs{\f}_{\H{\frac{3}{2}l+2}} + \abs{\pd_t\f}_{\H{\frac{3}{2}l+\frac{1}{2}}}) \times \\
				&\quad\times \qty(\abs{\f}_{\H{\frac{3}{2}l+2}} + \abs{\pd_t\f}_{\H{\frac{3}{2}l+\frac{1}{2}}} + \abs{\g}_{\H{\frac{3}{2}l+\frac{1}{2}}}),
			\end{split}
		\end{equation}
		which completes the proof of Lemma \ref{lem est E_l}.
	\end{proof}
	
	Based on Lemma \ref{lem est E_l}, the following energy estimate for \eqref{eqn linear 1} holds:
	\begin{prop}
		There is a constant $ C_* > 0 $ determined by $ \Lambda_* $ so that for any integer $ 0 \le l \le k-2 $, and $ 0 \le t \le T $ ($ T \le C $ for some constant $ C = C(L_1, L_2) $), it holds that
		\begin{equation}\label{est linear eqn ka}
			\begin{split}
				&\hspace{-1em}\abs{\f(t)}_{\H{\frac{3}{2}l+2}}^2 + \abs{\pd_t\f(t)}_{\H{\frac{3}{2}l+\frac{1}{2}}}^2 \\
				\le\, &C_* \exp\qty\big[Q(L_1, L_2)t]\times \\
				& \times \qty{ \abs{\f_0}_{\H{\frac{3}{2}l+2}}^2 + \abs{\f_1}_{\H{\frac{3}{2}l+\frac{1}{2}}}^2 + (L_0)^{12} \abs{\f_0}_{\H{\frac{3}{2}l+\frac{1}{2}}}^2 + \int_0^t \abs{\g(t')}^2_{\H{\frac{3}{2}l+\frac{1}{2}}} \dd{t'} },
			\end{split}
		\end{equation}
		where $ Q $ is a generic polynomial determined by $ \Lambda_* $.
	\end{prop}
	
	\begin{proof}
		It is clear that for some $ C_* > 0 $, it holds that
		\begin{equation}\label{eqn linear energy est}
			\begin{split}
				&\abs{\baf}_{\frac{3}{2}l+2}^2 + \abs{\Dt\baf}_{\frac{3}{2}l+\frac{1}{2}}^2 \\
				&\le C_* E_l(t, \f, \pd_t\f) + C_* \qty(\abs{\int_\Gt \baf \dd{S_t}}^2 + \abs{\int_\Gt \Dt\baf \dd{S_t}}^2)  +C_* \qty(\abs{\DD_{\vw}\baf}_{\frac{3}{2}l + \frac{1}{2}}^2 + \abs{\DD_\vbu \baf}_{\frac{3}{2}l + \frac{1}{2}}^2) \\
				&\le C_* E_l(t, \f, \pd_t\f) + C_* \qty(\abs{\int_\Gt \baf \dd{S_t}}^2 + \abs{\int_\Gt \Dt\baf \dd{S_t}}^2)  +C_* \qty(\abs{\vv_{+}}_{\kk-2}^2 + \abs{\vv_{-}}_{\kk-2}^2) \abs{\baf}_{\frac{3}{2}l + \frac{7}{4}}^2.
			\end{split}
		\end{equation}
		For the last term above, it follows from the interpolation inequality that
		\begin{equation}
			\begin{split}
				\qty(\abs{\vv_{+}}_{\kk-2}^2 + \abs{\vv_{-}}_{\kk-2}^2) \abs{\baf}_{\frac{3}{2}l + \frac{7}{4}}^2  
				\le \frac{5}{6C_*} \abs{\baf}_{\frac{3}{2}l+2}^2 + \dfrac{(C_*)^5}{6}\qty(\abs{\vv_+}_{\kk-2}^{12} + \abs{\vv_-}_{\kk-2}^{12}) \abs{\baf}_{\frac{3}{2}l+\frac{1}{2}}^2,
			\end{split}
		\end{equation}
		with
		\begin{equation}
			\begin{split}
				\abs{\vv_\pm (t)}_{\kk-2} &\less \abs{\vv_{\pm*}(t)}_{\H{\kk-2}} \\
				&\less \abs{\vv_{\pm*}(0)}_{\H{\kk-2}} + L_2 t.
			\end{split}
		\end{equation}
		Now, observe that
		\begin{equation}
			\dv{t} \int_\Gt \baf(t) \dd{S_t} = \int_\Gt \Dt\baf + \baf \Div_\Gt \vmu \dd{S_t}.
		\end{equation}
		Thus
		\begin{equation}
			\begin{split}
				\abs{\int_\Gt \baf(t) \dd{S_t} - \int_{\Gamma_0} \baf(0) \dd{S_0} } 
				\lesssim_{L_1} \int_0^t \abs{\pd_t\f(t')}_{\L2} + \abs{\f(t')}_{\L2} \dd{t'}.
			\end{split}
		\end{equation}
		To deal with the integral of $ \Dt\baf $, one notes that
		\begin{equation}
			\dv{t} \int_\Gt \Dt\baf \dd{S_t} = \int_\Gt \Dt^2 \baf + \qty(\Dt \baf) \Div_\Gt \vmu \dd{S_t},
		\end{equation}
		and (\ref{eqn Dt2 baf}) implies that
		\begin{equation}
			\begin{split}
				\int_\Gt \Dt^2\baf \dd{S_t} = &-\int_\Gt \DD_\vbu \qty(2\Dt\baf + \DD_\vbu \baf ) \dd{S_t} +\int_\Gt \lap_\Gt \wn \baf \dd{S_t} \\
				&- \dfrac{\rho_+ \rho_-}{\qty(\rho_+ + \rho_-)^2} \int_\Gt \DD_\vw \DD_\vw \baf \dd{S_t} + \dfrac{\rho_+}{\rho_+ + \rho_-}\int_\Gt \DD_{\vh_+}\DD_{\vh_+} \baf \dd{S_t} \\
				&+\dfrac{\rho_-}{\rho_+ + \rho_-} \int_\Gt \DD_{\vh_-} \DD_{\vh_-} \baf \dd{S_t} + \int_\Gt \bar{\mathfrak{g}} \dd{S_t}.
			\end{split}
		\end{equation}
		Due to the fact that $ \pd \Gt = \OO $, one has
		\begin{equation}
			\int_\Gt \lap_\Gt \wn \baf \dd{S_t} \equiv 0,
		\end{equation}
		and
		\begin{equation}
			-\int_\Gt \DD_\vbu \Dt \baf \dd{S_t} = \int_\Gt \qty(\Dt \baf) \Div_\Gt \vbu \dd{S_t}.
		\end{equation}
		Thus
		\begin{equation}
			\abs{\dv{t}\int_\Gt \Dt\baf \dd{S_t} } \lesssim_{L_1} \abs{\Dt \baf}_0 + \abs{\baf}_1 + \abs{\bar{\g}}_0,
		\end{equation}
		and
		\begin{equation}\label{est Dt baf int}
			\begin{split}
				&\abs{\int_\Gt \qty(\Dt\baf)(t) \dd{S_t} - \int_{\Gamma_0} \qty(\Dt\baf)(0)\dd{S_0} }  \\ &\quad\lesssim_{L_1}  \int_0^t \abs{\pd_t\f(t')}_{\L2} + \abs{\f(t')}_{\H1} + \abs{\g(t')}_{\L2} \dd{t'}.
			\end{split}
		\end{equation}
		
		Combining (\ref{eqn linear energy est})-(\ref{est Dt baf int}), (\ref{est E_l}) and the following relation:
		\begin{equation}
			\abs{\baf(t)}_{\frac{3}{2}l+\frac{1}{2}} \less \abs{\f(0)}_{\H{\frac{3}{2}l+\frac{1}{2}}} + \int_0^t \abs{\pd_t\f(t')}_{\H{\frac{3}{2}l+\frac{1}{2}}} \dd{t'},
		\end{equation}
		one can get that
		\begin{equation}
			\begin{split}
				&\hspace{-3em}\abs{\baf(t)}_{\frac{3}{2}l+2}^2 + \abs{\Dt\baf(t)}_{\frac{3}{2}l+\frac{1}{2}}^2 \\
				\less &\abs{\f_0}_{\H{\frac{3}{2}l+2}}^2 + \abs{\f_1}^2_{\H{\frac{3}{2}l+\frac{1}{2}}} + \qty(L_0)^{12}\abs{\f_0}_{\H{\frac{3}{2}l+\frac{1}{2}}}^2 \\
				&+ \bar{Q} \cdot \int_0^t \abs{\f(t')}_{\H{\frac{3}{2}l+2}}^2 + \abs{\pd_t\f(t')}_{\H{\frac{3}{2}l+\frac{1}{2}}}^2 + \abs{\g(t')}_{\H{\frac{3}{2}l+\frac{1}{2}}}^2 \dd{t'} \\
				&+ \qty(\abs{\vv_{\pm*}(0)}_{\H{\kk-2}}^{12} + (L_2)^{12}t^{12})\cdot C(L_1)t \int_0^t \abs{\pd_t\f(t')}_{\H{\frac{3}{2}l+\frac{1}{2}}}^2 \dd{t'} \\
				&+C(L_1)t\int_0^t \abs{\pd_t\f(t')}_{\L2}^2 + \abs{\f(t')}_{\H1}^2 + \abs{\g(t')}_{\L2}^2 \dd{t'},
			\end{split}
		\end{equation}
		where $ \bar{Q} = \bar{Q}(L_1, L_2) $ is the generic polynomial in the previous lemma.
		If $ T \le Q_* (L_1, L_2) $, it follows that
		\begin{equation}
			\begin{split}
				&\hspace{-2em}\abs{\baf(t)}_{\frac{3}{2}l+2}^2 + \abs{\Dt\baf(t)}_{\frac{3}{2}l+\frac{1}{2}}^2 \\
				\le\,  &C_* \qty(\abs{\f_0}_{\H{\frac{3}{2}l+2}}^2 + \abs{\f_1}_{\H{\frac{3}{2}l+\frac{1}{2}}}^2 + (L_0)^{12}\abs{\f_0}_{\H{\frac{3}{2}l+\frac{1}{2}}}^2) \\
				&+ Q\int_0^t \abs{\f(t')}_{\H{\frac{3}{2}l+2}}^2 + \abs{\pd_t\f(t')}_{\H{\frac{3}{2}l+\frac{1}{2}}}^2 + \abs{\g(t')}_{\H{\frac{3}{2}l+\frac{1}{2}}}^2 \dd{t'},
			\end{split}
		\end{equation}
		where $ Q = Q(L_1, L_2) $ is a generic polynomial determined by $ \Lambda_* $. Hence (\ref{est linear eqn ka}) follows from Gronwall's inequality.
	\end{proof}
	Then the local well-posedness of \eqref{eqn linear 1} follows from this energy estimate:
	\begin{cor}
		For $ 0 \le l \le k-2 $, $ T \le C(L_1, L_2) $ and $ \g \in C^0 \qty([0, T]; H^{\frac{3}{2}l + \frac{1}{2}}(\Gs)) $, the problem \eqref{eqn linear 1} is well-posed in $ C^0\qty([0, T]; \H{\frac{3}{2}l+2}) \cap C^1\qty([0, T]; \H{\frac{3}{2}l+\frac{1}{2}}) $, and the estimate \eqref{est linear eqn ka} holds.
	\end{cor}

	\subsection{Linearized system for the current and vorticity}\label{sec linear current vortex}
	Assume still that $ \Gs \in H^{\kk+1}$ $ (k\ge2)$ is a reference hypersurface and consider a family of hypersurfaces $ \{\Gt\}_{0\le t \le T} \subset \Lambda_* $, for which each $ \Gt $ separates $ \Om $ into two disjoint simply-connected domains $ \Om_t^\pm $. Suppose that $ \vv_\pm(t), \vh_\pm(t) : \Om_t^\pm \to \R^3 $ are given vector fields satisfying
	\begin{equation}
		\begin{cases*}
			\div\vv_\pm = 0 = \div\vh_\pm &in $ \Om_t^\pm $, \\
			\vh_+ \vdot \vn_+ = 0 = \vh_- \vdot \vn_+ &on $ \Gt $, \\
			\vv_+ \vdot \vn_+ = \vn_+ \vdot \qty(\pd_t\gt\vnu)\circ\Phi_\Gt^{-1} = \vv_- \vdot \vn_+ &on $ \Gt $, \\
			\vv_- \vdot \wt{\vn} = 0 = \vh_- \vdot \wt{\vn} &on $ \pd\Om $.
		\end{cases*}
	\end{equation}
	Assume further that there is a constant $ \bar{L}_1 $ so that
	\begin{equation}
		\sup_{t \in [0, T]} \qty(\abs{\ka(t)}_{\H{\kk-\frac{3}{2}}}, \abs{\pd_t\ka(t)}_{\H{\kk-\frac{5}{2}}}, \norm{\qty(\vv_\pm(t), \vh_\pm(t))}_{H^{\kk}(\Om_t^\pm)}) \le \bar{L}_1.
	\end{equation}
	
	Consider the following linearized current-vorticity system in $ \Om \setminus \Gt $:
	\begin{equation}\label{eqn linear vom}
		\pd_t \wt{\vom} + (\vv\vdot\grad)\wt{\vom} - (\vh\vdot\grad)\wt{\vj} = (\wt{\vom}\vdot\grad)\vv - (\wt{\vj}\vdot\grad)\vh,
	\end{equation}
	\begin{equation}\label{eqn linear vj}
		\pd_t \wt{\vj} + (\vv\vdot\grad)\wt{\vj} - (\vh\vdot\grad)\wt{\vom} = (\wt{\vj}\vdot\grad)\vv - (\wt{\vom}\vdot\grad)\vh -2\tr(\grad\vv\cp \grad\vh).
	\end{equation}
	Set 
	\begin{equation}\label{def vxi veta}
		\vb*{\xi} \coloneqq \wt{\vom} - \wt{\vj}, \quad \vb*{\eta} \coloneqq \wt{\vom} + \wt{\vj}.
	\end{equation}
	Then
	\begin{equation}
		\pd_t\vxi  + \qty[(\vv+\vh)\vdot\grad]\vb*{\xi} = (\vb*{\xi}\vdot\grad)(\vv+\vh) + 2 \tr(\grad\vv\cp\grad\vh),
	\end{equation}
	\begin{equation}
		\pd_t \veta + \qty[(\vv-\vh)\vdot\grad]\veta = (\veta\vdot\grad)(\vv-\vh) - 2\tr(\grad\vv\cp\grad\vh).
	\end{equation}
	Define the flow map $ \mathcal{Y}^\pm $ by
	\begin{equation}
		\dv{t}\mathcal{Y}^\pm(t, y) = \qty(\vv_\pm - \vh_\pm) \qty(t, \mathcal{Y}^\pm(t, y)), \quad y \in \Om_0^\pm.
	\end{equation}
	As indicated in \cite{Sun-Wang-Zhang2018}, due to the fact that $ \vh_\pm \vdot \vn_+ \equiv 0 $, $ \mathcal{Y}^\pm(t) $ is a bijection from $ \Om_0^\pm $ to $ \Om_t^\pm $ for small time $ t $. Furthermore, the evolution equation for $ \veta $ can be rewritten as
	\begin{equation}
		\dv{t}(\veta\circ\mathcal{Y}) = \qty[(\veta\vdot\grad)(\vv-\vh)-2\tr(\grad\vv\cp\grad\vh)]\circ\mathcal{Y},
	\end{equation}
	or equivalently,
	\begin{equation}
		\dv{t}\qty[\qty(\DD\mathcal{Y})^{-1}\qty(\veta\circ\mathcal{Y})] = -2\tr(\grad\vv\cp\grad\vh)\circ\mathcal{Y},
	\end{equation}
	which is a linear ODE system. Thus, the local well-posedness follows routinely. Similarly, the evolution equation for $ \vxi $ is also locally well-posed on $ [0, T] $, with the life span $ T $ depending on $ \bar{L}_1 $. Furthermore, the following energy estimates hold:
	\begin{prop}\label{prop linear vom vj}
		For $ 0 \le t \le T $, it follows that
		\begin{equation}\label{est linear vom vj}
			\begin{split}
				&\hspace{-2em}\norm{\wt{\vom}_\pm(t)}_{H^{\kk-1}(\Om_t^\pm)}^2 + \norm{\wt{\vj}_\pm(t)}_{H^{\kk-1}(\Om_t^\pm)}^2 \\
				\le\, &\exp{Q\qty(\bar{L}_1)t} \qty(1+ \norm{\wt{\vom}_\pm(0)}_{H^{\kk-1}(\Om_0^\pm)}^2 + \norm{\wt{\vj}_\pm(0)}_{H^{\kk-1}(\Om_0^\pm)}^2),
			\end{split}
		\end{equation}
		where $ Q $ is a generic polynomial depending on $ \Lambda_* $.
	\end{prop}
	\begin{proof}
		For $ 0 \le s \le \kk-1 $, observe that
		\begin{equation}
			\begin{split}
				&\hspace{-1em}\dfrac{1}{2}\dv{t}\int_{\Om_t^\pm} \abs{\nabla^s \veta_\pm}^2\dd{x} \\
				=\,& \int_{\Om_t^\pm} \nabla^s \veta_\pm \vdot \nabla^s\pd_t\veta_\pm \dd{x} + \dfrac{1}{2}\int_{\Om_t^\pm} \DD_{(\vv_\pm - \vh_\pm)} \abs{\nabla^s \veta_\pm}^2 \dd{x} \\
				=\, &\int_{\Om_t^\pm} \nabla^s \veta_\pm \vdot \nabla^s\qty[\DD_{(\vh_\pm-\vv_\pm)}\veta_\pm + \DD_{\veta_\pm}(\vv_\pm-\vh_\pm) - 2\tr(\grad\vv_\pm \cp \grad \vh_\pm) ] \dd{x} \\
				&+ \dfrac{1}{2}\int_{\Om_t^\pm} \DD_{(\vv_\pm - \vh_\pm)} \abs{\nabla^s \veta_\pm}^2 \dd{x} \\
				=\, &\int_{\Om_t^\pm} \nabla^s \veta_\pm \vdot \comm{\nabla^s}{\DD_{(\vh_\pm - \vv_\pm)}}\veta_\pm \dd{x} +\int_{\Om_t^\pm} \nabla^s \veta_\pm \vdot \nabla^s\qty{\DD_{\veta_\pm}(\vv_\pm-\vh_\pm) - 2\tr(\grad\vv_\pm\cp\grad\vh_\pm)} \dd{x} \\
				\le\, &Q\qty(\bar{L}_1) \qty(1+\norm{\veta_\pm}_{H^s(\Om_t^\pm)}^2).
			\end{split}
		\end{equation}
		It follows from a similar argument that
		\begin{equation}
			\dv{t} \norm{\vxi_\pm (t)}_{H^s(\Om_t^\pm)}^2 \le Q\qty(\bar{L}_1) \qty(1+\norm{\vxi_\pm (t)}_{H^s(\Om_t^\pm)}^2).
		\end{equation}
		Therefore, (\ref{est linear vom vj}) holds due to Gronwall's inequality and (\ref{def vxi veta}).
	\end{proof}
	
	To show the compatibility conditions:
	\begin{equation}\label{linear compatibility condition}
		\int_{\pd\Om} \wt{\vom}_- \vdot \wt{\vn} \dd{\wt{S}} \equiv 0 \equiv \int_{\pd\Om} \wt{\vj}_- \vdot \wt{\vn} \dd{\wt{S}},
	\end{equation}
	one observes that
	\begin{equation}
		\begin{split}
			&\hspace{-2em}\dv{t}\int_{\pd\Om} \veta_- \vdot \wt{\vn} \dd{\wt{S}} \\
			=\, &\int_{\pd\Om} \qty(\pd_t + \DD_{(\vv_- - \vh_-)}) \qty(\veta_- \vdot \wt{\vn}) + \qty(\veta_- \vdot \wt{\vn}) \Div_{\pd\Om} (\vv_- - \vh_-) \dd{\wt{S}} \\
			=\, &\int_{\pd\Om} \wt{\vn}\vdot\DD_{\veta_-} (\vv_- - \vh_-) - 2\wt{\vn} \vdot \tr(\grad\vv_- \cp \grad \vh_-) \dd{\wt{S}} \\
			&+\int_{\pd\Om} - \wt{\vn} \vdot \DD_{\veta_-^\top}(\vv_- - \vh_-) + (\veta_- \vdot \wt{\vn}) \Div_{\pd\Om}(\vv_- - \vh_-) \dd{\wt{S}} \\
			=:\, &I_1 + I_2.
		\end{split}
	\end{equation}
	Since $ \div \qty(\grad \phi \cp \grad \psi) \equiv 0 $ and $ \div(\vv_- - \vh_- )\equiv 0 $, 
	\begin{equation}
		\begin{split}
			I_1 =\, &\int_{\pd\Om} \wt{\vn} \vdot \DD_{\veta_-^\top}(\vv_- - \vh_-) + (\veta_- \vdot \wt{\vn})  \DD_{\wt{\vn}}(\vv_- - \vh_-) \vdot \wt{\vn} \dd{\wt{S}} \\
			=\, &\int_{\pd\Om}  \wt{\vn} \vdot \DD_{\veta_-^\top}(\vv_- - \vh_-) - (\veta_- \vdot \wt{\vn}) \Div_{\pd\Om}(\vv_- - \vh_-) \dd{\wt{S}} \\
			=\, &-I_2,
		\end{split}
	\end{equation}
	where the geometric relation \eqref{eqn div Gt div Rd} has been used.
	Thus, the similar arguments yield
	\begin{equation}
		\dv{t}\int_{\pd\Om} \vxi \vdot \wt{\vn} \dd{\wt{S}} = 0,
	\end{equation}
	which implies the following lemma:
	\begin{lem}\label{lem compatibility}
		Suppose that $ (\wt{\vom}(t), \wt{\vj}(t)) $ is the solution to the linear system  (\ref{eqn linear vom})-(\ref{eqn linear vj}) with initial data $ (\wt{\vom}_0, \wt{\vj}_0) $. If 
		\begin{equation}
			\int_{\pd\Om} \wt{\vom}_{0-} \vdot \wt{\vn} \dd{\wt{S}} = 0 = \int_{\pd\Om} \wt{\vj}_{0-} \vdot \wt{\vn} \dd{\wt{S}},
		\end{equation}
		then for all $ t $ such that the solution exists, there holds
		\begin{equation}
			\int_{\pd\Om} \wt{\vom}_-(t) \vdot \wt{\vn} \dd{\wt{S}} \equiv 0 \equiv \int_{\pd\Om} \wt{\vj}_-(t) \vdot \wt{\vn} \dd{\wt{S}}.
		\end{equation}
	\end{lem}

	\section{Nonlinear Problems}\label{sec nonlinear}
	\subsection{Initial settings}\label{sec nonlinear set-up}
	Take a reference hypersurface $ \Gs \in H^{\kk+1} $ and $ \delta_0 > 0 $ so that
	\begin{equation*}
		\Lambda_* \coloneqq \Lambda \qty(\Gs, \kk-\frac{1}{2}, \delta_0)
	\end{equation*}
	satisfies all the properties discussed in the preliminary. We will solve the nonlinear current-vortex sheet problems by an iteration scheme based on solving the linearized problems in the space:
	\begin{equation}
		\begin{split}
			&\ka \in C^0\qty([0, T]; \H{\kk-1}) \cap C^1\qty([0, T]; B_{\delta_1}\subset\H{\kk-\frac{5}{2}}) \cap C^2\qty([0, T]; \H{\kk-4}); \\
			&\vom_{\pm*}, \vj_{\pm*} \in  C^0\qty([0, T]; H^{\kk-1}(\Om_\Gs^\pm)) \cap C^1\qty([0, T]; H^{\kk-2}(\Om_\Gs^\pm)).
		\end{split}
	\end{equation}
	In order to construct the iteration map, we define the following function space:
	\begin{defi}
		For given positive constants $ T, M_0, M_1, M_2$, and $ M_3 $, define $ \mathfrak{X} $ to be the collection of $ \qty(\ka, \vom_*, \vj_*) $ satisfying:
		\begin{equation*}
			\abs{\ka(0) - \kappa_{*+}}_{\H{\kk-\frac{5}{2}}} \le \delta_1,
		\end{equation*}
		\begin{equation*}
			\abs{\qty(\pd_t \ka)(0)}_{\H{\kk-4}}, \norm{\vom_*(0)}_{H^{\kk-\frac{5}{2}}(\OGs)}, \norm{\vj_*(0)}_{H^{\kk-\frac{5}{2}}(\OGs)} \le M_0,
		\end{equation*}
		\begin{equation*}
			\sup_{t \in [0, T]} \qty(\abs{\ka}_{\H{\kk-1}}, \abs{\pd_t \ka}_{\H{\kk-\frac{5}{2}}}, \norm{\vom_*}_{H^{\kk-1}(\OGs)}, \norm{\vj_*}_{H^{\kk-1}(\OGs)}) \le M_1,
		\end{equation*}
		\begin{equation*}
			\sup_{t \in [0, T]} \qty(\norm{\pd_t\vom_*}_{H^{\kk-2}(\OGs)}, \norm{\pd_t\vj_*}_{H^{\kk-2}(\OGs)}) \le M_2,
		\end{equation*}
		\begin{equation*}
			\sup_{t \in [0, T]} \abs{\pd^2_{tt}\ka}_{\H{\kk-4}} \le a^2 M_3 \ (\text{here $ a $ is the constant in the definition of $ \ka $}).
		\end{equation*}
		In addition, the compatibility conditions
		\begin{equation}\label{eqn compati}
			\int_{\pd\Om} \wt{\vn} \vdot \vom_{*-} \dd{\wt{S}} = \int_{\pd\Om} \wt{\vn} \vdot \vj_{*-} \dd{\wt{S}} = 0
		\end{equation}
		hold for all $ t\in [0, T] $.
	\end{defi}
	
	For $ 0 < \epsilon \ll \delta_1 $ and $ A > 0 $, the collection of initial data
	\begin{equation*}
		\mathfrak{I}(\epsilon, A) \coloneqq\qty{\qty\big(\ini{\ka}, \ini{\pd_t\ka}, \ini{\vom_*}, \ini{\vj_*})}
	\end{equation*}
	is defined by:
	\begin{gather*}
		\abs{\ini{\ka}-\kappa_{*+}}_{\H{\kk-1}}<\epsilon; \\ \abs{\ini{\pd_t\ka}}_{\H{\kk-\frac{5}{2}}},\
		\norm{\ini{\vom_*}}_{H^{\kk-1}(\OGs)},\ \norm{\ini{\vj_*}}_{H^{\kk-1}(\OGs)} < A,
	\end{gather*}
	and
	\begin{equation*}
		\int_{\pd\Om} \wt{\vn} \vdot {\ini{\vom_*}}_- \dd{\wt{S}} = \int_{\pd\Om} \wt{\vn} \vdot {\ini{\vj_*}}_{-} \dd{\wt{S}} = 0.
	\end{equation*}
	Thus, $ \mathfrak{I}(\epsilon, A) \subset \H{\kk-1}\times\H{\kk-\frac{5}{2}}\times H^{\kk-1}(\OGs) \times H^{\kk-1}(\OGs) $.
	
	\subsection{Recovery of the fluid region, velocity and magnetic fields}\label{section recovery}
	For $ \qty\big(\ka, \vom_*, \vj_*)  \in \mathfrak{X}$, $ \ka(t) $ induces a family of hypersurfaces $ \Gt \in \Lambda_* $ if $ M_1 T $ is not too large. Thus $ \Phi_\Gt $ and $ \X_\Gt $ can be defined by $ \ka(t) $.
	
	For a vector field $ \vb{Y} : \OGt \to \R^3 $, define
	\begin{equation*}
		\qty(\Pb \vb{Y})_\pm \coloneqq \vb{Y}_\pm - \grad\phi_{\pm},
	\end{equation*}
	for which
	\begin{equation*}
		\begin{cases*}
			\lap \phi_\pm = \div \vb{Y}_\pm &in $ \Om_t^\pm $, \\
			\phi_\pm = 0 &on $ \Gt $, \\
			\DD_{\wt{\vn}} \phi_- = 0 &on $ \pd\Om $.
		\end{cases*}
	\end{equation*}
	Namely,  $ \Pb $ is the Leray projection. Set
	\begin{equation}\label{def bar vom vj}
		\bar{\vom} \coloneqq \Pb \qty(\vom_*(t) \circ \X_\Gt^{-1}), \quad \bar{\vj} \coloneqq \Pb\qty(\vj_*(t) \circ \X_\Gt^{-1}).
	\end{equation}
	Thus, $ \div\bar{\vom} = \div\bar{\vj} = 0 $ in $ \OGt $ and
	\begin{equation*}
		\int_{\pd\Om} \wt{\vn} \vdot \bar{\vom}_- \dd{S} = 0 = \int_{\pd\Om} \wt{\vn} \vdot \bar{\vj}_- \dd{S},
	\end{equation*}
	since $ \X_\Gt |_{\pd\Om} = \mathrm{id}|_{\pd\Om} $.
	
	Now, by solving the following div-curl problems:
	\begin{equation}\label{div-curl nonlinear v}
		\begin{cases*}
			\div \vv = 0 &in $ \OGt $, \\
			\curl \vv = \bar{\vom} &in $ \OGt $, \\
			\vv_\pm \vdot \vn_+ = \vn_+ \vdot \qty(\pd_t \gt \vnu)\circ\Phi_\Gt^{-1} &on $ \Gt $, \\
			\vv_- \vdot \wt{\vn} = 0 &on $ \pd\Om $;
		\end{cases*}
	\end{equation}
	and
	\begin{equation}\label{div-curl nonlinear h}
		\begin{cases*}
			\div \vh = 0 &in $ \OGt $, \\
			\curl \vh = \bar{\vj} &in $ \OGt $, \\
			\vh_\pm \vdot \vn_+ = 0 &on $ \Gt $, \\
			\vh_- \vdot \wt{\vn} = 0 &on $ \pd\Om $,
		\end{cases*}
	\end{equation}
	one can obtain the corresponding velocity and magnetic fields $ \vv_\pm, \vh_\pm : \Om_t^\pm \to \R^3 $. Furthermore, the following estimate holds thanks to Theorem \ref{thm div-curl}:
	\begin{equation}
		\sup_{t \in [0, T]}\qty(\norm{\vv_\pm}_{H^{\kk}(\Om_t^\pm)}, \norm{\vh_\pm}_{H^{\kk}(\Om_t^\pm)}) \le Q(M_1),
	\end{equation}
	where $ Q $ is a generic polynomial determined by $ \Lambda_* $.
	
	
	\subsection{Iteration map}\label{sec itetarion map}
	For $ \qty(\itn{\ka}, \itn{\vom_*}, \itn{\vj_*}) \in \mathfrak{X} $ and $ \qty\big{\ini{\ka}, \ini{\pd_t\ka}, \ini{\vom_*}, \ini{\vj_*}} \in \mathfrak{I}(\epsilon, A) $,
	define the $ (n+1)$-th step by solving the following initial value problems:
	\begin{equation}\label{eqn (n+1)ka}
		\begin{cases*}
			\pd^2_{tt}\itm{\ka} + \opC\qty(\itn{\ka}, \itn{\pd_t\ka}, \itn{\vv_*}, \itn{\vh_*})\itm{\ka} \\ \qquad = \opF\qty(\itn{\ka})\pd_t \itn{\vom_*} + \opG\qty(\itn{\ka}, \itn{\pd_t\ka}, \itn{\vom_*}, \itn{\vj_*}) \\
			\itm{\ka}(0) = \ini{\ka}, \quad \itm{\pd_t\ka}(0) = \ini{\pd_t\ka};
		\end{cases*}
	\end{equation}
	and
	\begin{equation}\label{eqn linear (n+1) vom vj}
		\begin{cases*}
			\pd_t\itm{\vom} + \DD_{\itn{\vv}}\itm{\vom} - \DD_{\itn{\vh}}\itm{\vj} = \DD_{\itm{\vom}}\itn{\vv} - \DD_{\itm{\vj}}\itn{\vh}, \\
			\pd_t\itm{\vj} + \DD_{\itn{\vv}}\itm{\vj} - \DD_{\itn{\vh}}\itm{\vom} \\ \qquad =  \DD_{\itm{\vj}}\itn{\vv} - \DD_{\itm{\vom}}\itn{\vh}
			- 2\tr(\grad\itn{\vv} \cp \grad \itn{\vh}), \\
			\itm{\vom}(0) = \Pb \qty(\ini{\vom_*} \circ \X_{\itn{\Gamma_0}}^{-1}), \quad \itm{\vj}(0) = \Pb\qty(\ini{\vj_*}\circ \X_{\itn{\Gamma_0}}^{-1}),
		\end{cases*}
	\end{equation}
	where $ \qty(\itn{\vv}, \itn{\vh}) $ is induced by $ \qty(\itn{\ka}, \itn{\vom_*}, \itn{\vj_*}) $ via solving \eqref{div-curl nonlinear v}-\eqref{div-curl nonlinear h}, the tangential vector fields $ \itn{\vv_*} $ and $ \itn{\vh_*} $ on $ \Gs $ are defined by
	\begin{equation}
		\begin{split}
			\itn{\vv_{*\pm}} &\coloneqq \qty(\DD\Phi_{\itn{\Gt}})^{-1}\qty[\itn{\vv_\pm}\circ\Phi_{\itn{\Gt}} - \qty(\pd_t\gamma_{\itn{\Gt}})\vnu], \\
			\itn{\vh_{*\pm}} &\coloneqq \qty(\DD\Phi_{\itn{\Gt}})^{-1} \qty(\itn{\vh_\pm} \circ \Phi_{\itn{\Gt}}),
		\end{split}
	\end{equation}
	and the current-vorticity equations are considered in the domain $ \Om \setminus \itn{\Gt} $.
	
	Denoting by
	\begin{equation}
		\itm{\vom_*}\coloneqq \itm{\vom}\circ\X_{\itn{\Gt}}, \qand \itm{\vj_*}\coloneqq \itm{\vj}\circ\X_{\itn{\Gt}},
	\end{equation}
	we will show that $ \qty(\itm{\ka}, \itm{\vom_*}, \itm{\vj_*}) \in \mathfrak{X} $.
	Indeed, in view of Lemma \ref{lem 3.4}, there hold
	\begin{equation*}
		\begin{split}
			&\hspace{-2em}\abs{\opF\qty(\itn{\ka})\itn{\pd_t\vom_*}}_{\H{\kk-\frac{5}{2}}} \\
			\le\, &Q\qty(\abs{\itn{\ka}}_{\H{\kk-1}})\norm{\itn{\pd_t\vom_*}}_{H^{\kk-\frac{5}{2}}(\OGs)} \\
			\le\, &Q(M_1) \cdot M_2,
		\end{split} 
	\end{equation*}
	and
	\begin{equation*}
		\begin{split}
			&\hspace{-2em}\abs{\opG\qty(\itn{\ka}, \itn{\pd_t\ka}, \itn{\vom_*}, \itn{\vj_*})}_{\H{\kk-\frac{5}{2}}} \\
			\le\, &a^2 Q\qty(\abs{\itn{\ka}}_{\H{\kk-1}}, \abs{\itn{\pd_t\ka}}_{\H{\kk-\frac{5}{2}}}, \norm{\qty(\itn{\vom_*}, \itn{\vj_*})}_{H^{\kk-1}(\OGs)}) \\
			\le\, &a^2 Q(M_1).
		\end{split}
	\end{equation*}
	Furthermore, by the definition of constants $ L_1, L_2 $ in \textsection~\ref{sec linear ka} and Lemma \ref{lem3.1}, one has
	\begin{equation*}
		\begin{split}
			L_1 \le\, &\sup_{t \in [0, T]} \qty(\abs{\itn{\ka}}_{\H{\kk-1}}, \abs{\itn{\pd_t\ka}}_{\H{\kk-\frac{5}{2}}}) \\
			&\quad +\sup_{t \in [0, T]} \qty(\abs{\itn{\vv_{\pm*}}}_{\H{\kk-\frac{1}{2}}}, \abs{\itn{\vh_{\pm*}}}_{\H{\kk-\frac{1}{2}}}) \\
			\le\, &Q(M_1),
		\end{split}
	\end{equation*}
	and
	\begin{equation*}
		L_2 = \sup_{t \in [0, T]} \qty(\abs{\itn{\pd_t\vv_{\pm*}}}_{\H{\kk-2}}, \abs{\itn{\pd_t\vh_{\pm*}}}_{\H{\kk-2}} ) \le Q\qty(M_1, M_2, a^2 M_3).
	\end{equation*}
	Thus, by taking $ l=k-2 $ in (\ref{est linear eqn ka}), it follows that
	\begin{equation*}
		\begin{split}
			&\sup_{t \in [0, T]}\qty(\abs{\itm{\ka}}_{\H{\kk-1}} + \abs{\itm{\pd_t\ka}}_{\H{\kk-\frac{5}{2}}}) \\
			&\le C_* \exp\qty{Q\qty(M_1, M_2, a^2M_3)T}   \times \qty[C_* + \epsilon + A + (M_0)^{12} + T\cdot \qty(a^2+M_2)Q(M_1)].
		\end{split}
	\end{equation*}
	If  $ T $ is taken small enough, and $ M_1 $ is much larger than $ M_0 $ and $ A $, then there holds
	\begin{equation}
		\sup_{t \in [0, T]}\qty(\abs{\itm{\ka}}_{\H{\kk-1}}+ \abs{\itm{\pd_t\ka}}_{\H{\kk-\frac{5}{2}}}) \le M_1.
	\end{equation}		
	Moreover, choosing $ M_3 $ large enough compared to $ M_1$ and $M_2 $, one gets from (\ref{eqn (n+1)ka}) that 
	\begin{equation}
		\sup_{t \in [0, T]} \abs{\itm{\pd^2_{tt}\ka}}_{\H{\kk-4}} \le a^2M_3.
	\end{equation}
	
	Similarly, by applying Proposition \ref{prop linear vom vj} to (\ref{eqn linear (n+1) vom vj}), one can deduce that
	\begin{equation}
		\begin{split}
			&\hspace{-2em}\sup_{t \in [0, T]}\qty(\norm{\itm{\vom_*}}_{H^{\kk-1}(\OGs)}, \norm{\itm{\vj_*}}_{H^{\kk-1}(\OGs)}) \\
			\le\, &Q\qty(\abs{\itn{\ka}}_{C^0_t\H{\kk-\frac{5}{2}}}) e^{Q(M_1)T}(1+2A) \\
			\le\, &M_1,
		\end{split}
	\end{equation}
	if $ T $ is small and $ M_1 \gg \abs{\kappa_{*+}}_{\H{\kk-\frac{5}{2}}} $. Next, by taking $ M_2 \gg M_1 $, it holds that
	\begin{equation*}
		\begin{split}
			&\hspace{-2em}\norm{\itm{\pd_t\vom_*}(t)}_{H^{\kk-2}(\OGs)} \\
			\le\, &C_* \qty(\norm{\itm{\pd_t\vom}}_{H^{\kk-2}(\Om\setminus\itn{\Gt})} + \norm{\itm{\vom}}_{H^{\kk-1}(\Om\setminus\itn{\Gt})}\abs{\itn{\pd_t\ka}}_{\H{\kk-\frac{5}{2}}} ) \\
			\le\,  &Q(M_1) \le M_2
		\end{split}
	\end{equation*}
	for all $ 0 \le t \le T $. Similarly,
	\begin{equation}
		\sup_{t \in [0, T]} \qty(\norm{\itm{\pd_t\vom_*}}_{H^{\kk-2}(\OGs)}, \norm{\itm{\pd_t\vj_*}}_{H^{\kk-2}(\OGs)}) \le M_2.
	\end{equation}
	
	The compatibility condition \eqref{eqn compati} follows from Lemma \ref{lem compatibility}. 
	
	With the following notation:
	\begin{equation}
		\begin{split}
			\mathfrak{T}\qty{\qty[\ini{\ka}, \ini{\pd_t\ka}, \ini{\vom_*}, \ini{\vj_*}], \qty[\itn{\ka}, \itn{\vom_*}, \itn{\vj_*}]}
			\coloneqq \qty[\itm{\ka}, \itm{\vom_*}, \itm{\vj_*}],
		\end{split}
	\end{equation}
	one can conclude from the previous arguments that:
	\begin{prop}
		Suppose that $ k \ge 2 $. For any $ 0 < \epsilon \ll \delta_0 $ and $ A > 0 $, there are positive constants $ M_0, M_1, M_2, M_3 $,  so that for small $ T > 0 $,
		\begin{equation*}
			\mathfrak{T}\qty\Big{\qty[\ini{\ka}, \ini{\pd_t\ka}, \ini{\vom_*}, \ini{\vj_*}], \qty[\ka, \vom_*, \vj_*]} \in \mathfrak{X},
		\end{equation*}
		holds for any $ \qty[\ini{\ka}, \ini{\pd_t\ka}, \ini{\vom_*}, \ini{\vj_*}] \in \mathfrak{I}(\epsilon, A) $ and $ \qty[\ka, \vom_*, \vj_*] \in \mathfrak{X} $.
	\end{prop}
	
	\subsection{Contraction of the iteration map}\label{sec contra ite}
	In this subsection, it is always assumed that $ k \ge 3 $. Suppose that there are two one-parameter families $ \qty(\itn{\ka}(\beta), \itn{\vom_*}(\beta), \itn{\vj_*}(\beta)) \subset \mathfrak{X} $ and \\ $ \qty\Big(\ini{\ka}(\beta), \ini{\pd_t\ka}(\beta), \ini{\vom_*}(\beta), \ini{\vj_*}(\beta)) \subset \mathfrak{I}(\epsilon, A)$ with parameter $ \beta $. Define 
	\begin{equation*}
		\begin{split}
			&\hspace{-2em}\qty(\itm{\ka}(\beta), \itm{\vom_*}(\beta), \itm{\vj_*}(\beta)) \\
			\coloneqq\, &\mathfrak{T}\qty{\qty\Big(\ini{\ka}(\beta), \ini{\pd_t\ka}(\beta), \ini{\vom_*}(\beta), \ini{\vj_*}(\beta)), \qty(\itn{\ka}(\beta), \itn{\vom_*}(\beta), \itn{\vj_*}(\beta))}.
		\end{split}
	\end{equation*}
	Then by applying $ \pdv{\beta} $ to (\ref{eqn (n+1)ka}) and (\ref{eqn linear (n+1) vom vj}) respectively, one gets
	\begin{equation}\label{eqn contra 1}
		\begin{cases*}
			\pd^2_{tt} \pd_\beta \itm{\ka} + \itn{\opC}\pd_\beta\itm{\ka}  = - \qty(\pd_\beta\itn{\opC})\itm{\ka} + \pd_\beta \qty(\itn{\opF}\itn{\pd_t\vom_*} + \itn{\opG}), \\
			\pd_\beta\itm{\ka}(0) = \pd_\beta\ini{\ka}(\beta), \quad \pd_t\qty(\pd_\beta\itm{\ka})(0) = \pd_\beta\ini{\pd_t\ka}(\beta),
		\end{cases*}
	\end{equation}
	and
	\begin{equation}\label{eqn beta n+1}
		\begin{cases*}
			\pd_t\Dbt\itm{\vom} + \DD_{\itn{\vv}}\Dbt\itm{\vom} - \DD_{\itn{\vh}}\Dbt\itm{\vj} = \DD_{\Dbt\itm{\vom}}\itn{\vv} - \DD_{\Dbt\itm{\vj}}\itn{\vh} + \va{\g}_1, \\
			\pd_t \Dbt\itm{\vj} + \DD_{\itn{\vv}}\Dbt\itm{\vj} - \DD_{\itn{\vh}}\Dbt\itm{\vom} = \DD_{\Dbt\itm{\vj}}\itn{\vv} - \DD_{\Dbt\itm{\vom}}\itn{\vh} + \va{\g}_2, \\
			\Dbt\itm{\vom}(0) = \Pb\qty{\qty[\pd_\beta\ini{\vom_*}] \circ \X_{\itn{\Gamma_0}(\beta)}^{-1}},  \quad \Dbt\itm{\vj}(0) = \Pb\qty{\qty[\pd_\beta\ini{\vj_*}]\circ\X_{\itn{\Gamma_0}(\beta)}^{-1}},
		\end{cases*}
	\end{equation}
	where
	\begin{equation}
		\Dbt \coloneqq \pdv{\beta} + \DD_{\vb*{\mu}} \qc \vb*{\mu} \coloneqq \h\qty[\qty(\pd_\beta\gamma_{\itn{\Gt}}\vnu)\circ\Phi_{\itn{\Gt}(\beta)}^{-1}],
	\end{equation}
	\begin{equation}
		\begin{split}
			\va{\g}_1 \coloneqq\, &\comm{\pd_t}{\Dbt}\itm{\vom} + \comm{\DD_{\itn{\vv}}}{\Dbt}\itm{\vom} - \comm{\DD_{\itn{\vh}}}{\Dbt}\itm{\vj} \\
			&+\itm{\vom} \vdot \Dbt\DD\itn{\vv} - \itm{\vj} \vdot \Dbt\DD\itn{\vh},
		\end{split}
	\end{equation}
	and
	\begin{equation}\label{eqn g2}
		\begin{split}
			\va{\g}_2 \coloneqq\, &-2\Dbt\tr(\grad\itn{\vv} \cp \grad\itn{\vh}) + \comm{\pd_t}{\Dbt}\itm{\vj} + \comm{\DD_{\itn{\vv}}}{\Dbt}\itm{\vj} \\
			&-\comm{\DD_{\itn{\vh}}}{\Dbt}\itm{\vom}+ \itm{\vj}\vdot \Dbt\DD\itn{\vh} - \itm{\vom}\vdot\Dbt\DD\itn{\vh}.
		\end{split}
	\end{equation}
	To estimate the Lipschitz constant for the iteration map $ \mathfrak{T} $, we consider the following energy functionals:
	\begin{equation}
		\begin{split}
			\itn{\E}(\beta)\coloneqq &\sup_{t \in [0, T]}\left(\abs{\pd_\beta\itn{\ka}}_{\H{\kk-\frac{5}{2}}} + \abs{\pd_\beta\itn{\pd_t\ka}}_{\H{\kk-4}} + \right. \\
			&\qquad\qquad \left.+ \norm{\pd_\beta\itn{\vom_*}}_{H^{\kk-\frac{5}{2}}(\OGs)} + \norm{\pd_\beta\itn{\vj_*}}_{H^{\kk-\frac{5}{2}}(\OGs)}   \right),
		\end{split}
	\end{equation}
	and
	\begin{equation}
		\begin{split}
			\ini{\E}(\beta)\coloneqq\, &\left(\abs{\pd_\beta\ini{\ka}}_{\H{\kk-\frac{5}{2}}} + \abs{\pd_\beta\ini{\pd_t\ka}}_{\H{\kk-4}} + \right. \\
			&\qquad \left.+ \norm{\pd_\beta\ini{\vom_*}}_{H^{\kk-\frac{5}{2}}(\OGs)} + \norm{\pd_\beta\ini{\vj_*}}_{H^{\kk-\frac{5}{2}}(\OGs)}   \right).
		\end{split}
	\end{equation}
	
	Thanks to Lemmas \ref{lem3.1}, \ref{lem 3.3}, and \ref{lem 3.4}, it holds that
	\begin{equation*}
		\begin{split}
			\abs{\qty(\pd_\beta\itn{\opC})\itm{\ka}}_{\H{\kk-4}}
			\le Q(M_1) \itn{\E}\abs{\itm{\ka}}_{\H{\kk-2}} \le Q(M_1) \itn{\E},
		\end{split}
	\end{equation*}
	\begin{equation*}
		\abs{\opF\qty(\itn{\ka})\pd_\beta\pd_t\itn{\vom_*}}_{\H{\kk-4}} \le Q(M_1) \norm{\pd_\beta\pd_t \itn{\vom_*}}_{H^{\kk-4}(\OGs)},
	\end{equation*}
	and
	\begin{equation*}
		\begin{split}
			&\hspace{-2em}\abs{\qty(\pd_\beta\itn{\opF})\pd_t\itn{\vom_*} + \pd_\beta\itn{\opG}}_{\H{\kk-4}} \\
			\le\, &C_* \abs{\pd_\beta\itn{\ka}}_{\H{\kk-\frac{5}{2}}}\norm{\pd_t\itn{\vom_*}}_{H^{\kk-\frac{7}{2}}(\OGs)} + a^2 Q(M_1) \itn{\E} \\
			\le\, &\qty(C_*M_2 + a^2 Q(M_1)) \itn{\E}.
		\end{split}
	\end{equation*}
	Taking $ l=k-3 $ in (\ref{est linear eqn ka}) leads to
	\begin{equation}\label{est beta ka(n+1)}
		\begin{split}
			&\hspace{-1em}\sup_{t \in [0, T]} \qty(\abs{\pd_\beta\itm{\ka}}_{\H{\kk-\frac{5}{2}}} + \abs{\pd_\beta\itm{\pd_t\ka}}_{\H{\kk-4}}) \\
			\le\, &Q(M_1) \exp{Q\qty(M_1, M_2, a^2M_3)T} \\
			& \times\qty(\ini{\E} + T \cdot \qty[C_*M_2 + \qty(1+a^2)Q(M_1)]\itn{\E} +T\cdot Q(M_1)\sup_{t \in [0, T]}\norm{\pd_\beta\pd_t\itn{\vom_*}}_{H^{\kk-4}(\OGs)}).
		\end{split}
	\end{equation}
	To estimate (\ref{eqn beta n+1}), one can derive that
	\begin{equation*}
		\begin{split}
			&\norm{\va{\g}_1}_{H^{\kk-\frac{5}{2}}(\Om\setminus\itn{\Gt})} \\
			&\quad \le Q(M_1) \qty(\abs{\pd_t\pd_\beta\itn{\ka}}_{\H{\kk-4}} + \norm{\qty(\Dbt \itn{\vv}, \Dbt\itn{\vh})}_{H^{\kk-\frac{3}{2}}(\Om\setminus\itn{\Gt})} + \abs{\pd_\beta \ka}_{\H{\kk-\frac{5}{2}}}), 
		\end{split}
	\end{equation*}
	which, together with \eqref{est pd beta v+*}, implies
	\begin{equation*}
		\norm{\va{\g}_1}_{H^{\kk-\frac{5}{2}}(\Om\setminus\itn{\Gt})} \le Q(M_1) \itn{\E}.
	\end{equation*}
	Similarly, one can obtain
	\begin{equation*}
		\norm{\va{\g}_2}_{H^{\kk-\frac{5}{2}}(\Om\setminus\itn{\Gt})} \le Q(M_1) \itn{\E}.
	\end{equation*}
	It follows from the same arguments as in Proposition \ref{prop linear vom vj} that
	\begin{equation}\label{est dbt vom vj (n+1)}
		\begin{split}
			&\hspace{-2em}\sup_{t \in [0, T]}\qty(\norm{\Dbt\itm{\vom}}_{H^{\kk-\frac{5}{2}}(\Om\setminus\itn{\Gt})} + \norm{\Dbt\itm{\vj}}_{H^{\kk-\frac{5}{2}}(\Om\setminus\itn{\Gt})}) \\
			\le\, &e^{Q(M_1)T} \qty{\ini{\E} + TQ(M_1)\itn{\E}},
		\end{split}
	\end{equation}
	and thus
	\begin{equation*}
		\norm{\pd_t\Dbt\itm{\vom}}_{H^{\kk-4}(\Om\setminus\itn{\Gt})} \le e^{Q(M_1)T} Q(M_1)\qty{\ini{\E} + TQ(M_1)\itn{\E}}.
	\end{equation*}
	Since
	\begin{equation*}
		\qty[\pd_t\pd_\beta\itm{\vom_*}]\circ\X_{\itn{\Gt}}^{-1} = \pd_t\Dbt\itm{\vom} + \DD_{\h\qty[\qty(\pd_t\itn{\gt}\vnu)\circ\Phi_{\itn{\Gt}}^{-1}]}\Dbt\itm{\vom},
	\end{equation*}
	one has
	\begin{equation}\label{est pd beta t vom (n+1)}
		\norm{\pd_\beta\pd_t\itm{\vom_*}}_{H^{\kk-4}(\OGs)} \le e^{Q(M_1)T} Q(M_1)\qty{\ini{\E} + TQ(M_1)\itn{\E}}.
	\end{equation}
	Set
	\begin{equation}
		\itn{\mathfrak{F}}\coloneqq\norm{\pd_\beta\pd_t\itn{\vom_*}}_{H^{\kk-4}(\OGs)}.
	\end{equation}
	Then (\ref{est beta ka(n+1)})-(\ref{est pd beta t vom (n+1)}) imply that
	\begin{equation}
		\begin{split}
			\itm{\E} \le\, &C_* \exp\qty{Q\qty(M_1, M_2, a^2M_3)T} \\
			&\quad\times \qty{\ini{\E} + T \cdot \qty[M_2\itn{\E} + \qty(1+a^2)Q(M_1)\itn{\E} + Q(M_1)\itn{\mathfrak{F}}]} \\
			&+ e^{Q(M_1)T} Q(M_1)\qty{\ini{\E} + TQ(M_1)\itn{\E}},
		\end{split}
	\end{equation}
	and
	\begin{equation}
		\itm{\mathfrak{F}} \le e^{Q(M_1)T} Q(M_1)\qty{\ini{\E} + TQ(M_1)\itn{\E}}.
	\end{equation}
	Thus, if $ T $ is small compared to $ M_1, M_2, M_3 $ and $ a $, then
	\begin{equation}
		\itm{\E} + \itm{\mathfrak{F}} \le \dfrac{1}{2}\qty(\itn{\E}+\itn{\mathfrak{F}}) + Q(M_1)\ini{\E}.
	\end{equation}
	
	An immediate consequence is that if the initial data is fixed, the iteration map $ \mathfrak{T} $ is a contraction in a space containing $ \mathfrak{X} $. Since $ \mathfrak{T} $  is a map from $ \mathfrak{X} $ to $ \mathfrak{X} $, it follows that $ \mathfrak{T} $ has a unique fixed point on $ \mathfrak{X} $. Namely, one obtains:
	\begin{prop}\label{fixed point}
		Assume that $ k \ge 3 $. For any $ 0 < \epsilon \ll \delta_0 $ and $ A > 0 $, there are positive constants $ M_0, M_1, M_2, M_3 $ so that if $ T $ is small enough, then there is a map $ \mathfrak{S} : \mathfrak{I}(\epsilon, A) \to  \mathfrak{X} $ such that
		\begin{equation}
			\mathfrak{T}\qty{\mathfrak{x}, \mathfrak{S(x)}} = \mathfrak{S(x)}
		\end{equation}
		for each $ \mathfrak{x} = \qty\Big(\ini{\ka}, \ini{\pd_t\ka}, \ini{\vom_*}, \ini{\vj_*}) \in \mathfrak{I}(\epsilon, A) $.
	\end{prop}
	
	\subsection{The original nonlinear MHD problem}\label{sec back to MHD}
	For any given initial data $ \Gamma_0 \in H^{\kk+1} $ and $ \vv_0, \vh_0 \in H^{\kk}(\Om\setminus\Gamma_0) $, one can construct $$ \qty\Big(\ka(0), \pd_t\ka(0), \vom_*(0), \vj_*(0)). $$ Indeed, for a reference hypersurface $ \Gs \in H^{\kk+1} $ close enough to $ \Gamma_0 $ (or $ \Gs = \Gamma_0 $) and a transversal field $ \vnu \in \H{\kk-1} $, $ \ka(0) $ can be given by $ \Gamma_0 $, and $ \pd_t\ka(0) $ is determined by $ \theta_0 = \vv_{0\pm}\vdot\vn_+ $. In addition, let
	\begin{equation}
		\vom_* \coloneqq \qty(\curl \vv_0) \circ \X_{\Gamma_0}^{-1}, \quad \vj_*\coloneqq\qty(\curl \vh_0) \circ \X_{\Gamma_0}^{-1}.
	\end{equation}
	Thus, $ \ka(0) \in \H{\kk-1}, \pd_t\ka(0) \in \H{\kk-\frac{5}{2}} $, and $ \vom_*(0), \vj_*(0)  \in H^{\kk-1}(\OGs)$.
	
	Let $ \qty{\ini{\ka}, \ini{\pd_t\ka}, \ini{\vom_*}, \ini{\vj_*}} \coloneqq \qty{\ka(0), \pd_t\ka(0), \vom_*(0), \vj_*(0)} $, and take the corresponding fixed point $ \qty{\ka(t), \vom_*(t), \vj_*(t)} \in \mathfrak{X} $ of the map $ \mathfrak{S} : \mathfrak{I}(\epsilon, A) \to \mathfrak{X} $ given in  Proposition~\ref{fixed point}. Thus, $ \qty(\Gt, \vv, \vh) $ can be obtained as discussed in \textsection~\ref{section recovery}. We will show that
	the induced quantity $ \qty(\Gt, \vv, \vh) $ is a solution to the (MHD)-(BC) problem with initial data $ \qty(\Gamma_0, \vv_0, \vh_0) $.
	
	Indeed, it is clear that $ \Gamma(0) = \Gamma_0 $, $ \vv(0) = \vv_0$, and $\vh(0) = \vh_0 $ by the definition and the uniqueness of div-curl systems.
	
	First, we claim that
	\begin{equation}
		\Pb \qty(\vom_*(t) \circ \X_\Gt^{-1}) = \vom_*(t) \circ \X_\Gt^{-1}, \quad \Pb \qty(\vj_*(t) \circ \X_\Gt^{-1}) = \vj_*(t) \circ \X_\Gt^{-1}.
	\end{equation}
	Indeed, taking the divergence of (\ref{eqn linear (n+1) vom vj}) and using the fact that $ \div \vv \equiv 0 \equiv \div\vh $ yield
	\begin{equation}
		\begin{cases*}
			\pd_t (\div\vom) + \DD_{\vv}(\div \vom) = \DD_{\vh} (\div\vj), \\
			\pd_t (\div\vj) + \DD_{\vv}(\div\vj) = \DD_{\vh} (\div\vom),
		\end{cases*}
	\end{equation}
	where
	\begin{equation*}
		\vom(t) \coloneqq \vom_*(t) \circ \X_\Gt^{-1}, \quad \vj(t)\coloneqq \vj_*(t) \circ \X_\Gt^{-1}.
	\end{equation*}
	Since $ \div \vom(0) = 0 = \div \vj(0) $, it follows from the arguments in \textsection~\ref{sec linear current vortex} that $ \div\vom \equiv 0 \equiv \div \vj $ for all $ t $, which proves the claim.
	
	Consequently,
	\begin{equation}\label{eqn curl vv vh}
		\curl\vv = \vom \qand \curl \vh = \vj.
	\end{equation}
	
	Next, as in (\ref{decomp p}), define the pressure functions via
	\begin{equation*}
		p^{\pm} = \rho_\pm \qty(p_{\vv, \vv}^{\pm} - p_{\vh, \vh}^{\pm} + p_\kappa^\pm) + \h_\pm \bar{\n}^{-1}(-g^+ + g^-),
	\end{equation*}
	with $ p^\pm_\kappa $ defined by (\ref{def p_kappa}) and
	\begin{equation*}
		g^{\pm} \coloneqq 2 \DD_{\vv_\pm^\top}\theta - \II_+\qty(\vv_\pm^\top, \vv_\pm^\top) + \II_+\qty(\vh_\pm^\top, \vh_\pm^\top) + \DD_{\vn_+}\qty(p_{\vv, \vv}^\pm - p_{\vh, \vh}^\pm).
	\end{equation*}
	
	Inspired by \cite{Sun-Wang-Zhang2018}, define
	\begin{equation}
		\vV_\pm \coloneqq \pd_t \vv_\pm + \DD_{\vv_\pm}\vv_\pm + \dfrac{1}{\rho_\pm}\grad p^\pm - \DD_{\vh_\pm}\vh_\pm,
	\end{equation}
	and
	\begin{equation}
		\vH_\pm \coloneqq \pd_t \vh_\pm + \DD_{\vv_\pm}\vh_\pm - \DD_{\vh_\pm}\vv_\pm.
	\end{equation}
	It suffices to show that $ \vV \equiv \vb{0} \equiv \vH $ for $ 0 < t < T $. Indeed, since $ \Om_t^\pm $ are both assumed to be simply-connected, one only needs to verify for $ \vb{Z} = \vV$ or $ \vH $:
	\begin{equation}
		\begin{cases*}
			\div \vb{Z}_\pm = 0 &in $ \Om_t^\pm $, \\
			\curl \vb{Z}_\pm = \vb{0} &in $ \Om_t^\pm $, \\
			\vb{Z}_\pm \vdot \vn_+ = 0 &on $ \Gt $, \\
			\vb{Z}_- \vdot \wt{\vn} = 0 &on $ \pd\Om $.
		\end{cases*}
	\end{equation}
	
	\subsubsection*{Verification of $ \vV $}
	Observe that
	\begin{equation*}
		\div \vv \equiv 0 \equiv \div\vh.
	\end{equation*}
	Then it follows from the definitions of $ p^\pm_{\vb{a}, \vb{b}} $ that,
	\begin{equation}
		\div \vV \equiv 0.
	\end{equation}
	Taking curl of $ \vV $ and using (\ref{eqn curl vv vh}) with (\ref{eqn pdt vom}) lead to
	\begin{equation}
		\curl \vV = \pd_t \vom + \DD_{\vv}\vom - \DD_{\vom}\vv - \DD_{\vh}\vj + \DD_{\vj}\vh = \vb{0}.
	\end{equation}
	In addition, it follows from (\ref{def p_ab^-}) that on $ \pd\Om $:
	\begin{equation}
		\vV_- \vdot \wt{\vn} =  - \wt{\II}(\vv_-, \vv_-) + \wt{\II}(\vh_-, \vh_-) + \DD_{\wt{\vn}}(p^-_{\vv, \vv} - p^-_{\vh, \vh}) = 0. 
	\end{equation}
	Thus, it remains to show that $ \vV_\pm \vdot \vn_+ \equiv 0 $. Note that for $ \theta\coloneqq \vv_\pm \vdot \vn_+ $ and $ \ev{g^\pm}\coloneqq\fint_\Gt g^\pm \dd{S_t} $, there hold
	\begin{equation}
		\begin{split}
			&\hspace{-1em}\vV_+ \vdot \vn_{+} \\
			=\, &\vn_+ \vdot \qty(\Dt_+ \vv_+ - \DD_{\vh_+}\vh_+) + \DD_{\vn_+}\qty(p_{\vv, \vv}^+ - p_{\vh, \vh}^+ + p_\kappa^+) - \qty(\dfrac{1}{\rho_+}\n_+)\bar{\n}^{-1}(g^+ - g^-) \\
			=\, &\Dt_+ \theta + \vn_+ \vdot (\DD\vv_+)\vdot \vv_+^\top + \II_+(\vh_+, \vh_+) + \DD_{\vn_+}\qty(p_{\vv, \vv}^+ - p_{\vh, \vh}^+) \\
			&+ \wt{\n}\kappa_+ - \qty(\dfrac{1}{\rho_+}\n_+)\bar{\n}^{-1}(g^+ - g^-) \\
			=\, &\Dt_+ \theta + \vn_+ \vdot \DD_{\vv_+^\top}\vv_+ + \II_+(\vh_+, \vh_+) + \DD_{\vn_+}\qty(p_{\vv, \vv}^+ - p_{\vh, \vh}^+) \\
			&+ \wt{\n}\kappa_+  + \qty(\dfrac{1}{\rho_-}\n_-)\bar{\n}^{-1}g^+ + \qty(\dfrac{1}{\rho_+}\n_+)\bar{\n}^{-1}g^- -g^+ + \ev{g^+} \\
			=\, &\Dt_+ \theta - \DD_{\vv_+^\top}\theta + \wt{\n}\kappa_+ + \qty(\dfrac{1}{\rho_-}\n_-)\bar{\n}^{-1}g^+ + \qty(\dfrac{1}{\rho_+}\n_+)\bar{\n}^{-1}g^- + \ev{g^+},
		\end{split} 
	\end{equation}
	and
	\begin{equation}
		\begin{split}
			&\hspace{-2em}\vV_- \vdot \vn_- \\
			=\, &-\Dt_- \theta + \vn_- \vdot (\DD\vv_-) \vdot \vv_- + \II_-(\vh_-, \vh_-) + \DD_{\vn_-}(p_{\vv, \vv}^- - p_{\vh, \vh}^-) \\
			&-\wt{\n}\kappa_+ - \qty(\dfrac{1}{\rho_-}\n_-) \bar{\n}^{-1}(g^+ - g^-) \\
			=\, &-\Dt_- \theta + \vn_- \vdot \DD_{\vv_-^\top}\vv_- + \II_-(\vh_-, \vh_-) + \DD_{\vn_-}(p_{\vv, \vv}^- - p_{\vh, \vh}^-) \\
			&-\wt{\n}\kappa_+ - \qty(\dfrac{1}{\rho_-}\n_-) \bar{\n}^{-1}g^+ - \qty(\dfrac{1}{\rho_+}\n_+)\bar{\n}^{-1}g^- + g^- - \ev{g^-} \\
			=\, &-\Dt_- \theta + \DD_{\vv_-^\top}\theta -\wt{\n}\kappa_+ - \qty(\dfrac{1}{\rho_-}\n_-)\bar{\n}^{-1}g^+ - \qty(\dfrac{1}{\rho_+}\n_+)\bar{\n}^{-1}g^- - \ev{g^-}.
		\end{split}
	\end{equation}
	Hence,
	\begin{equation}
		\begin{split}
			\vV_+ \vdot \vn_+ + \vV_- \vdot \vn_- &= \Dt_+ \theta - \Dt_- \theta - \DD_{\vv_+^\top}\theta + \DD_{\vv_-^\top}\theta + \ev{g^+ - g^-} \\
			&= \DD_{\vv_+ - \vv_-}\theta - \DD_{\vv_+^\top -\vv_-^\top}\theta + \ev{g^+ - g^-} \\
			&=0,
		\end{split}
	\end{equation}
	where the last equality follows from (\ref{int g^+ - g^-}) and the relation that $ \vv_+ \vdot \vn_+ = \vv_- \vdot \vn_+ $.
	Therefore, one can define
	\begin{equation}
		\Theta \coloneqq \vV_+ \vdot \vn_+ = \vV_- \vdot \vn_+.
	\end{equation}

	For $ \vW $ defined via (\ref{def vW}), the relation that $ \div \vv \equiv 0 $ implies $ q^\pm = p^\pm $, that is,
	\begin{equation}\label{eqn vW vV}
		\vW = \dfrac{\rho_+}{\rho_+ + \rho_-}\vV_+ + \dfrac{\rho_-}{\rho_+ + \rho_-} \vV_- \qq{on} \Gt.
	\end{equation}
	Thus,
	\begin{equation}
		\Theta = \vW \vdot \vn_+.
	\end{equation}
	Because $ \qty{\ka, \vom_*, \vj_*} \in \mathfrak{X} $  is a fixed point of $ \mathfrak{T} $, by (\ref{eqn pd2 tt ka}), it holds that
	\begin{equation}\label{last eqn vW}
		-\lap_\Gt \qty(\vW \vdot \vn_+) + \vW \vdot \lap_\Gt \vn_+ + a^2 \dfrac{\vW \vdot \vn_+}{\vn_+ \vdot (\vnu \circ \Phi_\Gt^{-1})} = 0.
	\end{equation}
	In addition, since $ \curl \vV = \vb{0}, \div \vV = 0 $, $ \vV_- \vdot \wt{\vn} = 0 $ and $ \Om_t^\pm $ are both simply-connected, there are two mean-zero functions $ r^\pm(t, x') : \Gt \to \R $ so that
	\begin{equation}
		\vV_\pm = \dfrac{1}{\rho_\pm} \grad \h_\pm r^\pm,
	\end{equation}
	which implies that
	\begin{equation}\label{eqn Theta r}
		\begin{split}
			\Theta &= \vV_+ \vdot \vn_+ = \qty(\dfrac{1}{\rho_+}\n_+) r^+ 
			\\&= -\vV_- \vdot \vn_- = -\qty( \dfrac{1}{\rho_-}\n_- ) r^-.
		\end{split}
	\end{equation}
	It follows from (\ref{last eqn vW}), (\ref{eqn vW vV}), and the identity (\ref{lap vn}) that
	\begin{equation}
		-\lap_\Gt \Theta + \qty(\dfrac{a^2}{\vn_+ \vdot (\vnu\circ \Phi_\Gt^{-1}) } - \abs{\II_+}^2)\Theta + \dfrac{1}{\rho_+ + \rho_-} \grad^\top \qty(r^+ + r^-) \vdot \grad^\top \kappa_+ = 0.
	\end{equation}
	If  $ a_0 $ is taken large enough, so that $ \dfrac{a^2}{\vn_+ \vdot (\vnu\circ \Phi_\Gt^{-1}) } - \abs{\II_+}^2 > a $ holds for all $ t \in [0, T] $ whenever $ a \ge a_0 $ (indeed, it holds for all $ \Gamma \in \Lambda_* $), then
	\begin{equation}
		\begin{split}
			\abs{\grad^\top \Theta}_{L^2(\Gt)}^2 + a\abs{\Theta}_{L^2(\Gt)}^2 &\le \frac{1}{\rho_+ + \rho_-} \abs{\Theta}_{L^2(\Gt)} \cdot \abs{\grad^\top \kappa_+ \vdot \grad^\top\qty(r^+ + r^-)}_{L^2(\Gt)} \\
			&\le C_* \abs{\Theta}_{L^2(\Gt)}\abs{r^+ + r^-}_{H^{\frac{3}{2}}(\Gt)}\abs{\kappa_+}_{H^{\kk-\frac{5}{2}}(\Gt)} \\
			&\le C_{*0} \qty(\abs{\Theta}_{L^2(\Gt)}^2 + \abs{r^+ + r^-}_{H^{\frac{3}{2}}(\Gt)}^2 ).
		\end{split}
	\end{equation}
	It can be deduced directly from \eqref{eqn Theta r} that
	\begin{equation*}
		r^\pm = \pm \qty(\dfrac{1}{\rho_\pm}\n_\pm)^{-1}\Theta,
	\end{equation*}
	which implies that
	\begin{equation}
		\begin{split}
			\abs{r^+ + r^-}_{H^{\frac{3}{2}}(\Gt)}^2 &\le C_* \abs{\Theta}_{H^{\frac{1}{2}}(\Gt)}^2 \\
			&\le \frac{1}{2C_{*0}}\abs{\grad^\top{\Theta}}_{L^2(\Gt)}^2 + C_* \abs{\Theta}_{L^2(\Gt)}^2 .
		\end{split}
	\end{equation}
	Thus, for a generic constant $ C_* $ determined by $ \Lambda_* $, it holds that
	\begin{equation}
		\abs{\grad^\top\Theta}_{L^2(\Gt)}^2 + \qty(2a-C_*) \abs{\Theta}_{L^2(\Gt)}^2 \le 0.
	\end{equation}
	If $ a_0 $ is taken large enough (depending only on $ \Lambda_* $), then $ \Theta = 0 $, which yields $ \vV (t) \equiv \vb{0} $.
	
	\subsubsection*{Verification of $ \vb{H} $}
	Similar arguments show that
	\begin{equation}
		\div \vH \equiv 0, \quad \curl \vH \equiv \vb{0}.
	\end{equation}
	
	As for the boundary terms, observe that $ \vh_\pm \vdot \vn_+ \equiv 0 $, which implies
	\begin{equation}
		\begin{split}
			\vH_\pm \vdot \vn_+ &=\vn_+ \vdot\Dt_\pm \vh_\pm - \vn_+ \vdot \DD_{\vh_\pm}\vv_\pm \\
			&=-\vh_\pm\vdot\Dt_\pm \vn_+ - \vn_+ \vdot \DD_{\vh_\pm}\vv_\pm \\
			&=\vn_+ \vdot \DD_{\vh_\pm}\vv_\pm - \vn_+ \vdot \DD_{\vh_\pm}\vv_\pm \\
			&=0.
		\end{split}
	\end{equation}
	In addition, one can derive from $ \vv_- \vdot \wt{\vn} = 0 = \vh_- \vdot \wt{\vn} $  that
	\begin{equation}
		\begin{split}
			\vH_- \vdot \wt{\vn} = \pd_t \vh_- \vdot \wt{\vn} + \comm{\vh_-}{\vv_-}\vdot \wt{\vn} = 0.
		\end{split}
	\end{equation}
	Therefore, $ \vH(t) \equiv \vb{0} $.
	
	The previous arguments ensure that $ \qty(\Gt, \vv, \vh) $ is a solution to the original (MHD) system, with (BC) following from the construction. The uniqueness and the continuous dependence on the initial data of the original problem follow from those of the div-curl ones. In conclusion, Theorem \ref{thm s.t.} holds.
	
	\section{Stabilization Effect of the Syrovatskij Condition}\label{sec syro}
	
	\subsection{The strict Syrovatskij condition}
	Since the free interface is a compact surface, $ \abs*{\vh_+ \cp \vh_-} > 0 $ on $ \Gamma $ implies that $ \vh_\pm $ form a global frame of $ \Gamma $. Therefore, for any tangential vector field $ \vb{a} $ on $ \Gamma $, there is a unique decomposition:
	\begin{equation}\label{decomp syro}
		\vb{a} = a^+ \vh_+ + a^- \vh_-.
	\end{equation}
	Furthermore, the following relation holds:
	\begin{lem}\label{lem syro equiv 1}
		If \eqref{Syro 3"} holds on a compact hypersurface $ \Gamma \subset \R^3$, then for any non-vanishing tangential vector field $ \vb{a} $ on $ \Gamma $, it holds that
		\begin{equation}\label{strong syro equiv}
			\frac{\rho_+}{\rho_+ + \rho_-}\abs{\vb{a} \vdot \vh_+}^2 + \frac{\rho_-}{\rho_+ + \rho_-}\abs{\vb{a} \vdot \vh_-}^2 - \frac{\rho_+ \rho_-}{(\rho_+ + \rho_-)^2}\abs{\vb{a} \vdot \llbracket\vv\rrbracket}^2 > 0 \qq{on} \Gamma.
		\end{equation}
	\end{lem}
	\begin{proof}
		For simplicity, we shall use the notations:
		\begin{equation}
			g_{++}\coloneqq \vh_+ \vdot \vh_+ \qc g_{--}\coloneqq \vh_- \vdot \vh_- \qc g_{+-} \equiv g_{-+} \coloneqq \vh_- \vdot \vh_+,
		\end{equation}
		the decomposition \eqref{decomp syro}, and
		\begin{equation}
			\llbracket\vv\rrbracket \equiv w^+ \vh_+ + w^- \vh_- \qq{on} \Gamma.
		\end{equation}
		Thus, \eqref{Syro 3"} is equivalent to 
		\begin{equation}
			\abs{\vh_+ \cp \vh_-}^2 > \frac{\rho_+}{\rho_+ + \rho_-}\abs{w^-}^2\abs{\vh_+ \cp \vh_-}^2 + \frac{\rho_-}{\rho_+ + \rho_-}\abs{w^+}^2\abs{\vh_+ \cp \vh_-}^2,
		\end{equation}
		namely,
		\begin{equation}
			1 > \frac{\rho_+}{\rho_+ + \rho_-}\abs{w^-}^2 + \frac{\rho_-}{\rho_+ + \rho_-}\abs{w^+}^2.
		\end{equation}
		Hence, direct calculations yield
		\begin{equation}
			\begin{split}
				&\hspace{-1em}\frac{\rho_+}{\rho_+ + \rho_-}\abs{\vb{a} \vdot \vh_+}^2 + \frac{\rho_-}{\rho_+ + \rho_-}\abs{\vb{a} \vdot \vh_-}^2 \\
				&= \frac{\rho_+}{\rho_+ + \rho_-}\abs{a^+ g_{++} + a^- g_{+-}}^2 + \frac{\rho_-}{\rho_+ + \rho_-}\abs{a^- g_{--} + a^+ g_{+-}}^2 \\
				&> \qty(\frac{\rho_+}{\rho_+ + \rho_-}\abs{w^-}^2 + \frac{\rho_-}{\rho_+ + \rho_-}\abs{w^+}^2) \frac{\rho_+}{\rho_+ + \rho_-}\abs{a^+ g_{++} + a^- g_{+-}}^2 \\
				&\qquad + \qty(\frac{\rho_+}{\rho_+ + \rho_-}\abs{w^-}^2 + \frac{\rho_-}{\rho_+ + \rho_-}\abs{w^+}^2) \frac{\rho_-}{\rho_+ + \rho_-}\abs{a^- g_{--} + a^+ g_{+-}}^2 \\ 
				&\ge \frac{\rho_+\rho_-}{(\rho_+ + \rho_-)^2} \abs{a^+ w^+ g_{++} + a^-w^-g_{--} + a^-w^+ g_{+-} + a^+ w^- g_{+-}}^2 \\
				&= \frac{\rho_+ \rho_-}{(\rho_+ + \rho_-)^2}\abs{\vb{a} \vdot \llbracket\vv\rrbracket}^2,
			\end{split}
		\end{equation}
		which is exactly \eqref{strong syro equiv}.
	\end{proof}
	Since $ \Gamma $ is assumed to be compact, the following corollary follows:
	\begin{cor}
		Suppose that \eqref{strong syro equiv} holds on $ \Gamma $. Then it holds that 
		\begin{equation}\label{syro inf}
			\begin{split}
				\Upsilon\qty(\vh_{\pm}, \llbracket\vv\rrbracket) &\coloneqq \inf_{\substack{\vb{a} \in \mathrm{T}\Gamma_t; \\ \abs{\vb{a}}=1}} \inf_{z \in \Gamma_t} \frac{\rho_+}{\rho_+ + \rho_-}\abs{\vb{a} \vdot \vh_+(z)}^2 + \frac{\rho_-}{\rho_+ + \rho_-}\abs{\vb{a} \vdot \vh_-(z)}^2 \\
				&\hspace{6em}  - \frac{\rho_+ \rho_-}{(\rho_+ + \rho_-)^2}\abs{\vb{a} \vdot \llbracket\vv\rrbracket(z)}^2 \\
				&=: \mathfrak{s_0}>0.
			\end{split}
		\end{equation}
		Equivalently, the following relation holds on $ \Gamma $:
		\begin{equation}\label{syro stab condi useful}
			\qty(\frac{\rho_+}{\rho_+ + \rho_-} (\vh_+ \otimes \vh_+) + \frac{\rho_-}{\rho_+ + \rho_-}  (\vh_- \otimes \vh_-) - \frac{\rho_+ \rho_-}{(\rho_+ + \rho_-)^2}(\llbracket\vv\rrbracket \otimes \llbracket\vv\rrbracket)) \ge \mathfrak{s}_0 \vb{I}.
		\end{equation}
	\end{cor}
	
	\subsection{Interfaces, coordinates and div-curl systems}\label{sec prelimi'}
	From now on, $ \Om $ is assumed to be $ \mathbb{T}^2 \times (-1, 1) $ and $ \Om_t^+ $ has a solid boundary $ \mathbb{T}^2 \times \{+1\} $. Hence, some statements in \textsection~\ref{sec harmonic coord} and \textsection~\ref{sec div-cul system} need slight changes in order to be compatible to the topology of $ \Om_t^\pm $. More precisely, the harmonic coordinate maps introduced in \textsection~\ref{sec harmonic coord} are now replaced by
	\begin{equation}
		\begin{cases*}
			\lap_y \X_\Gamma^\pm = 0 &for $ y \in \Om_*^\pm $, \\
			\X_\Gamma^\pm (z) = \Phi_\Gamma(z) &for $ z \in \Gamma_* $, \\
			\X_\Gamma^\pm (z) = z &for $ z \in \mathbb{T}^2 \times \{\pm1\}$.
		\end{cases*}
	\end{equation}
	Similarly, the definitions of harmonic extensions of a function $ f $ defined on $ \Gamma $ are modified to
	\begin{equation}\label{eqn harm ext'}
		\begin{cases*}
			\lap \h_\pm f = 0 \qfor x \in \Om_\Gamma^\pm, \\
			\h_\pm f = f \qfor x \in \Gamma, \\
			\DD_{\wt{\vn}_\pm} \h_\pm f = 0 \qfor x \in \mathbb{T}^2\times \{\pm1\}.
		\end{cases*}
	\end{equation}
	The Dirichlet-Neumann operators are also defined by \eqref{def DN op} for which $ \h_\pm $ are given by \eqref{eqn harm ext'}.
	
	Therefore, Lemmas \ref{lem composition harm coordi} - \ref{lem lap-n}, and those properties of the Dirichlet-Neumann operators introduced in \textsection~\ref{sec harmonic coord} still hold.
	
	As for the div-curl systems, due to the different topology, we introduce the following modification of Theorem \ref{thm div-curl}:
	\begin{thm}\label{thm div-curl'}
		Assume that $ \Gamma $ is an $ H^{\kk-\frac{1}{2}} (k \ge 3) $ surface diffeomorphic to $ \mathbb{T}^2 $, with
		\begin{equation}\label{surface away condi}
			\dist(\Gamma, \mathbb{T}^2 \times \{\pm1\}) \ge c_0 > 0
		\end{equation}
		for some positive constant $ c_0 $.
		Given $ \vb{f}, g \in H^{l-1}(\Om^+) $ and $ h \in H^{l-\frac{1}{2}}(\Gamma) $ with the compatibility condition
		\begin{equation*}
			\int_{\Om^+} g \dd{x} = \int_{\Gamma} h \dd{S},
		\end{equation*}
		and suppose further that $ \vb{f} $ satisfies
		\begin{equation}
			\div \vb{f} = 0 \text{ in } \Om^+ \qc  \int_{\mathbb{T}^2\times\{+1\}} \vb{f} \vdot \wt{\vn}_+ \dd{S} = 0.
		\end{equation}
		Then, for $  2 \le l \le \kk-1 $, the following system:
		\begin{equation}
			\begin{cases*}
				\curl \vbu = \vb{f} &in $ \Om^+ $, \\
				\div \vbu = g &in $ \Om^+ $, \\
				\vbu \vdot \vn_+ = h &on $ \Gamma $, \\
				\vbu \vdot \wt{\vn}_+ = 0 \qc \int_{\mathbb{T}^2\times\{+1\}} \vbu \dd{\wt{S}} = \va{\mathfrak{u}} &on $ \mathbb{T}^2 \times \{+1\} $
			\end{cases*}
		\end{equation}
		admits a unique solution $ \vbu \in H^{l}(\Om^+) $ satisfying the estimate:
		\begin{equation}
			\norm{\vbu}_{H^l(\Om^+)} \le C\qty(\abs{\Gamma}_{H^{\kk-\frac{1}{2}}}, c_0) \times \qty( \norm{\vf}_{H^{l-1}(\Om^+)} + \norm{g}_{H^{l-1}(\Om^+)} + \abs{h}_{H^{l-\frac{1}{2}}(\Gamma)} + \abs{\va{\mathfrak{u}}}).
		\end{equation}
	\end{thm}
	One may refer to \cite{Cheng-Shkoller2017} and \cite{Sun-Wang-Zhang2018} for a proof of Theorem \ref{thm div-curl'}.
	
	\subsection{Reformulation of the problem}
	We shall consider the free interface problem (MHD)-(BC') under the assumption that $ k \ge 3 $. Due to the difference between Theorems \ref{thm div-curl} and \ref{thm div-curl'}, the velocity and magnetic fields depend on one more boundary condition -- their integrals on the bottom or the top solid boundary. Therefore, when considering the variation, one needs to assume further that the integrals of $ \vv_\pm $ and $ \vh_\pm $ on $ \mathbb{T}^2 \times \{\pm1\} $ also depend on the parameter $ \beta $. More precisely, set
	\begin{equation}
		\va{\mathfrak{v}}_\pm \coloneqq \int_{\mathbb{T}^2 \times \{\pm1\}} \vv_\pm  \dd{\wt{S}} \qand \va{\mathfrak{h}}_\pm \coloneqq \int_{\mathbb{T}^2 \times \{\pm1\}} \vh_\pm  \dd{\wt{S}}.
	\end{equation}
	Then, for each fixed $ t $, $ \va{\mathfrak{v}}_\pm $ and $ \va{\mathfrak{h}}_\pm $ are constant tangential vectors on $ \mathbb{T}^2 \times\{\pm1\} $.
	With the same notations in \textsection~\ref{sec var v_*}, assume that $ \ka, \vom_{*\pm} $ and $ \va{\mathfrak{v}}_\pm $ are parameterized by $ \beta $. Thus, \eqref{eqn pd beta v} can be rewritten as
	\begin{equation}\label{eqn pd beta vv '}
		\pd_\beta \vv_{\pm*} = \opb_\pm(\ka)\pd^2_{t\beta}\ka + \opf_\pm(\ka)\pd_\beta\vom_{*\pm} + \opg_\pm(\ka, \pd_t \ka, \vom_{*\pm})\pd_\beta\ka + \vb{S}_\pm (\ka) \pd_\beta \va{\mathfrak{v}}_{\pm}.
	\end{equation}
	It follows from the same arguments that Lemma \ref{lem3.1} still holds with a subtle modification of indices, namely, $ s, s', \sigma, \sigma' \ge \frac{3}{2} $ rather than $ \ge \frac{1}{2} $. As for the new term $ \vb{S}_\pm (\ka)\pd_\beta \va{\mathfrak{v}}_\pm $, they are the pull-backs to $ \Gs $ of the solutions to the following boundary value problems:
	\begin{equation}\label{def opS}
		\begin{cases*}
			\div \vb{y}_\pm = 0 \qc \curl \vb{y}_\pm = \vb{0} &in $ \Om^\pm_t $, \\
			\vb{y}_\pm \vdot \vn_\pm = 0 &on $ \Gt $, \\
			\vb{y}_\pm \vdot \wt{\vn}_\pm = 0 &on $ \mathbb{T}^2 \times \{\pm1\} $, \\
			\int_{\mathbb{T}^2 \times \{\pm 1\}} \vb{y}_\pm \dd{\wt{S}} = \pd_\beta \va{\mathfrak{v}}_\pm &on $ \mathbb{T}^2 \times \{\pm 1\} $.
		\end{cases*}
	\end{equation}
	Therefore, since $ \pd_\beta \va{\mathfrak{v}}_\pm $ are constant on $ \mathbb{T}^2 \times \{\pm1\} $ for each fixed $ \beta $, the following estimates hold:
	\begin{equation}\label{est S}
		\abs{\vb{S}_\pm (\ka)}_{\LL\qty(\R^2; \H{s})} \le C_* \qfor \frac{3}{2} \le s \le \kk-\frac{3}{2},
	\end{equation}
	and
	\begin{equation}\label{est var S}
		\abs{\var \vb{S}_\pm(\ka)}_{\LL\qty(\H{\kk-\frac{5}{2}}; \LL\qty(\R^2; \H{s'}))} \le C_* \qfor \frac{3}{2} \le s' \le \kk-\frac{3}{2}.
	\end{equation}
	
	Due to the change of $ \Om^\pm_t $, we shall also modify the definition of $ p^\pm_{\vb{a}, \vb{b}} $ to:
	\begin{equation}\label{eqn p_ab '}
		\begin{cases*}
			-\lap p_{\vb{a}, \vb{b}}^\pm = \tr(\DD\vb{a}_\pm \vdot \DD\vb{b}_\pm) &in $ \Om_t^\pm $, \\
			p^\pm_{\vb{a}, \vb{b}} = 0 &on $ \Gt $, \\
			\DD_{\wt{\vn}_\pm} p^{\pm}_{\vb{a}, \vb{b}} = \wt{\II}_\pm(\vb{a}_\pm, \vb{b}_\pm) &on $ \mathbb{T}^2 \times \{\pm1\} $,
		\end{cases*}
	\end{equation}
	for which the solenoidal vector fields $ \vb{a}_\pm, \vb{b}_\pm $ satisfy $ \vb{a}_\pm \vdot \wt{\vn}_\pm = 0 = \vb{b}_\pm \vdot \wt{\vn}_\pm  $ on $ \mathbb{T}^2 \times \{\pm1\} $.
	
	With all the previous modifications on the definitions of the harmonic extensions, Lagrange multiplier pressures, and div-curl systems, it follows from the similar arguments that \eqref{eqn pd2 tt ka} can be rewritten as
	\begin{equation}
		\begin{split}
			&\pd^2_{tt}\ka + \opC_{\alpha} (\ka, \pd_t\ka, \vv_{*\pm}, \vh_{*\pm} ) \ka  - \opF(\ka)\pd_t\vom_* - \opG(\ka, \pd_t\ka, \vom_*, \vj_*, \va{\mathfrak{v}}, \va{\mathfrak{h}}) - \mathscr{S}(\ka)\pd_t \va{\mathfrak{v}} \\
			&\quad = \qty[\mathrm{I} + \opB(\ka)]^{-1}\qty{\qty[-\lap_\Gt \qty(\vW \vdot \vn_+) + \vW \vdot \lap_\Gt \vn_+ + a^2 \dfrac{\vW \vdot \vn_+}{\vn_+ \vdot (\vnu \circ \Phi_\Gt^{-1})}] \circ \Phi_\Gt}.
		\end{split}
	\end{equation}
	Since $ \va{\mathfrak{v}}_\pm $ and $ \va{\mathfrak{h}}_\pm $ are all constants, Lemma \ref{lem 3.4} holds with $ k \ge 3, \alpha = 0 $ and a slight change of \eqref{est var opF} and \eqref{est var opG} as:
	\begin{equation}\label{est var opF'}
		\abs{\var\opF(\ka)}_{\LL\qty[\H{\kk-\frac{5}{2}}; \LL\qty(H^{\kk-\frac{7}{2}}(\OGs); \H{\kk-\frac{7}{2}})]} \le C_*,
	\end{equation}
	and
	\begin{equation}\label{est var opG'}
		\begin{split}
			&\hspace{-2em}\abs{\var\opG}_{\LL\qty[\H{\kk-\frac{5}{2}}\times\H{\kk-\frac{7}{2}}\times H^{\kk-2}(\OGs) \times H^{\kk-2}(\OGs) \times \R^2 \times \R^2; \H{\kk-\frac{7}{2}}]} \\
			\le\, &a^2 Q\qty(\abs{\pd_t\ka}_{\H{\kk-\frac{5}{2}}}, \norm{\vom_*}_{H^{\kk-1}(\OGs)}, \norm{\vj_*}_{H^{\kk-1}(\OGs)}),
		\end{split}
	\end{equation}
	whose proof follows from the same arguments.
	Furthermore, the operator $ \mathscr{S}(\ka) $ satisfies
	\begin{equation}\label{est opS}
		\abs{\mathscr{S}(\ka)}_{\LL\qty(\R^2; \H{\kk-\frac{5}{2}})} \le Q\qty(\abs{\ka}_{\H{\kk-\frac{3}{2}}}),
	\end{equation}
	and
	\begin{equation}\label{est var opS}
		\abs{\var\mathscr{S}(\ka)}_{\LL\qty[\H{\kk-\frac{7}{2}}; \LL\qty(\R^2; \H{\kk-\frac{7}{2}})]} \le Q\qty(\abs{\ka}_{H^{\kk-\frac{3}{2}}}).
	\end{equation}
	
	Indeed, the leading order term of $ \mathscr{S}(\ka)\pd_t\va{\mathfrak{v}} $ is $ \grad^\top(\kappa_+ \circ \Phi_\Gt) \vdot \vb{S}(\ka)\pd_t \va{\mathfrak{v}} $, so the above estimates follow from the standard commutator and product ones.
	
	\subsection{Linear systems}\label{sec syro linear system}
	Similar to the arguments in \textsection~\ref{sec linear ka}, assume that $ \Gs \in H^{\kk+\frac{1}{2}} (k\ge 3) $ is a reference hypersurface, and $ \Lambda_* $ defined by (\ref{def lambda*}) satisfies all the properties discussed in the preliminary. Suppose further that there are a family of hypersurfaces $ \Gt \in \Lambda_* $ parameterized by $ t \in [0, T] $ and four tangential vector fields $ \vv_{\pm*}, \vh_{\pm*} : \Gs \to \mathrm{T}\Gs $ satisfying:
	\begin{equation}
		\ka \in C^0\qty{[0, T]; \H{\kk-\frac{3}{2}}} \cap C^1\qty{[0, T]; B_{\delta_1} \subset \H{\kk-\frac{5}{2}}}, \tag{H1'}
	\end{equation}
	and
	\begin{equation}
		\vv_{\pm*}, \vh_{\pm*} \in C^0\qty{[0, T]; \H{\kk-\frac{1}{2}}} \cap C^1\qty{[0, T]; \H{\kk-\frac{3}{2}}}. \tag{H2'}
	\end{equation}
	Moreover, assume that there are positive constants $ c_0, \mathfrak{s}_0 $ so that \eqref{syro stab condi useful} and \eqref{surface away condi} hold uniformly on $ [0, T] $.
	
	The positive constants $ \wt{L_1}$ and $ \wt{L_2} $ are defined respectively by:
	\begin{equation}
		\sup_{t\in[0, T]} \qty{\abs{\ka (t)}_{\H{\kk-\frac{3}{2}}},  \abs{\pd_t \ka(t)}_{\H{\kk-\frac{5}{2}}}, \abs{\qty(\vv_{\pm*}(t), \vh_{\pm*}(t))}_{\H{\kk-\frac{1}{2}}}} \le \wt{L_1},
	\end{equation}
	and
	\begin{equation}
		\sup_{t\in [0, T]} \abs{\qty(\pd_t\vv_{\pm*}(t), \pd_t\vh_{\pm*}(t))}_{\H{\kk-\frac{3}{2}}} \le \wt{L_2}.
	\end{equation}
	
	Consider the following linear initial value problem similar to \eqref{eqn linear 1}:
	\begin{equation}\label{eqn linear 1'}
		\begin{cases}
			\pd^2_{tt}\f + \opC_0(\ka, \pd_t\ka, \vv_*, \vh_*)\f = \g, \\
			\f(0) = \f_0, \quad \pd_t \f(0) = \f_1,
		\end{cases}
	\end{equation}
	where $ \f_0, \f_1, \g(t) : \Gs \to \R $ are three given functions, and $ \opC_0 $ is given by:
	\begin{equation*}
		\begin{split}
			\opC_0(\ka, \pd_t\ka, \vv_*, \vh_*) \coloneqq\, &2\DD_{\vbu_*}\pd_t + \DD_{\vbu_*}\DD_{\vbu_*} +  \dfrac{\rho_+\rho_-}{\qty(\rho_+ + \rho_-)^2}\opR(\ka, \vw_*) \\ &-\dfrac{\rho_+}{\rho_+ + \rho_-}\opR(\ka, \vh_{+*}) - \dfrac{\rho_-}{\rho_+ + \rho_-}\opR(\ka, \vh_{-*}),
		\end{split}
	\end{equation*}
	which is exactly (\ref{def opC}) with $ \alpha = 0 $.
	
	Thus, for $ 0 \le l \le k-2 $, the energy \eqref{eqn E_l} is replaced by
	\begin{equation}\label{eqn E_l'}
		\begin{split}
			\wt{E_l}(t, \f, \pd_t\f) \coloneqq \int_{\Gt}
			&\abs{\qty(-\wn^{\frac{1}{2}}\lap_\Gt \wn^{\frac{1}{2}})^{\frac{l}{2}} \wn^{\frac{1}{2}} \qty[\qty(\pd_t \f + \DD_{\vbu_*} \f) \circ \Phi_\Gt^{-1}] }^2 \\
			&- \dfrac{\rho_+\rho_-}{\qty(\rho_+ + \rho_-)^2} \abs{\qty(-\wn^{\frac{1}{2}}\lap_\Gt  \wn^{\frac{1}{2}})^{\frac{l}{2}}\wn^{\frac{1}{2}} \qty[\qty(\DD_{\vw_*} \f) \circ\Phi_\Gt^{-1}]}^2 \\
			&+ \dfrac{\rho_+}{\rho_+ + \rho_-} \abs{\qty(-\wn^{\frac{1}{2}}\lap_\Gt \wn^{\frac{1}{2}})^{\frac{l}{2}}\wn^{\frac{1}{2}} \qty[\qty(\DD_{\vh_{+*}} \f) \circ\Phi_\Gt^{-1}] }^2 \\
			&+ \dfrac{\rho_-}{\rho_+ + \rho_-} \abs{\qty(-\wn^{\frac{1}{2}}\lap_\Gt \wn^{\frac{1}{2}})^{\frac{l}{2}}\wn^{\frac{1}{2}}  \qty[\qty(\DD_{\vh_{-*}} \f) \circ\Phi_\Gt^{-1}]}^2 \dd{S_t}.
		\end{split}
	\end{equation}
	It follows from the same arguments as in the proof of Lemma \ref{lem est E_l} that there exists a generic polynomial $ Q $ determined by $ \Lambda_* $, such that the following estimate holds:
	\begin{equation}\label{est E_l'}
		\begin{split}
			&\hspace{-2em}\wt{E_l}(t, \f, \pd_t\f) - \wt{E_l}(0, \f_0, \f_1) \\
			\le&Q(\wt{L_1}, \wt{L_2})\int_0^t \qty(\abs{\f(s)}_{\H{\frac{3}{2}l + \frac{3}{2}}} + \abs{\pd_t\f(s)}_{\H{\frac{3}{2}l + \frac{1}{2}}} + \abs{\g(s)}_{\H{\frac{3}{2}l + \frac{1}{2}}}) \times \\
			&\hspace{8em} \times \qty(\abs{\f(s)}_{\H{\frac{3}{2}l + \frac{3}{2}}} + \abs{\pd_t\f(s)}_{\H{\frac{3}{2}l + \frac{1}{2}}}) \dd{s}.
		\end{split}
	\end{equation}
	Thanks to the uniform stability condition \eqref{syro stab condi useful}, one can derive an estimate similar  to \eqref{est linear eqn ka}, as long as $ T \le C $ for some constant $ C = C(\wt{L_1}, \wt{L_2}, \mathfrak{s}_0) $:
	\begin{equation}\label{est linear eqn ka'}
		\begin{split}
			&\hspace{-1em}\abs{\f(t)}_{\H{\frac{3}{2}l+\frac{3}{2}}}^2 + \abs{\pd_t\f(t)}_{\H{\frac{3}{2}l+\frac{1}{2}}}^2 \\
			\le\, &C_* e^{Q(\wt{L_1}, \wt{L_2}, \mathfrak{s}_0^{-1})t} \qty( \abs{\f_0}_{\H{\frac{3}{2}l+\frac{3}{2}}}^2 + \abs{\f_1}_{\H{\frac{3}{2}l+\frac{1}{2}}}^2 + \int_0^t \abs{\g(t')}^2_{\H{\frac{3}{2}l+\frac{1}{2}}} \dd{t'}),
		\end{split}
	\end{equation}
	for any integer $ 0 \le l \le k-2 $, $ 0 \le t \le T $, a generic polynomial $ Q $ and a positive constant $ C_* $ depending on $ \Lambda_* $.
	Thus, one has
	\begin{prop}
		For $ 0 \le l \le k-2 $, $ T \le C(\wt{L_1}, \wt{L_2}, \mathfrak{s}_0) $ and $ \g \in C^0 \qty([0, T]; H^{\frac{3}{2}l + \frac{1}{2}}(\Gs)) $, the linear problem \eqref{eqn linear 1'} is well-posed in $ C^0\qty([0, T]; \H{\frac{3}{2}l+\frac{3}{2}}) \cap C^1\qty([0, T]; \H{\frac{3}{2}l+\frac{1}{2}}) $, and the energy estimate \eqref{est linear eqn ka'} holds.
	\end{prop}
	
	It is also noted that the arguments in \textsection~\ref{sec linear current vortex} are still valid for the linear systems for the current and vorticity here.
	
	\subsection{Nonlinear problems}
	As in \textsection~\ref{sec nonlinear}, take a reference hypersurface $ \Gs \in H^{\kk+\frac{1}{2}} $ and $ \delta_0 > 0  $ so that
	\begin{equation*}
		\Lambda_* \coloneqq \Lambda \qty(\Gs, \kk-\frac{1}{2}, \delta_0)
	\end{equation*}
	satisfies all the properties discussed in the preliminary. Furthermore, assume that there is a constant $ c_0 > 0 $ so that \eqref{surface away condi} holds for $ \Gs $. We shall solve the nonlinear problem by iterations on the linearized problems in the spaces:
	\begin{equation*}
		\begin{split}
			&\ka \in C^0\qty([0, T]; \H{\kk-\frac{3}{2}}) \cap C^1\qty([0, T]; B_{\delta_1}\subset\H{\kk-\frac{5}{2}}) \cap C^2\qty([0, T]; \H{\kk-\frac{7}{2}}); \\
			&\vom_{\pm*}, \vj_{\pm*} \in  C^0\qty([0, T]; H^{\kk-1}(\Om_*^\pm)) \cap C^1\qty([0, T]; H^{\kk-2}(\Om_*^\pm)); \\
			&\va{\mathfrak{v}}_\pm, \va{\mathfrak{h}}_\pm \in C^1\qty([0, T]; \R^2).
		\end{split}
	\end{equation*}
	
	\subsubsection{Fluid region, velocity and magnetic fields}
	
	As discussed in \textsection~\ref{section recovery}, the bulk region, velocity and magnetic fields can be obtained by solving the following div-curl problems:
	\begin{equation}\label{div-curl nonlinear v'}
		\begin{cases*}
			\div \vv_\pm = 0 \qc \curl \vv_\pm = \bar{\vom}_\pm &in $ \Om^\pm_t $, \\
			\vv_\pm \vdot \vn_+ = \vn_+ \vdot (\pd_t\gt \vnu) \circ\Phi_\Gt^{-1} &on $ \Gt $, \\
			\vv_\pm \vdot \wt{\vn}_\pm = 0 &on $ \mathbb{T}^2 \times \{\pm1\} $, \\
			\int_{\mathbb{T}^2\times\{\pm1\}} \vv_\pm \dd{\wt{S}} = \va{\mathfrak{v}}_\pm &on $ \mathbb{T}^2 \times \{\pm1\} $;
		\end{cases*}
	\end{equation}
	and
	\begin{equation}\label{div-curl nonlinear h'}
		\begin{cases*}
			\div \vh_\pm = 0 \qc \curl \vh_\pm = \bar{\vj}_\pm &in $ \Om_t^\pm $, \\
			\vh_\pm \vdot \vn_\pm = 0 &on $ \Gt $, \\
			\vh_\pm \vdot \wt{\vn}_\pm = 0 &on $ \mathbb{T}^2 \times \{\pm1\} $, \\
			\int_{\mathbb{T}^2 \times \{\pm1\}} \vh_\pm \dd{\wt{S}} = \va{\mathfrak{h}}_\pm &on $ \mathbb{T}^2 \times \{\pm1\} $,
		\end{cases*}
	\end{equation}
	where $ \bar{\vom}_\pm $ and $ \bar{\vj}_\pm $ are given by \eqref{def bar vom vj}.
	
	\subsubsection{Iteration mapping}\label{sec syro ite map}
	
	In order to construct the iteration map, we consider the following function space:
	\begin{defi}
		For given constants $ T, M_0, M_1, M_2, M_3, c_0, \mathfrak{s}_0 >0 $, define $ \mathfrak{X} $ to be the collection of $ \qty(\ka, \vom_*, \vj_*, \va{\mathfrak{v}}_\pm, \va{\mathfrak{h}}_\pm) $ satisfying:
		\begin{equation*}
			\abs{\ka(0) - \kappa_{*+}}_{\H{\kk-\frac{5}{2}}} \le \delta_1,
		\end{equation*}
		\begin{equation*}
			\abs{\qty(\pd_t \ka)(0)}_{\H{\kk-\frac{7}{2}}}, \norm{\vom_*(0)}_{H^{\kk-\frac{5}{2}}(\OGs)}, \norm{\vj_*(0)}_{H^{\kk-\frac{5}{2}}(\OGs)}, \abs{\va{\mathfrak{v}}_\pm(0)}, \abs{\va{\mathfrak{h}}_\pm(0)} \le M_0,
		\end{equation*}
		\begin{equation*}
			\sup_{t \in [0, T]} \qty(\abs{\ka}_{\H{\kk-\frac{3}{2}}}, \abs{\pd_t \ka}_{\H{\kk-\frac{5}{2}}}, \norm{(\vom_*, \vj_*)}_{H^{\kk-1}(\OGs)}, \abs{\va{\mathfrak{v}}_\pm}, \abs{\va{\mathfrak{h}}_\pm} ) \le M_1,
		\end{equation*}
		\begin{equation*}
			\sup_{t \in [0, T]} \qty(\norm{\pd_t\vom_*}_{H^{\kk-2}(\OGs)}, \norm{\pd_t\vj_*}_{H^{\kk-2}(\OGs)}, \abs{\pd_t \va{\mathfrak{v}}_\pm}, \abs{\pd_t\va{\mathfrak{h}}_\pm}) \le M_2,
		\end{equation*}
		\begin{equation*}
			\sup_{t \in [0, T]} \abs{\pd^2_{tt}\ka}_{\H{\kk-\frac{7}{2}}} \le a^2 M_3 \ (\text{here $ a $ is the constant in the definition of $ \ka $}).
		\end{equation*}
		For $ \Upsilon(\vh_\pm, \llbracket\vv\rrbracket) $ defined by \eqref{syro inf},
		\begin{equation*}
			\Upsilon\qty(\vh_\pm, \llbracket\vv\rrbracket) \ge \mathfrak{s}_0
		\end{equation*}
		holds uniformly for $ 0 \le t \le T $.
		In addition, \eqref{surface away condi} and the compatibility conditions
		\begin{equation*}
			\int_{\mathbb{T}^2\times\{\pm1\}} \wt{\vn}_\pm \vdot \vom_{*\pm} \dd{\wt{S}} = \int_{\mathbb{T}^2\times\{\pm1\}} \wt{\vn}_\pm \vdot \vj_{*\pm} \dd{\wt{S}} = 0
		\end{equation*}
		hold for all $ t\in [0, T] $.
	\end{defi}
	
	As for the initial data, take $ 0 < \epsilon \ll \delta_1 $ and $ A > 0 $, and consider:
	\begin{equation*}
		\mathfrak{I}(\epsilon, A) \coloneqq\qty{\qty\big(\ini{\ka}, \ini{\pd_t\ka}, \ini{\vom_*}, \ini{\vj_*}), \ini{\va{\mathfrak{v}}_\pm}, \ini{\va{\mathfrak{h}}_\pm} },
	\end{equation*}
	where
	\begin{gather*}
		\abs{\ini{\ka}-\kappa_{*+}}_{\H{\kk-\frac{3}{2}}}<\epsilon; \\ \abs{\ini{\pd_t\ka}}_{\H{\kk-\frac{5}{2}}},\
		\norm{\ini{\vom_*}}_{H^{\kk-1}(\OGs)},\ \norm{\ini{\vj_*}}_{H^{\kk-1}(\OGs)},\ \abs{\ini{\va{\mathfrak{v}}_\pm}},\ \abs{\ini{\va{\mathfrak{h}}_\pm}} < A,
	\end{gather*}
	\begin{equation*}
		\Upsilon(\vh_\pm, \llbracket \vv\rrbracket) \ge 2\mathfrak{s}_0,
	\end{equation*}
	and
	\begin{equation*}
		\dist\qty(\Gamma_{\mathrm{I}}, \mathbb{T}^2\times\{\pm1\}) \ge 2c_0.
	\end{equation*}
	In addition, $ \ini{\vom_*} $ and $ \ini{\vj_*} $ satisfy the following compatibility conditions:
	\begin{equation*}
		\int_{\mathbb{T}^2\times\{\pm1\}} \wt{\vn}_\pm \vdot {\ini{\vom_{*}}}_\pm \dd{\wt{S}} = \int_{\mathbb{T}^2\times\{\pm1\}} \wt{\vn}_\pm \vdot {\ini{\vj_{*}}}_\pm \dd{\wt{S}} = 0.
	\end{equation*}
	Thus, $ \mathfrak{I}(\epsilon, A) \subset \H{\kk-\frac{3}{2}}\times\H{\kk-\frac{5}{2}}\times H^{\kk-1}(\OGs) \times H^{\kk-1}(\OGs) \times \R^4$.
	
	Then, as in \textsection~\ref{sec itetarion map}, one can define the iteration map:
	\begin{equation}\label{eqn (n+1)ka'}
		\begin{cases*}
			\pd^2_{tt}\itm{\ka} + \opC_0\qty(\itn{\ka}, \itn{\pd_t\ka}, \itn{\vv_{*\pm}}, \itn{\vh_{*\pm}})\itm{\ka} \\ \qquad = \opF\qty(\itn{\ka})\pd_t \itn{\vom_*} + \opG\qty(\itn{\ka}, \itn{\pd_t\ka}, \itn{\vom_{*\pm}}, \itn{\vj_{*\pm}}, \itn{\va{\mathfrak{v}}_\pm}, \itn{\va{\mathfrak{h}}_\pm}) \\
			\qquad\qquad + \mathscr{S}\qty(\itn{\ka})\pd_t \itn{\va{\mathfrak{v}}_\pm} \\
			\itm{\ka}(0) = \ini{\ka}, \quad \itm{\pd_t\ka}(0) = \ini{\pd_t\ka};
		\end{cases*}
	\end{equation}
	and
	\begin{equation}\label{eqn linear (n+1) vom vj'}
		\begin{cases*}
			\pd_t\itm{\vom_\pm} + \DD_{\itn{\vv_\pm}}\itm{\vom_\pm} - \DD_{\itn{\vh_\pm}}\itm{\vj_\pm} = \DD_{\itm{\vom_\pm}}\itn{\vv_\pm} - \DD_{\itm{\vj_\pm}}\itn{\vh_\pm}, \\
			\pd_t\itm{\vj_\pm} + \DD_{\itn{\vv_\pm}}\itm{\vj_\pm} - \DD_{\itn{\vh_\pm}}\itm{\vom_\pm} \\
			\qquad  =  \DD_{\itm{\vj_\pm}}\itn{\vv_\pm} - \DD_{\itm{\vom_\pm}}\itn{\vh_\pm}
			- 2\tr(\grad\itn{\vv_\pm} \cp \grad \itn{\vh_\pm}), \\
			\itm{\vom_\pm}(0) = \Pb \qty(\ini{\vom_{*\pm}} \circ (\X_{\itn{\Gamma_0}}^{\pm})^{-1}), \quad \itm{\vj_\pm}(0) = \Pb\qty(\ini{\vj_{*\pm}}\circ (\X_{\itn{\Gamma_0}}^{\pm})^{-1}),
		\end{cases*}
	\end{equation}
	where $ \qty(\itn{\vv_{\pm}}, \itn{\vh_{\pm}}) $ is induced by $ \qty(\itn{\ka}, \itn{\vom_{*\pm}}, \itn{\vj_{*\pm}}, \itn{\va{\mathfrak{v}}_\pm}, \itn{\va{\mathfrak{h}}_\pm}) $ via solving \eqref{div-curl nonlinear v'}-\eqref{div-curl nonlinear h'}, the tangential vector fields $ \itn{\vv_{*\pm}} $ and $ \itn{\vh_{*\pm}} $ on $ \Gs $ are defined by
	\begin{equation*}
		\begin{split}
			\itn{\vv_{*\pm}} &\coloneqq \qty(\DD\Phi_{\itn{\Gt}})^{-1}\qty[\itn{\vv_\pm}\circ\Phi_{\itn{\Gt}} - \qty(\pd_t\gamma_{\itn{\Gt}})\vnu], \\
			\itn{\vh_{*\pm}} &\coloneqq \qty(\DD\Phi_{\itn{\Gt}})^{-1} \qty(\itn{\vh_\pm} \circ \Phi_{\itn{\Gt}}),
		\end{split}
	\end{equation*}
	and the current-vorticity equations are considered in the domains $ \itn{\Om_t^\pm} $.
	
	Define
	\begin{equation}\label{eqn dvt va v}
		\itm{\va{\mathfrak{v}}_\pm}(t) \coloneqq \ini{\va{\mathfrak{v}}_\pm} + \int_0^t\int_{\mathbb{T}^2\times\{\pm1\}} - \DD_{\itn{\vv_\pm}}\itn{\vv_\pm} - \frac{1}{\rho_\pm}\grad\itn{p^\pm} + \DD_{\itn{\vh_\pm}}\itn{\vh_\pm} \dd{\wt{S}} \dd{t'},
	\end{equation}
	\begin{equation}\label{eqn dvt va h}
		\itm{\va{\mathfrak{h}}_\pm}(t) \coloneqq \ini{\va{\mathfrak{h}}_\pm} + \int_0^t \int_{\mathbb{T}^2\times\{\pm1\}} \DD_{\itn{\vh_\pm}}\itn{\vv_\pm} - \DD_{\itn{\vv_\pm}} \itn{\vh_\pm} \dd{\wt{S}} \dd{t'},
	\end{equation}
	and
	\begin{equation}
		\itm{\vom_*}\coloneqq \itm{\vom}\circ\X_{\itn{\Gt}} \qc \itm{\vj_*}\coloneqq \itm{\vj}\circ\X_{\itn{\Gt}},
	\end{equation}
	where $ \itn{p^\pm} $ are given by \eqref{decomp p} with $ \qty(\itn{\ka}, \itn{\vv_{\pm}}, \itn{\vh_{\pm}}) $ plugged in.

	In order to show that the iteration map is a mapping from $ \mathfrak{X} $ to $ \mathfrak{X} $, one may first check that
	\begin{equation}
		\dist\qty(\Gamma^{(n+1)}_t, \mathbb{T}^2\times\{\pm 1\}) \ge 2c_0 - CT \abs{\pd_t \itm{\ka}}_{C^0_t \H{\kk-\frac{5}{2}}} \ge c_0,
	\end{equation}
	and
	\begin{equation}
		\abs{\Upsilon(\itm{\vh}_\pm, \llbracket\itm{\vv}\rrbracket) - \Upsilon\qty({\ini{\vh}}_\pm, \llbracket\ini{\vv}\rrbracket)} \le T Q(M_1, M_2) \le \mathfrak{s}_0,
	\end{equation}
	if $ T $ is small compared to $ M_1 $ and $ M_2 $.
	
	Next, for $ \itm{\va{\mathfrak{v}}_\pm} $ and $ \itm{\va{\mathfrak{h}}_\pm} $, observe that
	\begin{equation}
		\abs{\itm{\va{\mathfrak{v}}}_\pm (t)} + \abs{\itm{\va{\mathfrak{h}}_\pm}(t)} \le A + T Q(M_1) \le M_1,
	\end{equation}
	and
	\begin{equation}
		\abs{\pd_t\va{\mathfrak{v}}_\pm(t)} + \abs{\pd_t\va{\mathfrak{h}}_\pm(t)} \le Q(M_1) \le M_2,
	\end{equation}
	provided that $ T $ is small and $ M_2 \gg M_1 $.
	
	Then, with the same notation as in \textsection~\ref{sec itetarion map}, one can define
	\begin{equation}
		\begin{split}
			&\mathfrak{T}\qty(\qty[\ini{\ka}, \ini{\pd_t\ka}, \ini{\vom_{*\pm}}, \ini{\vj_{*\pm}}, \ini{\va{\mathfrak{v}}_\pm}, \ini{\va{\mathfrak{h}}_\pm}], \qty[\itn{\ka}, \itn{\vom_{*\pm}}, \itn{\vj_{*\pm}}, \itn{\va{\mathfrak{v}}_\pm}, \itn{\va{\mathfrak{h}}_\pm}]) \\
			&\coloneqq \qty(\itm{\ka}, \itm{\vom_{*\pm}}, \itm{\vj_{*\pm}}, \itm{\va{\mathfrak{v}}_\pm}, \itm{\va{\mathfrak{h}}_\pm}).
		\end{split}
	\end{equation}
	It follows from the arguments in \textsection~\ref{sec itetarion map} and the linear estimates in \textsection~\ref{sec syro linear system} that the following proposition holds:
	\begin{prop}
		Suppose that $ k \ge 3 $. For any $ 0 < \epsilon \ll \delta_0 $ and $ A > 0 $, there are positive constants $ M_0, M_1, M_2, M_3, c_0, \mathfrak{s}_0 $, so that for small $ T > 0 $,
		\begin{equation*}
			\mathfrak{T}\qty\Big{\qty[\ini{\ka}, \ini{\pd_t\ka}, \ini{\vom_{*\pm}}, \ini{\vj_{*\pm}}, \ini{\va{\mathfrak{v}}_\pm}, \ini{\va{\mathfrak{h}}_\pm}], \qty[\ka, \vom_{*\pm}, \vj_{*\pm}, \va{\mathfrak{v}}_\pm, \va{\mathfrak{h}}_\pm]} \in \mathfrak{X},
		\end{equation*}
		holds for any $ \qty(\ini{\ka}, \ini{\pd_t\ka}, \ini{\vom_{*\pm}}, \ini{\vj_{*\pm}}, \ini{\va{\mathfrak{v}}_\pm}, \ini{\va{\mathfrak{h}}_\pm}) \in \mathfrak{I}(\epsilon, A) $ and \\ $ \qty(\ka, \vom_{*\pm}, \vj_{*\pm}, \va{\mathfrak{v}}_\pm, \va{\mathfrak{h}}_\pm) \in \mathfrak{X} $.
	\end{prop}
	
	For the contraction of the iteration mapping, as in \textsection~\ref{sec contra ite}, assume that \\ $ \qty(\itn{\ka}(\beta), \itn{\vom_{*\pm}}(\beta), \itn{\vj_{*\pm}}(\beta), \itn{\va{\mathfrak{v}}_\pm}(\beta), \itn{\va{\mathfrak{h}}_\pm}(\beta) ) \subset \mathfrak{X} $ and \\ $ \qty(\ini{\ka}(\beta), \ini{\pd_t\ka}(\beta), \ini{\vom_{*\pm}}(\beta), \ini{\vj_{*\pm}}(\beta), \ini{\va{\mathfrak{v}}_\pm}(\beta), \ini{\va{\mathfrak{h}}_\pm}(\beta)) \subset \mathfrak{I}(\epsilon, A) $ are two families of data depending on a parameter $ \beta $.
	
	Define $ \qty(\itm{\ka}(\beta), \itm{\vom_{*\pm}}(\beta), \itm{\vj_{*\pm}}(\beta), \itm{\va{\mathfrak{v}}_\pm}(\beta), \itm{\va{\mathfrak{h}}_\pm}(\beta) ) $ to be the output of the iteration map. Then, by applying $ \pdv*{\beta} $ to \eqref{eqn (n+1)ka'}-\eqref{eqn dvt va h}, one has the variational problems \eqref{eqn beta n+1}-\eqref{eqn g2} as well as:
	\begin{equation}
		\begin{cases*}
			\pd^2_{tt} \pd_\beta \itm{\ka} + \itn{\opC}\pd_\beta\itm{\ka} \\ 
			\qquad  = - \qty(\pd_\beta\itn{\opC})\itm{\ka} + \pd_\beta \qty(\itn{\opF}\itn{\pd_t\vom_*} + \itn{\opG} + \itn{\mathscr{S}}\pd_t\itn{\va{\mathfrak{v}}}), \\
			\pd_\beta\itm{\ka}(0) = \pd_\beta\ini{\ka}(\beta), \quad \pd_t\qty(\pd_\beta\itm{\ka})(0) = \pd_\beta\ini{\pd_t\ka}(\beta),
		\end{cases*}
	\end{equation}
	\begin{equation}\label{eqn var pdt va v}
		\begin{split}
			&\pd_t\pd_\beta \itm{\va{\mathfrak{v}}_\pm}\\ &\quad=\pd_\beta\ini{\va{\mathfrak{v}}_\pm}+\int_0^t \int_{\mathbb{T}^2\times\{\pm1\}} \Dbt \qty(- \DD_{\itn{\vv_\pm}}\itn{\vv_\pm} - \frac{1}{\rho_\pm}\grad\itn{p^\pm} + \DD_{\itn{\vh_\pm}}\itn{\vh_\pm}) \dd{\wt{S}} \dd{t'},
		\end{split}
	\end{equation}
	and
	\begin{equation}\label{eqn var pdt va h}
		\pd_t\pd_\beta\itm{\va{\mathfrak{h}}_\pm} = \pd_\beta\ini{\va{\mathfrak{h}}_\pm} +  \int_0^t \int_{\mathbb{T}^2\times\{\pm1\}} \Dbt\qty(\DD_{\itn{\vh_\pm}}\itn{\vv_\pm} - \DD_{\itn{\vv_\pm}} \itn{\vh_\pm}) \dd{\wt{S}} \dd{t'}.
	\end{equation}
	
	Consider the energy functionals:
	\begin{equation}
		\begin{split}
			\itn{\E}(\beta)\coloneqq &\sup_{t \in [0, T]}\left(\abs{\pd_\beta\itn{\ka}}_{\H{\kk-\frac{5}{2}}} + \abs{\pd_\beta\itn{\pd_t\ka}}_{\H{\kk-\frac{7}{2}}} + \right. \\
			&\qquad\qquad \left.+ \norm{\pd_\beta\itn{\vom_{*\pm}}}_{H^{\kk-2}(\Om_*^\pm)} + \norm{\pd_\beta\itn{\vj_{*\pm}}}_{H^{\kk-2}(\Om_*^\pm)} +  \right. \\
			&\qquad\qquad\qquad \left. + \norm{\pd_\beta\pd_t\itn{\vom_{*\pm}}}_{H^{\kk-4}(\Om_*^\pm)} + \abs{\pd_\beta\itn{\va{\mathfrak{v}}_\pm}} + \abs{\pd_\beta\itn{\va{\mathfrak{h}}_\pm}} \right),
		\end{split}
	\end{equation}
	and
	\begin{equation}
		\begin{split}
			\ini{\E}(\beta)\coloneqq &\sup_{t \in [0, T]}\left(\abs{\pd_\beta\ini{\ka}}_{\H{\kk-\frac{5}{2}}} + \abs{\pd_\beta\ini{\pd_t\ka}}_{\H{\kk-\frac{7}{2}}} + \right. \\
			&\qquad\qquad \left.+ \norm{\pd_\beta\ini{\vom_{*\pm}}}_{H^{\kk-2}(\Om_*^\pm)} + \norm{\pd_\beta\ini{\vj_{*\pm}}}_{H^{\kk-2}(\Om_*^\pm)} +  \right. \\
			&\qquad\qquad\qquad \left.  + \abs{\pd_\beta\ini{\va{\mathfrak{v}}_\pm}} + \abs{\pd_\beta\ini{\va{\mathfrak{h}}_\pm}} \right).
		\end{split}
	\end{equation}
	
	It follows from \eqref{eqn var pdt va v}-\eqref{eqn var pdt va h}, \eqref{est var opG'}-\eqref{est var opS}, and the arguments in \textsection~\ref{sec contra ite} that
	\begin{equation}
		\itm{\E} \le \frac{1}{2}\itn{\E} + Q(M_1)\ini{\E},
	\end{equation} 
	provided that $ T $ is sufficiently small. That is, the following proposition holds:
	\begin{prop}\label{prop fixed point'}
		Assume that $ k \ge 3 $. For any $ 0 < \epsilon \ll \delta_0 $ and $ A > 0 $, there are positive constants $ M_0, M_1, M_2, M_3, c_0, \mathfrak{s}_0 $, so that if $ T $ is small enough, then there is a map $ \mathfrak{S} : \mathfrak{I}(\epsilon, A) \to  \mathfrak{X} $ such that
		\begin{equation}
			\mathfrak{T}\qty{\mathfrak{x}, \mathfrak{S(x)}} = \mathfrak{S(x)},
		\end{equation}
		for each $ \mathfrak{x} = \qty(\ini{\ka}, \ini{\pd_t\ka}, \ini{\vom_{*\pm}}, \ini{\vj_{*\pm}}, \ini{\va{\mathfrak{v}}_\pm}, \ini{\va{\mathfrak{h}}_\pm}) \in \mathfrak{I}(\epsilon, A)$.
	\end{prop}
	
	\subsubsection{The original MHD problem}\label{sec Back to original syro}
	For the fixed point given in Proposition \ref{prop fixed point'}, one can obtain $ \Gt, \vv_\pm$, and $ \vh_\pm $ by solving the div-curl problems \eqref{div-curl nonlinear v'}-\eqref{div-curl nonlinear h'}.
	Observe that \eqref{eqn dvt va v} and \eqref{eqn dvt va h} yield
	\begin{equation}
		\int_{\mathbb{T}^2\times\{\pm1\}}  \pd_t \vv_\pm + \DD_{\vv_\pm}\vv_\pm + \dfrac{1}{\rho_\pm}\grad p^\pm - \DD_{\vh_\pm}\vh_\pm \dd{\wt{S}} = 0
	\end{equation}
	and
	\begin{equation}
		\int_{\mathbb{T}^2\times\{\pm1\}} \pd_t \vh_\pm + \DD_{\vv_\pm}\vh_\pm - \DD_{\vh_\pm}\vv_\pm \dd{\wt{S}} = 0,
	\end{equation}
	which, together with the arguments in \textsection~\ref{sec back to MHD}, shows that $ (\Gt, \vv_\pm, \vh_\pm) $ is the unique solution to the (MHD)-(BC') problem.
	
	In particular, Theorem \ref{thm alpha=0 case} follows.

	\subsection{Vanishing surface tension limit}\label{sec vani st}
	
	In this subsection, it is always assumed that $ \Om = \mathbb{T}^2 \times \{\pm1\} $, $ k \ge 3 $ and the initial data satisfy the assumptions of Theorem \ref{thm alpha=0 case}. To derive the uniform-in-$ \alpha $ estimates, we consider the following four parts of the energies:
	\begin{equation}
		\mathcal{E}_0 (t) \coloneqq \frac{1}{2}\int_{\Om_t^+} \rho_+\qty(\abs{\vv_+}^2 + \abs{\vh_+}^2) \dd{x} + \frac{1}{2} \int_{\Om_t^-} \rho_- \qty(\abs{\vv_-}^2 + \abs{\vh_-}^2) \dd{x} + \int_\Gt \alpha^2 \dd{S_t},
	\end{equation}
	\begin{equation}
		\mathcal{E}_1 (t) \coloneqq \abs{\Dtb \kappa_+}_{H^{\kk-\frac{5}{2}}(\Gt)}^2 + \alpha^2 \abs{\kappa_+}_{H^{\kk-1}(\Gt)}^2 + \abs{\kappa_+}_{H^{\kk-\frac{3}{2}}(\Gt)}^2,
	\end{equation}
	\begin{equation}
		\mathcal{E}_2 (t) \coloneqq \norm{\vom_\pm}_{H^{\kk-1}(\Om_t^\pm)}^2 + \norm{\vj_\pm}_{H^{\kk-1}(\Om_t^\pm)}^2,
	\end{equation}
	and
	\begin{equation}
		\mathcal{E}_3 (t) \coloneqq \abs{\va{\mathfrak{v}}_\pm}^2 + \abs{\va{\mathfrak{h}}_\pm}^2.
	\end{equation}
	
	It follows from the (MHD)-(BC') system that
	\begin{equation}\label{eqn dt calE0}
		\dv{t}\mathcal{E}_0 (t) \equiv 0.
	\end{equation}
	Indeed, 
	\begin{equation}
		\begin{split}
			\dv{t} \frac{1}{2} \int_{\Om_t^+} \rho_+ \qty(\abs{\vv_+}^2 + \abs{\vh_+}^2) \dd{x} &= \int_{\Om_t^+} \rho_+ \DD_{\vh_+}(\vv_+ + \vh_+) - \vv_+ \vdot \grad p^+ \dd{x} \\
			&= \int_{\Om_t^+} - \vv_+ \vdot \grad p^+ + \rho_+ \DD_{\vh_+}(\vh_+ \vdot \vv_+) \dd{x} \\
			&= \int_{\Om_t^+} - \div(p^+\vv_+) + \rho_+ \div(\vh_+ [\vh_+ \vdot \vv_+]) \dd{x} \\
			&= \int_{\Gt} - p^+ \theta \dd{S_t}.
		\end{split}
	\end{equation}
	Similarly, one can calculate that
	\begin{equation}
		\dv{t} \frac{1}{2} \int_{\Om_t^-} \rho_- \qty(\abs{\vv_-}^2 + \abs{\vh_-}^2) \dd{x} = \int_{\Gt} p_- \theta \dd{S_t}.
	\end{equation}
	It follows from \eqref{eqn kp div N} and \eqref{eqn Dt dS} that
	\begin{equation}
		\begin{split}
			\dv{t} \int_\Gt \dd{S_t} &= \int_\Gt \Div_\Gt \vv \dd{S_t} \\
			&= \int_\Gt \Div_\Gt (\theta\vn_+) + \Div_\Gt (\vv-\theta\vn_+) \dd{S_t} \\
			&= \int_\Gt \theta \Div_\Gt\vn_+ + \grad^\top \theta \vdot \vn_+ \dd{S_t} \\
			&= \int_\Gt \theta \kappa_+ \dd{S_t}.
		\end{split}
	\end{equation}
	Thus, \eqref{eqn dt calE0} follows from the above relations and the boundary condition that $ p^+ - p^- = \alpha^2 \kappa_+ $ on $ \Gt $.
	
	Furthermore, applying the arguments in the proof of Lemma \ref{lem est E_l} to \eqref{eqn Dt2 kappa+}, \eqref{est frak R_0'} and \eqref{syro stab condi useful} yields
	\begin{equation}
		\mathcal{E}_1 (t) \lesssim_{\Lambda_{*}, \mathfrak{s}_0} 1 + \mathcal{E}_1(0) + \int_0^t Q \qty(\alpha\abs{\kappa_+}_{H^{\kk-1}({\Gamma_{t'}})}, \abs{\kappa_+}_{H^{\kk-\frac{3}{2}}({\Gamma_{t'}})}, \norm{(\vv, \vh)}_{H^{\kk}(\Om\setminus\Gamma_{t'})}) \dd{t'}
	\end{equation}
	for some generic polynomial $ Q $ determined by $ \Lambda_* $, $c_0$, and $ \mathfrak{s}_0 $.
	Similarly, \eqref{eqn pdt vom}-\eqref{eqn pdt vj} and the arguments in \textsection~\ref{sec linear current vortex} lead to:
	\begin{equation}
		\abs{\dv{t}\mathcal{E}_2(t)} \le Q \qty(\abs{\kappa_+}_{H^{\kk-\frac{3}{2}}(\Gt)}, \norm{\vv}_{H^{\kk}(\Om\setminus\Gt)}, \norm{\vh}_{H^{\kk} (\Om\setminus\Gt)}).
	\end{equation}
	As for $ \mathcal{E}_3 $, it follows from \eqref{eqn dvt va v} and \eqref{eqn dvt va h} that
	\begin{equation}
		\abs{\dv{t}\mathcal{E}_3(t)} \le Q \qty(\abs{\kappa_+}_{H^{\kk-\frac{3}{2}}(\Gt)}, \norm{\vv}_{H^{\kk}(\Om\setminus\Gt)}, \norm{\vh}_{H^{\kk} (\Om\setminus\Gt)}).
	\end{equation}
	
	On the other hand, if $ T $ is small compared to $ \norm{(\vv_{0\pm}, \vh_{0\pm})}_{H^{\kk}(\Om_0^\pm)} $, one has
	\begin{equation*}
		\dist \qty(\Gt, \mathbb{T}^2 \times \{\pm 1\}) \ge c_0.
	\end{equation*}
	Then it follows from the estimates of div-curl systems and Lemma \ref{est ii} that
	\begin{equation}
		\begin{split}
			\norm{\vv_\pm}_{H^{\kk}(\Om_t^\pm)} \le &Q\qty(\abs{\kappa_+}_{H^{\kk-\frac{3}{2}}(\Gt)}, c_0) \times \\
			&\quad \times \qty(\norm{\vom_\pm}_{H^{\kk-1}(\Om_t^\pm)} + \abs{\vn_+ \vdot \lap_{\Gt} \vv_\pm}_{H^{\kk-\frac{5}{2}}(\Gt)} + \abs{\va{\mathfrak{v}}_\pm}),
		\end{split}
	\end{equation}
	and
	\begin{equation}
		\begin{split}
			\norm{\vh_\pm}_{H^{\kk}(\Om_t^\pm)} \le &Q\qty(\abs{\kappa_+}_{H^{\kk-\frac{3}{2}}(\Gt)}, c_0) \times \\
			&\quad \times \qty(\norm{\vj_\pm}_{H^{\kk-1}(\Om_t^\pm)} + \abs{\vn_+ \vdot \lap_{\Gt} \vh_\pm}_{H^{\kk-\frac{5}{2}}(\Gt)} + \abs{\va{\mathfrak{h}}_\pm}).
		\end{split}
	\end{equation}
	
	In addition, \eqref{eqn dt kappa} implies
	\begin{equation*}
		-\vn_+ \vdot \lap_\Gt \vv_\pm = \Dt_\pm \kappa_+ + 2 \ip{\II_+}{(\DD\vv_\pm)^\top} = \Dtb\kappa_+ + \DD_{(\vv_\pm - \vbu)}\kappa_+ + 2 \ip{\II_+}{(\DD\vv_\pm)^\top}.
	\end{equation*}
	Thus, one has
	\begin{equation}
		\begin{split}
			\abs{\vn_+ \vdot \lap_\Gt \vv_\pm}_{H^{\kk-\frac{5}{2}}(\Gt)} \le C_* \qty(\abs{\Dtb\kappa_+}_{H^{\kk-\frac{5}{2}}(\Gt)} + \abs{\kappa_+}_{H^{\kk-\frac{3}{2}}(\Gt)}\abs{(\vv_+, \vv_-)}_{H^{\kk-\frac{3}{2}}(\Gt)}).
		\end{split}
	\end{equation}
	It follows from the interpolation inequalities of Sobolev norms that
	\begin{equation}
		\begin{split}
			\norm{\vv_\pm}_{H^{\kk}(\Om_t^\pm)} \le Q \qty(\abs{\Dtb \kappa_+}_{H^{\kk-\frac{5}{2}}(\Gt)},  \abs{\kappa_+}_{H^{\kk-\frac{3}{2}}(\Gt)}, \norm{\vom_\pm}_{H^{\kk-1}(\Om_t^\pm)}, \abs{\va{\mathfrak{v}}_\pm}, \norm{\vv_\pm}_{L^2(\Om_t^\pm)}),
		\end{split}
	\end{equation}
	where $ Q $ is a generic polynomial determined by $ \Lambda_* $ and $ c_0 $.
	Similarly,
	\begin{equation}
		-\vn_+ \vdot \lap_\Gt \vh_\pm = \DD_{\vh_\pm}\kappa_+ + 2\ip{\II_+}{(\DD\vh_\pm)^\top}
	\end{equation}
	implies that
	\begin{equation}
		\norm{\vh_\pm}_{H^{\kk}(\Om_t^\pm)} \le Q \qty(\abs{\kappa_+}_{H^{\kk-\frac{3}{2}}(\Gt)}, \norm{\vj_\pm}_{H^{\kk-1}(\Om_t^\pm)}, \abs{\va{\mathfrak{h}}_\pm}, \norm{\vh_\pm}_{L^2(\Om_t^\pm)}).
	\end{equation}
	
	In conclusion, setting
	\begin{equation}
		\mathcal{E} (t) \coloneqq \mathcal{E}_0 + \mathcal{E}_1 + \mathcal{E}_2 + \mathcal{E}_3,
	\end{equation}
	then it is not hard to derive that
	\begin{equation}
		\mathcal{E}(t) \lesssim_{\Lambda_{*}, \mathfrak{s}_0} 1 + \mathcal{E}(0) + \int_0^t Q\qty(\mathcal{E}(t')) \dd{t'},
	\end{equation}
	where $ Q $ is a generic polynomial depending on $ \Lambda_*, \mathfrak{s}_0 $ and $ c_0 $ (in particular, independent of $ \alpha $ and $ t $).
	
	Thus, Theorem \ref{thm vanishing surf limt} follows from the above energy estimates.
	
	\appendix
	\section{Proof of Lemma \ref{lem3.1}}
	\begin{proof}
		(\ref{est B}), (\ref{est F}), and (\ref{est B version2})-(\ref{est G}) follow from (\ref{est pd beta v+*})-(\ref{eqn pd2 gt}) and Proposition \ref{prop K}. As for the variational estimates, let $ f, h : \Gs \to \R $, $ \vg: \OGs \to \R^3 $  be given quantities, and $ \vw_1, \vw_2, \vw_3 : \Om\setminus\Gt \to \R^3 $ the solutions to the following div-curl problems respectively:
		\begin{equation}\label{eqn w1}
			\begin{cases*}
				\div\vw_1 = \comm{\div}{\DD_{\vf}}\vv &in $ \Om\setminus\Gt $, \\
				\curl\vw_1 = \comm{\curl}{\DD_{\vf}}\vv + \comm{\DD_{\vf}}{\Pb}\qty[\vom_* \circ \qty(\X_\Gt)^{-1}], &in $ \Om\setminus\Gt $, \\
				\vn_+ \vdot \vw_{1\pm} = \vn_+ \vdot \vb{b}_{\pm}(\ka, \pd_t\ka, f)  &on $ \Gt $, \\
				\wt{\vn} \vdot \vw_{1-} = 0 &on $ \pd\Om $,
			\end{cases*}
		\end{equation}
		where
		\begin{equation}
			\begin{split}
				\vb{b}_{\pm} = &\, \DD\qty[\qty(\var{\K^{-1}}(\ka)[f] \vnu) \circ (\Phi_\Gt)^{-1}] \vdot \qty[ \vv_\pm|_\Gt - \qty(\var{\K^{-1}}(\ka)[\pd_t\ka] \vnu) \circ (\Phi_\Gt)^{-1}] \\
				&+ \qty[\var[2]{\K^{-1}}(\ka)\qty[\pd_t\ka, f] \vnu] \circ (\Phi_\Gt)^{-1},
			\end{split}
		\end{equation}
		\begin{equation}
			\vf_\pm \coloneqq \h_\pm \qty[\qty(\var{\K^{-1}}(\ka)[f]\vnu)\circ(\Phi_\Gt)^{-1} ];
		\end{equation}
		\begin{equation}
			\begin{cases*}
				\div\vw_2 = 0 &in $ \OGt $, \\
				\curl\vw_2 = \Pb\qty(\vg \circ (\X_\Gt)^{-1}) &in $ \OGt $, \\
				\vn_+ \vdot \vw_{2\pm} = 0 &on $ \Gt $, \\
				\wt{\vn} \vdot \vw_{2-} = 0 &on $ \pd\Om $;
			\end{cases*}
		\end{equation}
		and
		\begin{equation}
			\begin{cases*}
				\div\vw_3 = 0 &in $ \OGt $, \\
				\curl\vw_3 = 0 &in $ \OGt $, \\
				\vn_+ \vdot \vw_{3\pm} = \qty[\var{\K^{-1}}(\ka)[h]\vnu] \circ (\Phi_\Gt)^{-1} \vdot \vn_+ &on $ \Gt $, \\
				\wt{\vn} \vdot \vw_{3-} = 0 &on $ \pd\Om $.
			\end{cases*}
		\end{equation}
		Then
		\begin{equation}
			\Dbt \vv = \vw_1 + \vw_2 + \vw_3,
		\end{equation}
		with the substitution $ f = \pd_\beta\ka $, $ \vg = \pd_\beta\vom_* $ and $ h = \pd^2_{t\beta}\ka $.
		
		By the definitions, there hold
		\begin{equation}\label{expression B}
			\opb_\pm(\ka)h = (\DD\Phi_\Gt)^{-1}\qty{\vw_{3\pm}|_\Gt\circ\Phi_\Gt - \var{\K^{-1}}(\ka)\qty[h]\vnu },
		\end{equation}
		\begin{equation}\label{expression F}
			\opf_\pm(\ka) \vg = (\DD\Phi_\Gt)^{-1}(\vw_{2\pm}|_\Gt\circ\Phi_\Gt),
		\end{equation}
		and
		\begin{equation}\label{expression G}
			\begin{split}
				\opg_\pm(\ka, \pd_t\ka, \vom_{*\pm})f = &\, (\DD\Phi_\Gt)^{-1} \qty{\vw_{1\pm}|_\Gt \circ (\Phi_\Gt) - \var[2]{\K^{-1}}(\ka)[\pd_t\ka, f]\vnu} \\
				&- (\DD\Phi_\Gt)^{-1}\qty{\DD\qty(\var{\K^{-1}}(\ka)[f]\vnu) \vdot \vv_{\pm*}}.
			\end{split}
		\end{equation}
		Therefore, if $ \ka $ and $ \vom_* $ depend on a parameter $ \beta $, then applying $ \pdv*{\beta} $ to (\ref{expression B}) yields that
		\begin{equation}
			\begin{split}
				&\hspace{-2em}\var{\opb_\pm}(\ka)[\pd_\beta \ka] h \\
				=\, &-(\DD\Phi_\Gt)^{-1} \DD(\pd_\beta\gt \vnu) (\DD\Phi_\Gt)^{-1}\qty{\vw_{3\pm}|_\Gt\circ\Phi_\Gt - \var{\K^{-1}}(\ka)\qty[h]\vnu } \\
				&+(\DD\Phi_\Gt)^{-1}\qty{(\Dbt \vw_{3\pm})|_\Gt\circ\Phi_\Gt - \var[2]\K^{-1}(\ka)[\pd_\beta\ka, h]\vnu },
			\end{split} 
		\end{equation}
		where
		\begin{equation}
			\Dbt \coloneqq \pd_\beta + \DD_{\h(\pd_\beta\gt \vnu)\circ(\X_\Gt)^{-1}}.
		\end{equation}
		Thus, one can derive from the commutator and div-curl estimates that for $ \frac{1}{2} \le s \le \kk - \frac{3}{2} $,
		\begin{equation}
			\abs{\var{\opb_\pm}(\ka)[\pd_\beta \ka] h}_{\H{s}} \less \abs{h}_{\H{s-2}} \cdot \abs{\pd_\beta \ka}_{\H{\kk-\frac{5}{2}}},
		\end{equation}
		namely, (\ref{est var B}) holds.
		(\ref{est var F}) can be shown in a similar way.
		
		The proof of (\ref{est var G}) is a little more complicated. Indeed, denote by $$ Q = Q\qty(\abs{\pd_t\ka}_{\H{\kk-\frac{5}{2}}}, \norm{\vom_*}_{H^{\kk-1}(\OGs)}), $$ and note that
		\begin{equation}
			\begin{split}
				&\hspace{-2em}\pdv{\beta}\qty[\opg_\pm(\ka, \pd_t \ka, \vom_{*\pm})f] \\
				=\, &-(\DD\Phi_\Gt)^{-1} \vdot \DD(\pd_\beta\gt\vnu) \vdot (\DD\Phi_\Gt)^{-1} \vdot \vb{a}_\pm + (\DD\Phi_\Gt)^{-1} \vdot \pd_\beta\vb{a}_\pm,
			\end{split}
		\end{equation}
		where 
		\begin{equation}
			\vb{a}_\pm \coloneqq \vw_{1\pm}|_\Gt \circ (\Phi_\Gt) - \var[2]{\K^{-1}}(\ka)[\pd_t\ka, f]\vnu - \DD\qty(\var{\K^{-1}}(\ka)[f]\vnu) \vdot \vv_{\pm*}.
		\end{equation}
		It follows from (\ref{eqn w1}), (\ref{def u lambda *}), and Proposition \ref{prop K} that
		\begin{equation}
			\begin{split}
				\abs{\vb{a_\pm}}_{\H{\sigma-\frac{1}{2}}} \lesssim_{Q}\, &\norm{\vw_{1\pm}}_{H^{\kk-1}(\OGt)} + \abs{\pd_t\ka}_{\H{\kk-\frac{9}{2}}} \cdot \abs{f}_{\H{\sigma-\frac{5}{2}}} \\
				&+ \norm{\vv_{\pm*}}_{H^{\kk-1}(\OGs)} \cdot \abs{f}_{\H{\sigma-\frac{3}{2}}} \\
				\lesssim_{Q}\, &\abs{f}_{\H{\sigma-\frac{5}{2}}}.
			\end{split}
		\end{equation}
		Therefore,
		\begin{equation}
			\begin{split}
				\abs{(\DD\Phi_\Gt)^{-1} \vdot \DD(\pd_\beta\gt\vnu) \vdot (\DD\Phi_\Gt)^{-1} \vdot \vb{a}_\pm}_{\H{\sigma-\frac{1}{2}}}
				\lesssim_{Q} \abs{\pd_\beta\ka}_{\H{\sigma-\frac{1}{2}}} \abs{f}_{\H{\sigma-\frac{3}{2}}}.
			\end{split}
		\end{equation}
		Next, by observing that
		\begin{equation}
			\begin{split}
				\pd_\beta\vb{a}_\pm =\, &(\Dbt \vw_{1\pm})|_\Gt \circ \Phi_\Gt - \var[3]{\K^{-1}}(\ka)\qty[\pd_\beta\ka, \pd_t\ka, f]\vnu \\
				&-\var[2]{\K^{-1}}(\ka)[\pd^2_{t\beta}\ka, f]\vnu - \DD\qty{\var[2]{\K^{-1}}(\ka)[\pd_\beta\ka, f]\vnu}\vdot\vv_{\pm*} \\
				&- \DD\qty{\var{\K^{-1}}(\ka)[f]\vnu}\vdot\pd_\beta\vv_{\pm*},
			\end{split}
		\end{equation}
		one can deduce from  (\ref{est pd beta v+*}) that
		\begin{equation}
			\begin{split}
				&\hspace{-2em}\abs{\pd_\beta \vb{a}_\pm - (\Dbt \vw_{1\pm})|_\Gt \circ \Phi_\Gt }_{\H{\sigma-\frac{1}{2}}} \\
				\lesssim_{Q}\, &\abs{f}_{\H{\sigma-\frac{3}{2}}} \qty(\abs{\pd_\beta\ka}_{\H{\sigma-\frac{1}{2}}} + \abs{\pd^2_{t\beta}\ka}_{\H{\sigma-\frac{3}{2}}} + \norm{\pd_\beta \vom_{*\pm}}_{H^{\sigma}(\Om_*^\pm)}).
			\end{split}
		\end{equation}
		As for the estimate of $ \Dbt \vw_{1\pm} $, first note that
		\begin{equation}
			\begin{split}
				\abs{\Dbt \vw_{1\pm}}_{H^{\sigma-\frac{1}{2}}(\Gt)} \lesssim_{Q}\, &\norm{\Dbt\vw_{1\pm}}_{H^{\sigma}(\Om_t^\pm)} \\
				\lesssim_{Q}\, &\norm{\div(\Dbt\vw_{1\pm})}_{H^{\sigma-1}(\Om_t^\pm)} + \norm{\curl(\Dbt\vw_{1\pm})}_{H^{\sigma-1}(\Om_t^\pm)} \\
				&+ \abs{\vn_+ \vdot (\Dbt\vw_{1\pm})}_{H^{\sigma-\frac{1}{2}}(\Gt)}.
			\end{split}
		\end{equation}
		The commutator estimates yield that
		\begin{equation}
			\begin{split}
				&\hspace{-2em}\norm{\div(\Dbt\vw_{1\pm})}_{H^{\sigma-1}(\Om_t^\pm)} + \norm{\curl(\Dbt\vw_{1\pm})}_{H^{\sigma-1}(\Om_t^\pm)} \\
				\lesssim_{Q}\, &\abs{f}_{\H{\sigma-\frac{3}{2}}} \cdot \qty(\abs{\pd_\beta\ka}_{\H{\sigma-\frac{1}{2}}} + \norm{\pd_\beta\vom_{*\pm}}_{H^{\sigma}(\Om_*^\pm)}).
			\end{split}
		\end{equation}
		For the boundary estimate of $ \Dbt \vw_{1\pm} $, due to the relations that
		\begin{equation}
			\begin{split}
				&\hspace{-2em}\vn_+ \vdot (\Dbt\vw_{1\pm}) - \vn_+ \vdot \Dbt\vb{b}_\pm \\
				=\, &\vn_+ \vdot \DD\qty[(\pd_\beta\gt \vnu)\circ (\Phi_\Gt)^{-1}] \vdot (\vw_{1\pm}^\top - \vb{b}_\pm^\top),
			\end{split}
		\end{equation}
		and
		\begin{equation}
			\begin{split}
				&\hspace{-2em}\abs{\vn_+ \vdot \DD\qty[(\pd_\beta\gt \vnu)\circ (\Phi_\Gt)^{-1}] \vdot (\vw_{1\pm}^\top - \vb{b}_\pm^\top)}_{H^{\sigma-\frac{1}{2}}(\Gt)} \\
				\lesssim_{Q}\, &\abs{f}_{\H{\sigma-\frac{3}{2}}} \cdot \abs{\pd_\beta\ka}_{\H{\sigma-\frac{1}{2}}},
			\end{split}
		\end{equation}
		it is direct to compute that
		\begin{equation}
			\begin{split}
				&\hspace{-1em}(\Dbt\vb{b}_\pm)\circ\Phi_\Gt = \pd_\beta(\vb{b}_\pm \circ \Phi_\Gt) \\
				=\, &\DD\qty{\var[2]{\K^{-1}}(\ka)[\pd_\beta\ka, f]\vnu} \vdot (\DD\Phi_\Gt)^{-1} \vdot \qty{\vv_{\pm}|_\Gt\circ\Phi_\Gt - \var{\K^{-1}}(\ka)[\pd_t\ka]\vnu} \\
				&-\DD\qty{\var{\K^{-1}}(\ka)[f]\vnu} \vdot (\DD\Phi_\Gt)^{-1}  \vdot \DD\qty(\pd_\beta\gt \vnu) \vdot (\DD\Phi_\Gt)^{-1} \vdot \\
				&\qquad \vdot  \qty{\vv_{\pm}|_\Gt\circ\Phi_\Gt - \var{\K^{-1}}(\ka)[\pd_t\ka]\vnu} \\
				&+ \DD\qty{\var{\K^{-1}}(\ka)[f]\vnu} \vdot (\DD\Phi_\Gt)^{-1} \vdot \\
				&\qquad\vdot\qty{(\Dbt\vv_\pm) \circ\Phi_\Gt - \var[2]{\K^{-1}}(\ka)[\pd_\beta\ka, \pd_t\ka]\vnu - \var{\K^{-1}(\ka)[\pd^2_{t\beta}\ka]\vnu}} \\
				&-\var[2]{\K^{-1}}(\ka)[\pd^2_{t\beta}\ka, f]\vnu - \var[3]{\K^{-1}}(\ka)[\pd_\beta\ka, \pd_t\ka, f]\vnu.
			\end{split}
		\end{equation}
		Therefore, it follows from \eqref{est pd beta v+*} and Proposition \ref{prop K} that
		\begin{equation}
			\begin{split}
				\abs{\pd_\beta(\vb{b}_\pm \circ \Phi_\Gt)}_{\H{\sigma-\frac{1}{2}}} 
				\lesssim_{Q}\abs{f}_{\H{\sigma-\frac{3}{2}}} \cdot \qty(\abs{\pd_\beta\ka}_{\H{\sigma-\frac{1}{2}}}+ \abs{\pd^2_{t\beta}\ka}_{\H{\sigma-\frac{3}{2}}}).
			\end{split}
		\end{equation}
		In conclusion,
		\begin{equation}
			\begin{split}
				&\hspace{-2em}\abs{\pdv{\beta}\qty[\opg_\pm(\ka, \pd_t \ka, \vom_{*\pm})f]}_{\H{\sigma-\frac{1}{2}}} \\
				\lesssim_{Q}\, &\abs{f}_{\H{\sigma-\frac{3}{2}}} \cdot \qty(\abs{\pd_\beta\ka}_{\H{\sigma-\frac{1}{2}}} + \abs{\pd^2_{t\beta}\ka}_{\H{\sigma-\frac{3}{2}}} + \norm{\pd_\beta\vom_*}_{H^{\sigma}(\OGs)} ), 
			\end{split}
		\end{equation}
		which is exactly (\ref{est var G}).
	\end{proof}
	
	\section{Proof of Lemma \ref{lem 3.4}}
	\begin{proof}
		Note that
		\begin{equation}\label{eqn opF}
			\opF(\ka)\vb{g} = \qty[\mathrm{I}+\opB(\ka)]^{-1} \qty[\grad^\top(\kappa_+ \circ \Phi_\Gt) \vdot \opf(\ka)\vb{g}].
		\end{equation}
		It follows from \eqref{est F} that
		\begin{equation}\label{est op scrF}
			\begin{split}
				\abs{\opF(\ka)\vb{g}}_{\H{s}} &\less \abs{\grad^\top(\kappa_+ \circ \Phi_\Gt) \vdot \opf(\ka)\vb{g}}_{\H{s}} \\
				&\less \abs{\grad^\top (\kappa_+ \circ \Phi_\Gt)}_{\H{\kk-2}} \abs{\opf(\ka)\vb{g}}_{\H{s+\epsilon}} \\
				&\less \abs{\ka}_{\H{\kk-1}} \norm{\vb{g}}_{H^{s+\epsilon-\frac{1}{2}}(\OGs)},
			\end{split}
		\end{equation}
		for $ \frac{1}{2} \le s \le \kk-2 $, which is exactly \eqref{est opF}.
		For $ k \ge 3 $ and $ \frac{1}{2} \le \sigma \le \kk-\frac{5}{2} $, similar arguments lead to \eqref{est opF'}.
		
		To estimate $ \opG $, one observes that
		\begin{equation}\label{formula (I+B)^(-1)G}
			\begin{split}
				&\hspace{-2em}\qty[\mathrm{I} + \opB(\ka)]\opG(\ka, \pd_t\ka, \vom_*, \vj_*) \\
				=\,&-a^2 \dfrac{\vn_+ \circ \Phi_\Gt}{\vnu \vdot (\vn_+ \circ \Phi_\Gt)} \vdot \qty[\va{\mathfrak{b}} \circ \Phi_\Gt + \DD_{\vbu_*}\qty(\vbu\circ\Phi_\Gt + \pd_t\gt\vnu)] \\
				&+a^2\opC_\alpha \gt + \opB\opC_\alpha\ka - \grad^\top(\kappa_+\circ\Phi_\Gt) \vdot \qty[\opg\pd_t\ka] + \mathfrak{R}_1 \circ\Phi_\Gt,
			\end{split}
		\end{equation}
		with $ \mathfrak{R_1} $ given by (\ref{def frak[R]_1}). Thus, \eqref{est opG}  and \eqref{est opG'} follow from \eqref{def vW}, \eqref{est frak R_0}, \eqref{est frak R_0'}, \eqref{est opB}, \eqref{est opB'}, Lemma \ref{lem 3.3} and Theorem \ref{thm div-curl}.
		
		Next, we consider the variational estimates under the assumption that $ k \ge 3 $. Suppose that $ \xi(\ka) $ and $ \eta(\ka) $ are two functionals so that 
		\begin{equation}\label{relation xi eta}
			\xi = \qty(\mathrm{I} + \opB)\eta.
		\end{equation} 
		Then, when computing the variation formula, the following relation holds:
		\begin{equation}\label{eqn xi eta I+B}
			\pdv{\xi}{\beta} = \qty(\mathrm{I}+\opB)\pdv{\eta}{\beta} + \pdv{\beta}(\opB)\eta.
		\end{equation}
		Therefore, if $ \ka $ is parameterized by $ \beta $, then
		\begin{equation}\label{eqn var xi eta I+B}
			\var{\eta}(\ka)\qty[\pd_\beta\ka] = \qty[\mathrm{I}+\opB(\ka)]^{-1} \qty{\var{\xi}(\ka)\qty[\pd_\beta\ka] - \var{\opB}(\ka)\qty[\pd_\beta\ka]\eta(\ka)},
		\end{equation}
		where $ \var\opB $ has the following estimate (by Lemma \ref{lem3.1}):
		\begin{equation}\label{est var opB}
			\abs{\var\opB(\ka)}_{\LL\qty[\H{\kk-\frac{5}{2}}; \LL\qty(\H{s-2}; \H{s})]} \le C_*,
		\end{equation}
		for $ \frac{1}{2} \le s \le \kk-\frac{7}{2} $ and a constant $ C_* > 0$ depending on $ \Lambda_* $ and $ s $.
		
		Similarly, it follows form \eqref{eqn opF}-\eqref{eqn xi eta I+B} and (\ref{est var F}) that
		\begin{equation}
			\begin{split}
				&\hspace{-2em}\abs{\var\qty{(\mathrm{I} + \opB(\ka))\opF(\ka)}[\pd_\beta\ka]\vb{g}}_{\H{\kk-4}} \\
				\less\, &\abs{\grad^\top\pd_\beta(\kappa_+ \circ \Phi_\Gt) \vdot \opf(\ka)\vb{g}}_{\H{\kk-4}} + \abs{\grad^\top(\kappa_+ \circ \Phi_\Gt) \vdot \var\opf(\ka)[\pd_\beta\ka]\vb{g}}_{\H{\kk-4}} \\
				\less\, &\abs{\pd_\beta\ka}_{\H{\kk-\frac{5}{2}}} \norm{\vb{g}}_{H^{\kk-4}(\OGs)},
			\end{split}
		\end{equation}
		which, together with (\ref{eqn var xi eta I+B}), (\ref{est var opB}), (\ref{eqn opF}) and Lemma \ref{lem3.1}, yields (\ref{est var opF}).
		
		To derive the variational estimate of $ \opG $, we suppose that $ \ka, \vom_*$, and $ \vj_* $ depend on a parameter $ \beta $. With the same notations as in (\ref{def Dbt}), the variational estimate of $ \mathfrak{R}_1 \circ \Phi_\Gt $ can be given by:
		\begin{equation}\label{est var frak R_1}
			\abs{\pdv{\beta}(\mathfrak{R}_1 \circ\Phi_\Gt)}_{\H{\kk-4}} \less	\abs{\Dbt \mathfrak{R}_1}_{H^{\kk-4}(\Gt)}.
		\end{equation}
		For simplicity, from now on, we shall use notations $ \abs{\cdot}_s \equiv \abs{\cdot}_{H^s(\Gt)} $, $ \norm{\phi_\pm}_s \equiv \norm{\phi_\pm}_{H^s(\Om_t^\pm)} $, and
		\begin{equation*}
			\begin{split}
				\va{\mathfrak{b}} \coloneqq \dfrac{1}{\rho_+ + \rho_-} \qty(\grad q^+ + \grad q^-) + \dfrac{\rho_+ \rho_-}{\qty(\rho_+ + \rho_-)^2}\DD_{\vw}\vw  - \dfrac{\rho_+}{\rho_+ + \rho_-}\DD_{\vh_+}\vh_+ - \dfrac{\rho_-}{\rho_+ + \rho_-}\DD_{\vh_-}\vh_-.
			\end{split}
		\end{equation*}
		First note that
		\begin{equation}
			\abs{\Dbt \qty(\va{\mathfrak{b}} \vdot \lap_\Gt\vn_+)}_{\kk-4} \less \abs*{\Dbt\va{\mathfrak{b}}}_{\kk-\frac{7}{2}}\abs{\lap_\Gt\vn_+}_{\kk-\frac{7}{2}} + \abs{\Dbt\lap_\Gt \vn_+}_{\kk-\frac{7}{2}}\abs*{\va{\mathfrak{b}}}_{\kk-\frac{7}{2}},
		\end{equation}
		and one can derive from Lemma \ref{Dt comm est lemma} and (\ref{eqn dt n}) that 
		\begin{equation}
			\abs{\Dbt\lap_\Gt \vn_+}_{\kk-4} \less \abs{\pd_\beta\ka}_{\H{\kk-\frac{5}{2}}}.
		\end{equation}
		As for the term $ \Dbt\va{\mathfrak{b}} $, Lemma \ref{Dt comm est lemma}, \eqref{est pd beta v+*} and (\ref{def vW}) lead to
		\begin{equation}\label{est dbt va b}
			\begin{split}
				\abs{\Dbt\va{\mathfrak{b}}}_{\kk-4} \less\, &\abs{\Dbt \grad q}_{\kk-4} + \abs{\Dbt \DD_{\vh} \vh}_{\kk-4} + \abs{\Dbt \DD_\vw \vw}_{\kk-4} \\
				\lesssim_{Q_*}\, &\abs{\pd_\beta \ka}_{\H{\kk-3}} + \abs{\pd^2_{t\beta}\ka}_{\H{\kk-5}}  + \norm{\pd_\beta\vom_*}_{H^{\kk-\frac{7}{2}}(\OGs)} \\
				& + \norm{\pd_\beta\vj_*}_{H^{\kk-\frac{7}{2}}(\OGs)},
			\end{split}
		\end{equation}
		where $ Q_* $ is a  generic polynomial depending on $ \Lambda_* $ of the quantities $ \abs{\pd_t\ka}_{\H{\kk-\frac{5}{2}}}$,\\ $ \norm{\vom_*}_{H^{\kk-1}(\OGs)}$, and $\norm{\vj_*}_{H^{\kk-1}(\OGs)} $.
		Next, observe that
		\begin{equation}
			\begin{split}
				\abs{\Dbt\qty[\vn \cdot \DD\vbu \cdot (\DD_\Gt)^2\vbu]}_{\kk-4} \lesssim_{Q_*}\,  &\abs{\pd_\beta \ka}_{\H{\kk-\frac{5}{2}}} + \abs{\Dbt\vbu}_{\kk-2} \\
				\lesssim_{Q_*}\, &\abs{\pd_\beta\ka}_{\H{\kk-\frac{5}{2}}} + \abs{\pd^2_{t\beta}\ka}_{\H{\kk-4}} +\norm{\pd_\beta\vom_*}_{H^{\kk-\frac{5}{2}}(\OGs)},
			\end{split}
		\end{equation}
		and
		\begin{equation}
			\abs{\Dbt \II_+}_{\kk-4} \lesssim_{Q_*} \abs{\pd_\beta \ka}_{\H{\kk-4}}.
		\end{equation}
		Thus, the variational estimates of the first six terms of (\ref{def frak[R]_1}) follow easily.
		To deal with the terms involving $ \lap_\Gt \II - \opr $, one can deduce from \eqref{simons' identity}, (\ref{est pd beta v+*}) and Lemma \ref{Dt comm est lemma} that
		\begin{equation}
			\begin{split}
				&\hspace{-3em}\abs{\Dbt\qty{\lap_\Gt\qty[\II_+(\vw, \vw)] - \opr(\Gt, \vw)\kappa_+}}_{\kk-4} \\
				\lesssim_{Q_*}\, &\abs{\Dbt \vw}_{\kk-2} + \abs{\Dbt \kappa_+}_{\kk-3} + \abs{\pd_\beta \ka}_{\H{\kk-\frac{5}{2}}} \\
				\lesssim_{Q_*}\, &\abs{\pd_\beta\ka}_{\H{\kk-\frac{5}{2}}} + \abs{\pd^2_{t\beta}\ka}_{\H{\kk-4}} +\norm{\pd_\beta\vom_*}_{H^{\kk-\frac{5}{2}}(\OGs)}.
			\end{split}
		\end{equation}
		Similar arguments yield that
		\begin{equation}
			\begin{split}
				&\hspace{-3em}\abs{\Dbt\qty{\lap_\Gt\qty[\II_+(\vh_\pm, \vh_\pm)] - \opr(\Gt, \vh_\pm)\kappa_+}}_{\kk-4} \\
				\lesssim_{Q_*}\, &\abs{\pd_\beta\ka}_{\H{\kk-\frac{5}{2}}}  +\norm{\pd_\beta\vj_*}_{H^{\kk-\frac{5}{2}}(\OGs)}.
			\end{split}
		\end{equation}
		For the last term of (\ref{def frak[R]_1}), one can deduce from (\ref{def frak[r_0]}) and Lemma \ref{Dt comm est lemma} that
		\begin{equation}
			\begin{split}
				&\hspace{-2em}\abs{\Dbt \lap_\Gt \mathfrak{r}_0}_{\kk-4} \\
				\lesssim_{Q_*}\, &\abs{\pd_\beta\ka}_{\H{\kk-\frac{5}{2}}} + \abs{\Dbt \mathfrak{r}_0}_{\kk-2} \\
				\lesssim_{Q_*}\, &\abs{\pd_\beta\ka}_{\H{\kk-\frac{5}{2}}} + \abs{\Dbt (g^+ - g^-)}_{\kk-3} + \norm{\Dbt\qty( p_{\vv, \vv} - p_{\vh, \vh})}_{\kk-\frac{1}{2}}\\
				\lesssim_{Q_*}\, &\abs{\pd_\beta \ka}_{\H{\kk-\frac{5}{2}}} + \norm{\Dbt \vv}_{\kk-\frac{3}{2}} + \norm{\Dbt \vh}_{\kk-\frac{3}{2}} \\
				\lesssim_{Q_*}\, &\abs{\pd_\beta \ka}_{\H{\kk-\frac{5}{2}}} + \abs{\pd^2_{t\beta}\ka}_{\H{\kk-4}} + \norm{\pd_\beta\vom_*}_{H^{\kk-\frac{5}{2}}(\OGs)}  + \norm{\pd_\beta\vj_*}_{H^{\kk-\frac{5}{2}}(\OGs)}.
			\end{split}
		\end{equation}
		Thus, the variational estimate of the last term of (\ref{formula (I+B)^(-1)G}) has been obtained. 
		
		Next, for the variation of the first term on the right hand side of (\ref{formula (I+B)^(-1)G}), observe that 
		\begin{equation}\label{pd tt gt  exp}
			\begin{split}
				&\hspace{-2em}\pdv{\beta}(\dfrac{\vn_+ \circ \Phi_\Gt}{\vnu \vdot (\vn_+ \circ \Phi_\Gt)} \vdot \qty[\va{\mathfrak{b}} \circ \Phi_\Gt + \DD_{\vbu_*}\qty(\vbu\circ\Phi_\Gt + \pd_t\gt\vnu)]) \\
				=\, &\pdv{\beta}(\dfrac{\vn_+ \circ \Phi_\Gt}{\vnu \vdot (\vn_+ \circ \Phi_\Gt)}) \vdot \qty\Big{\va{\mathfrak{b}}\circ\Phi_\Gt - \DD_{\vbu_*}\qty[\qty(\vbu\circ\Phi_\Gt) + \qty(\pd_t\gt) \vnu]} \\
				& +\dfrac{\vn_+ \circ \Phi_\Gt}{\vnu \vdot (\vn_+ \circ \Phi_\Gt)} \vdot \pdv{\beta}\qty\Big(\va{\mathfrak{b}}\circ\Phi_\Gt - \DD_{\vbu_*}\qty[\qty(\vbu\circ\Phi_\Gt) + \qty(\pd_t\gt) \vnu]).
			\end{split}
		\end{equation}
		Due to the relations that
		\begin{equation}
			\begin{split}
				\abs{\pd_\beta (\vn_+ \circ \Phi_\Gt)}_{\kk-\frac{3}{2}} \less \abs{\DD(\pd_\beta\gt\vnu)}_{\H{\kk-\frac{3}{2}}} \less \abs{\pd_\beta\ka}_{\H{\kk-\frac{5}{2}}},
			\end{split}
		\end{equation}
		and
		\begin{equation}
			\begin{split}
				&\hspace{-1em}\abs{\pdv{\beta}\qty\Big(\DD_{\vbu_*}\qty[(\vbu\circ\Phi_\Gt)+(\pd_t\gt)\vnu])}_{\H{\kk-4}} \\
				&=\abs{\pdv{\beta}\qty\Big(\DD_{\vbu_*}\qty[\DD\Phi_\Gt \vdot \vbu_* + 2(\pd_t\gt)\vnu])}_{\H{\kk-4}} \\
				&\lesssim_{Q_*} \abs{\pd_\beta \vbu_*}_{\H{\kk-3}} + \abs{\pd^2_{t\beta}\gt}_{\H{\kk-3}} + \abs{\pd_\beta \ka}_{\H{\kk-4}} \\
				&\lesssim_{Q_*} \abs{\pd_\beta \ka}_{\H{\kk-4}} + \abs{\pd^2_{t\beta}\ka}_{\H{\kk-5}} + \norm{\pd_\beta\vom_*}_{H^{\kk-\frac{7}{2}}(\OGs)} +\norm{\pd_\beta\vj_*}_{H^{\kk-\frac{7}{2}}(\OGs)},
			\end{split}
		\end{equation}
		the estimate of (\ref{pd tt gt  exp}) can be deduced via (\ref{est dbt va b}).
		
		Since $ \vh $ can be recovered from $ (\ka, \vj_*) $ by solving the div-curl problems, the operator $ \mathscr{C}_\alpha $ can also be expressed as $ \opC_\alpha(\ka, \pd_t\ka, \vom_*, \vj_*) $. It follows form (\ref{def opC}), \eqref{est pd beta v+*}, (\ref{eqn pd beta v}), Lemma \ref{lem3.1} and Lemma \ref{lem 3.3} that for $ 2 \le s' \le \kk-1 $, $ \mathscr{V} \coloneqq \H{\kk-\frac{5}{2}} \times \H{\kk-4} \times H^{\kk-\frac{5}{2}}(\OGs) \times H^{\kk-\frac{5}{2}}(\OGs) $, it holds that
		\begin{equation}
			\begin{split}
				&\abs{\var{\opC_\alpha}(\ka, \pd_t\ka, \vom_*, \vj_*)}_{\LL\qty[ \mathscr{V}; \LL\qty(\H{s'}; \H{s'-3})]} \\
				&\quad \le Q_*\qty(\abs{\pd_t\ka}_{\H{\kk-\frac{5}{2}}}, \norm{\vom_*}_{H^{\kk-1}(\OGs)}, \norm{\vj_*}_{H^{\kk-1}(\OGs)}).
			\end{split}
		\end{equation}
		In particular, by letting $ s' \coloneqq \kk-1 $, one can deduce that
		\begin{equation}
			\begin{split}
				&\abs{\pdv{\beta}(\opC_\alpha \gt + \opB\opC_\alpha\ka) }_{\H{\kk-4}} \\
				&\quad\lesssim_{Q_*}\abs{\pd_\beta \ka}_{\H{\kk-\frac{5}{2}}} + \abs{\pd^2_{t\beta}\ka}_{\H{\kk-4}} + \norm{\pd_\beta\vom_*}_{H^{\kk-\frac{5}{2}}(\OGs)} \\
				&\qquad\quad + \norm{\pd_\beta\vj_*}_{H^{\kk-\frac{5}{2}}(\OGs)}.
			\end{split}
		\end{equation}
		Furthermore, one may derive from Lemma \ref{lem3.1} that
		\begin{equation}\label{est grad ka G pdt ka}
			\begin{split}
				&\abs{\pdv{\beta}\qty\Big{\grad^\top(\kappa_+\circ\Phi_\Gt) \vdot \qty[\opg\pd_t\ka]}}_{\H{\kk-4}} \\
				&\quad\lesssim_{Q_*}\abs{\pd_\beta \ka}_{\H{\kk-3}} + \abs{\pd^2_{t\beta}\ka}_{\H{\kk-5}} + \norm{\pd_\beta\vom_*}_{H^{\kk-\frac{7}{2}}(\OGs)}.
			\end{split}
		\end{equation}
		
		In conclusion, (\ref{est var opG}) follows from (\ref{relation xi eta})-(\ref{est var opB}), (\ref{formula (I+B)^(-1)G}) and \eqref{est var frak R_1}-(\ref{est grad ka G pdt ka}).
	\end{proof}
	
	\bibliographystyle{alpha}
	\bibliography{ref.bib}
	
\end{document}